\documentclass[a4paper,12pt,twoside,reqno]{amsart}

\usepackage[T1]{fontenc}
\usepackage[utf8]{inputenc}
\usepackage[english]{babel}

\usepackage[%
	left=2.5cm,       
	right=2.5cm,      
	top=3.5cm,        
	bottom=2.5cm,     
	heightrounded,    
	bindingoffset=0mm 
]{geometry}

\usepackage{amsmath}
\usepackage{amssymb}
\usepackage{mathtools}
\usepackage{mathrsfs}
\usepackage{hyperref} 
\usepackage[noabbrev,capitalize]{cleveref}

\usepackage{lmodern}

\numberwithin{equation}{section}

\theoremstyle{plain}
	\newtheorem{theorem}{Theorem}[section]
	\newtheorem{lemma}[theorem]{Lemma}
	\newtheorem{proposition}[theorem]{Proposition}
	\newtheorem{corollary}[theorem]{Corollary}
	
\theoremstyle{definition}
	\newtheorem{definition}[theorem]{Definition}
	
	\newtheorem{remark}[theorem]{Remark}	

\newcommand{\N}{\mathbb{N}}
\newcommand{\R}{\mathbb{R}}
\newcommand{\G}{\mathbb{G}}

\newcommand{\M}{\mathbb{M}}

\newcommand{\E}{\mathcal{E}}
\newcommand{\de}{\partial}
\newcommand{\eps}{\varepsilon}
\newcommand{\longto}{\longrightarrow}
\newcommand{\weakto}{\rightharpoonup}
\newcommand{\Prob}{\mathscr{P}}
\newcommand{\F}{\mathsf{F}}
\newcommand{\Ch}{\mathsf{Ch}}
\newcommand{\Ent}{\mathsf{Ent}}
\newcommand{\D}{\mathrm{D}}
\newcommand{\m}{\mathfrak{m}}
\newcommand{\BE}{{\sf BE}}
\newcommand{\wBE}{{\sf BE}_w}
\newcommand{\CD}{{\sf CD}}
\newcommand{\RCD}{{\sf RCD}}

\newcommand{\EVI}{{\sf EVI}}

\newcommand{\MCP}{{\sf MCP}}
\newcommand{\dcc}{\mathsf{d}_{\mathrm{cc}}}
\newcommand{\leb}{\mathscr{L}}
\newcommand{\g}{\mathrm{g}}
\newcommand{\SU}{\mathbb{SU}}
\newcommand{\p}{\mathsf{p}}
\newcommand{\neu}{\mathsf{o}}
\newcommand{\HB}{\mathcal{H}}
\newcommand{\VB}{\mathcal{V}}

\def\Xint#1{\mathchoice
{\XXint\displaystyle\textstyle{#1}}%
{\XXint\textstyle\scriptstyle{#1}}%
{\XXint\scriptstyle\scriptscriptstyle{#1}}%
{\XXint\scriptscriptstyle\scriptscriptstyle{#1}}%
\!\int}
\def\XXint#1#2#3{{\setbox0=\hbox{$#1{#2#3}{\int}$ }
\vcenter{\hbox{$#2#3$ }}\kern-.6\wd0}}

\def\aint{\Xint-}

\let\svc\c
\DeclareRobustCommand\c{\ifmmode\mathsf{c}\else\expandafter\svc\fi}

\renewcommand{\phi}{\varphi}
\renewcommand{\rho}{\varrho}
\renewcommand{\theta}{\vartheta}
\renewcommand{\d}{\mathsf{d}}
\renewcommand{\P}{\mathsf{P}}
\renewcommand{\H}{\mathsf{H}}
\renewcommand{\L}{\mathscr{L}}
\renewcommand{\S}{\mathsf{S}}

\DeclareMathOperator{\dist}{dist}
\DeclareMathOperator{\diam}{diam}

\DeclareMathOperator{\supp}{supp}

\DeclareMathOperator{\Dom}{Dom}

\DeclareMathOperator{\vol}{Vol}
\DeclareMathOperator{\Ric}{Ric}

\DeclareMathOperator{\di}{d\!}
\DeclareMathOperator{\I}{I}
\DeclareMathOperator{\RI}{R}

\DeclareMathOperator{\Lip}{Lip}
\DeclareMathOperator{\AC}{AC}
\DeclareMathOperator{\Cont}{C}
\DeclareMathOperator{\Leb}{\mathrm{L}}
\DeclareMathOperator{\Sob}{W}
\DeclareMathOperator{\loc}{loc}
					    
\DeclarePairedDelimiter{\scalar}{<}{>}                                     
\DeclarePairedDelimiter{\set}{\{}{\}}
\DeclarePairedDelimiter{\abs}{|}{|}

\newcommand{\mres}{\mathbin{\vrule height 1.6ex depth 0pt width
0.13ex\vrule height 0.13ex depth 0pt width 1.3ex}}

\usepackage{bookmark}

\usepackage[initials]{amsrefs}

\usepackage{tikz}
\usepackage{graphicx}
\usepackage[font=small]{caption}

\newcommand{\closure}[2][3]{%
  {}\mkern#1mu\overline{\mkern-#1mu#2}}
  
\usepackage{enumitem}

\usepackage{url}

\allowdisplaybreaks

\begin{document}

\title[Generalized Bakry--\'{E}mery curvature condition]{Generalized Bakry--\'{E}mery curvature condition and equivalent entropic inequalities in groups}

\author[G. Stefani]{Giorgio Stefani}
\address{Department Mathematik und Informatik, Universit\"at Basel, Spiegelgasse 1, CH-4051 Basel, Switzerland}
\email{giorgio.stefani@unibas.ch}

\date{\today}

\keywords{Bakry--\'{E}mery curvature condition, metric-measure space, Carnot group, $\SU(2)$ group, topological groups, Gamma Calculus, entropy, entropic inequalities, Wasserstein distance}

\subjclass[2010]{Primary 47D07, 28D20. Secondary 53C17}

\thanks{\textit{Acknowledgements}. 
The author thanks Luigi Ambrosio and Giuseppe Savaré for helpful discussions and many valuable suggestions about the subject.
The author also thanks the anonymous referee for precious comments and for pointing the reference~\cite{W20}. 
The author is partially supported by the ERC Starting Grant 676675 FLIRT -- \textit{Fluid Flows and Irregular Transport}.
The author is a member of INdAM and is partially supported by the INdAM--GNAMPA Project 2020 \textit{Problemi isoperimetrici con anisotropie} (n.\ prot.\ U-UFMBAZ-2020-000798 15-04-2020)%
}

\begin{abstract}
We study a generalization of the Bakry--\'{E}mery pointwise gradient estimate for the heat semigroup and its equivalence with some entropic inequalities along the heat flow and Wasserstein geodesics for metric-measure spaces with a suitable group structure. 
Our main result applies to Carnot groups of any step and to the $\SU(2)$ group.  
\end{abstract}

\maketitle

\section{Introduction}

\subsection{The Riemannian framework}

Let $(\M,\g)$ be a (complete and connected) $N$-di\-men\-sio\-nal smooth Riemannian manifold with Laplace--Beltrami operator~$\Delta$.
The celebrated \emph{Bochner formula} states that 
\begin{equation}\label{intro_eq:bochner}
\frac12\,
\Delta|\nabla f|^2_\g
=
\scalar*{\nabla\Delta f,\nabla f}_\g
+
||\mathrm{Hess} f||_2^2
+
\Ric(\nabla f,\nabla f)
\end{equation}
for all $f\in\Cont^\infty(\M)$.
Defining
\begin{equation}
\label{intro_eq:Gamma_def}
\begin{split}
\Gamma(f,g)
&=
\frac12\,\big(\Delta(fg)-f\,\Delta g-g\,\Delta f\big)
=
\scalar*{\nabla f,\nabla g}_\g,\\
\Gamma_2(f,g)
&=
\frac12\big(\Delta\Gamma(f,g)-\Gamma(f,\Delta g)-\Gamma(\Delta f,g)\big)
\end{split}
\end{equation} 
for all $f,g\in\Cont^\infty(\M)$, we can rewrite~\eqref{intro_eq:bochner} as 
\begin{equation*}
\frac12\,
\Delta\Gamma(f)
=
\Gamma(\Delta f,f)
+
||\mathrm{Hess} f||_2^2
+
\Ric(\nabla f,\nabla f),
\end{equation*} 
so that
\begin{equation*}
\Gamma_2(f)
=
||\mathrm{Hess} f||_2^2
+
\Ric(\nabla f,\nabla f)
\end{equation*}
for all $f\in\Cont^\infty(\M)$. 
Here and in the following, we write $\Gamma(f)=\Gamma(f,f)$ and $\Gamma_2(f)=\Gamma_2(f,f)$ for simplicity.
Using Cauchy--Schwartz inequality, we can estimate 
\begin{equation*}
||\mathrm{Hess} f||_2^2
\ge	
\frac1N\,
(\Delta f)^2,
\end{equation*}
thus the \emph{geometric} information
\begin{equation}\label{intro_eq:Ric_ge_K}
\Ric\ge K\
\text{for some}\ K\in\R	
\end{equation}
implies the \emph{analytical} information
\begin{equation}\label{intro_eq:CD_ineq}
\Gamma_2(f)
\ge
\frac1N\,
(\Delta f)^2
+
K\,\Gamma(f)
\end{equation}
for all $f\in\Cont^\infty(\M)$.
Nowadays, \eqref{intro_eq:CD_ineq} is the so-called and well-known \emph{Bakry--\'Emery curvature-dimension inequality} $\CD(K,N)$.
Remarkably, it is also possible to prove the converse implication, see~\cite{B94}*{Proposition~6.2}: if a Riemannian manifold $\M$ satisfies $\CD(K,N)$ for some $K\in\R$ and $N\in(0,+\infty)$, then $\dim\M\le N$ and $\Ric\ge K$.

Let us now 
drop the role of the dimension of $\M$ (which formally corresponds to the choice $N=+\infty$ in~\eqref{intro_eq:CD_ineq}) and focus on the lower bound on the Ricci tensor encoded by the $\CD(K,\infty)$ condition.
After the works of Bakry--\'Emery~\cite{BE85}, Otto--Villani~\cite{OV00}, Cordero-Erausquin--McCann--Schmucken\-schl\"ager~\cite{CMS01} and von Renesse--Sturm~\cite{vRS05}, the analytical condition~\eqref{intro_eq:CD_ineq} for $N=+\infty$ on a Riemannian manifold can be equivalently formulated in other three ways (at least, see~\cite{vRS05} for other equivalent statements): via the \emph{pointwise gradient estimate} for the heat flow, via the Wasserstein \emph{contractivity property} of the dual heat flow and via the \emph{$K$-convexity} of the entropy along geodesics in the Wasserstein space.

The \emph{heat kernel} $\p_t(x,y)$ of the Riemannian manifold $(\M,\g)$ is the fundamental solution of the heat differential operator $\de_t-\Delta$. 
The function $\p_t(x,y)$ is smooth in $(t,x,y)\in(0,+\infty)\times\M\times\M$, symmetric in $(x,y)$ and naturally defines the associated \emph{heat semigroup} $\P_t\colon\Cont^\infty_c(\M)\to\Cont^\infty(\M)$ as
\begin{equation*}
\P_t f(x)
=
\int_\M f(y)\,\p_t(x,y)\di\vol_\g(y),
\quad
x\in\M,
\end{equation*}
for all $f\in\Cont^\infty_c(\M)$.
Inequality~\eqref{intro_eq:CD_ineq} describes the behavior of the commutation between the gradient $\nabla$ and the heat semigroup $(\P_t)_{t>0}$.
More precisely, the $\CD(K,\infty)$ condition is equivalent to the \emph{Bakry--\'Emery pointwise gradient estimate}
\begin{equation}
\label{intro_eq:P_t_contractivity}
\Gamma(\P_t f)
\le 
e^{-2Kt}\,
\P_t\Gamma(f)
\end{equation}
for all $t>0$ and $f\in\Cont^\infty_c(\M)$.

The \emph{dual heat semigroup} $\H_t\colon\Prob(\M)\to\Prob(\M)$ is nothing but the extension of the heat semigroup to the space $\Prob(\M)$ of probability measures on~$\M$ and can be defined by setting
\begin{equation*}
\int_\M f\di\H_t\mu
=
\int_\M \P_t f\di\mu
\end{equation*}
for all $f\in\Cont_c^\infty(\M)$, whenever $\mu\in\Prob(\M)$.
The subset of $\Prob(\M)$ of probability measures with finite second moment
\begin{equation*}
\Prob_2(\M)
=
\set*{
\mu\in\Prob(\M)
:
\int_\M\d_\g^2(x,x_0)\di\mu(x)<+\infty\
\text{for some}\ 
x_0\in\M
}
\end{equation*}
endowed with the \emph{$2$-Wasserstein distance} $W_2$, given by
\begin{equation*}
\frac12\,
W_2^2(\mu,\nu)
=
\sup\set*{\int_\M\phi\di\mu+\int_\M\psi\di\nu : 
\phi(x)+\psi(y)
\le
\frac12\,\d_\g^2(x,y)\ \text{for all}\ x,y\in\M}
\end{equation*}
for all $\mu,\nu\in\Prob_2(X)$, is a complete and separable geodesic space.
The lower bound~\eqref{intro_eq:Ric_ge_K} on the Ricci tensor can be equivalently stated as a \emph{contractivity property} of the dual heat semigroup with respect to the $2$-Wasserstein distance, in the sense that
\begin{equation}\label{intro_eq:H_t_contractivity}
W_2(\H_t\mu,\H_t\nu)
\le
e^{-Kt}\,
W_2(\mu,\nu)	
\end{equation}
for all $t>0$ and $\mu,\nu\in\Prob_2(X)$. 

The (Boltzmann) \emph{entropy} with respect to the volume measure $\vol_\g$ is defined as 
\begin{equation*}
\Ent(\mu)
=
\begin{cases}
\displaystyle\int_\M f\log f\di\vol_\g & \text{if}\ \mu=f\,\vol_\g,\\[3mm]
+\infty & \text{otherwise},
\end{cases}
\end{equation*}
whenever $\mu\in\Prob_2(X)$. 
Note that, by the Bishop Volume Comparison Theorem, it actually holds that $\Ent(\mu)>-\infty$ for all $\mu\in\Prob_2(\M)$, see~\cite{E10}*{Lemma~4.1} for the proof.
The $\CD(K,\infty)$ condition can be equivalently reformulated as a convexity property of the entropy along all (constant speed, as usual) geodesics joining two measures in~$\Prob_2(\M)$. 
More precisely, if $[0,1]\in s\mapsto\mu_s\in\Prob_2(\M)$ is a geodesic joining $\mu_0,\mu_1\in\Prob_2(\M)$, then the following \emph{displacement $K$-convexity inequality}
\begin{equation}\label{intro_eq:K-displacement}
\Ent(\mu_s)
\le
(1-s)\,\Ent(\mu_0)
+
s\,\Ent(\mu_1)
-\frac K2\,s(1-s)\,W_2^2(\mu_0,\mu_1)
\end{equation}
holds for all $s\in[0,1]$.  

\subsection{The non-smooth framework: \texorpdfstring{$\CD(K,\infty)$}{CD(K,infty)} spaces}

The peculiar feature of inequality~\eqref{intro_eq:K-displacement} is that it can be stated uniquely in terms of the distance and the volume measure, no matter they come from the underlying smooth structure of the Riemannian manifold, and thus can be considered as a metric-measure definition of the lower bound on the Ricci tensor.

This observation has led Lott--Villani and Sturm in their groundbreaking works~\cites{LV09,S06.I,S06.II} to study the properties of very general metric-measure spaces $(X,\d,\m)$ satisfying the displacement $K$-convexity for some $K\in\R$.
Besides the many powerful consequences successively derived from their ideas, see~\cites{G10,LV07,R12} and the monograph~\cite{V09} for example, a key feature of the Lott--Sturm--Villani approach is that the displacement $K$-convexity of the entropy actually provides a metric-measure definition of the $\CD(K,\infty)$ condition that is stable under (metric-measure) Gromov--Hausdorff convergence.

As pointed out by Gigli~\cite{G10}, starting from the $\CD(K,\infty)$ condition, it is possible to prove that the \emph{metric gradient flow} (see the monograph~\cite{AGS08}) $(\mathsf S_t)_{t>0}$ of the entropy functional in $(\Prob_2(X),W_2)$ is an evolution semigroup on the convex subset of $\Prob_2(X)$ given by probability measures with finite entropy.
However, since also Finsler geometries (as in the flat case of $\R^N$ endowed with a non-Euclidean norm) can satisfy the $\CD(K,\infty)$ condition, the semigroup $(\mathsf S_t)_{t>0}$ can be non-linear in such a general setting.
Nevertheless, $(\mathsf S_t)_{t>0}$ can be extended to a continuous semigroup of contractions in $\Leb^2(X,\m)$ (and actually in any $\Leb^p$-space) which can be also characterized  as the gradient flow in $\Leb^2(X,\m)$ of the convex and $2$-homogeneous functional
\begin{equation}\label{intro_eq:Cheeger_energy}
\Ch(f)
=
\inf
\set*{
\liminf_{n\to+\infty}
\frac12
\int_X|\D f_n|^2\di\m
:
f_n\in\Lip_b(X),\ f_n\to f\ \text{in}\ \Leb^2(X,\m)
},
\end{equation} 
the \emph{Cheeger energy} of $f\in\Leb^2(X,\m)$, see the celebrated work~\cite{C99}.
Here
\begin{equation}\label{intro_eq:slope}
|\D f|(x)=\limsup\limits_{y\to x}\frac{|f(y)-f(x)|}{\d(y,x)}
\end{equation}
is the \emph{slope} at $x\in X$ of the bounded Lipschitz function $f\in\Lip_b(X)$.
Since the slope~\eqref{intro_eq:slope} plays the same role of the absolute value of the gradient in the smooth framework, it is natural to consider the gradient flow of the Cheeger energy as a metric-measure definition of the heat flow $(\P_t)_{t>0}$ in the non-smooth context.
The identification of the entropic semigroup and the heat flow has been proved in Euclidean spaces in~\cite{JKO98} by Jordan--Kinderleher--Otto (see also~\cite{O01}) and then extended to Riemannian manifolds~\cites{E10,V09}, Hilbert spaces~\cite{ASZ09}, Finsler spaces~\cite{OS09}, Alexandrov spaces~\cite{GKO13} and eventually to $\CD(K,\infty)$ spaces in the fundamental work~\cite{AGS14} of Ambrosio--Gigli--Savaré.
We refer the reader also to the works of Kuwada~\cites{K10,K13} for their key role in the understanding of the equivalence between the gradient estimate~\eqref{intro_eq:P_t_contractivity} and the $W_2$-contraction inequality~\eqref{intro_eq:H_t_contractivity} in the non-smooth framework.

Having a metric-measure notion of heat flow $(\P_t)_{t>0}$ at hand, it is then natural to see if the displacement $K$-convexity is still equivalent to suitable analogues of the Bakry--\'Emery inequality~\eqref{intro_eq:P_t_contractivity} and the $W_2$-contractivity property~\eqref{intro_eq:H_t_contractivity} in this abstract setting.
Building upon the non-smooth Calculus developed in~\cite{AGS14}, Ambrosio--Gigli--Savaré in~\cites{AGS14-2,AGS15} and Ambrosio--Gigli--Mondino--Rajala in~\cite{AGMR15} proved this equivalence under the additional assumption that the heat flow $(\P_t)_{t>0}$ is linear or, equivalently, that the Cheeger energy~\eqref{intro_eq:Cheeger_energy} is a Dirichlet (and thus \emph{quadratic}) form on $\Leb^2(X,\m)$, in order to naturally rule out Finsler geometries.
For this reason, such metric-measure spaces, forming a smaller family of $\CD(K,\infty)$ spaces remarkably still stable under Gromov--Hausdorff convergence, are said to have \emph{Riemannian} Ricci curvature bounded from below by $K\in\R$, or \emph{infinitesimally Hilbertian} $\CD(K,\infty)$ spaces, or $\RCD(K,\infty)$ spaces for short.
We refer the reader to~\cite{OS14} and to~\cites{GKO13,ZZ12} for strictly related results in Finsler and Alexandrov spaces respectively.

One of the most important results of~\cites{AGS14-2,AGMR15} is that $\RCD(K,\infty)$ spaces can be equivalently characterized as those metric-measure spaces for which the gradient flow $(\S_t)_{t>0}$ of the entropy in the Wasserstein space $(\Prob_2(X),W_2)$ satisfies the following \emph{Evolution Variational Inequality} with parameter $K\in\R$, $\EVI_K$ for short, 
\begin{equation}\label{intro_eq:EVI}
\frac{\di}{\di t}\,
\frac{W_2^2(\S_t\mu,\nu)}{2}
+\frac{K}{2}\,
W_2^2(\S_t\mu,\nu)
+
\Ent(\S_t\mu)
\le
\Ent(\nu)
\quad
\text{for a.e.}\ t>0
\end{equation}  
whenever $\mu,\nu\in\Prob_2(X)$.
Thus, in the infinitesimally Hilbertian case (and so in the particular case of smooth Riemannian manifolds), inequality~\eqref{intro_eq:EVI} provides an alternative equivalent metric-measure formulation of lower bound on the Ricci curvature.

The above analysis has been extended also to the finite dimensional case $N\in(0,+\infty)$, where however the equivalent metric-measure formulations of the lower bound on the Ricci curvature and the upper bound on the dimension become more involved.
We refer the reader to the seminal works of Erbar--Kuwada--Sturm~\cite{EKS15} and Kuwada~\cite{K15}, and to the more recent developments obtained in~\cites{BGG14,BGGK18,Ge15}.
The theory of $\CD(K,N)$ spaces has been extended also to the case of \emph{negative} dimension $N\in(-\infty,0)$, see~\cite{O16}. 
Finally, we refer the reader to the recent work~\cite{S19} by Sturm for a more general approach to curvature in the metric-measure setting.

\subsection{The sub-Riemannian framework}

Although the class of $\CD(K,N)$ spaces is very broad, a large and widely-studied family of spaces is left out, the sub-Riemannian ma\-ni\-folds.
For an introduction on the subject, we refer the reader to the papers~\cites{LeD17,S86,S89} and to the monographs~\cites{ABB20,M02,BLU07}.

A \emph{sub-Riemannian manifold} is a triple $(\M,\HB,\scalar*{\cdot,\cdot}_{\HB})$, where $\M$ is a (connected) smooth manifold, $\HB\subset T\M$ is a sub-bundle of the tangent bundle~$T\M$ and~$\scalar*{\cdot,\cdot}_{\HB}$ is a smoothly varying positive definite quadratic form on~$\HB$. 
Typically, the sub-bundle~$\HB$ is assumed to be \emph{bracket ge\-ne\-ra\-ting} and \emph{equiregular}, that is, at each point $x\in\M$ the directions in~$\HB_x$ together with all their Lie brackets generate the full tangent space~$T_x\M$, and the dimensions of the intermediate sub-bundles of commutators obtained at each step do not depend on the choice of $x\in\M$.

A sub-Riemannian manifold $(\M,\HB,\scalar*{\cdot,\cdot}_{\HB})$ naturally carries a metric notion, the so-called \emph{Carnot--Carathéodory} (CC for short) \emph{distance}, defined as
\begin{equation}
\label{intro_eq:CC_distance}
\dcc(x,y)
=
\inf\set*{
\int_0^1|\dot\gamma_t|_{\HB}^2\di t
:
\gamma\colon[0,1]\to\M\ \text{is horizontal},\ \gamma_0=x,\ \gamma_1=y
}
\end{equation}
for all $x,y\in\M$. 
A Lipschitz curve $\gamma\colon[0,1]\to\M$ is \emph{horizontal} if $\dot\gamma_t\in\HB_{\gamma_t}$ for a.e.\ $t\in[0,1]$.
The bracket generating assumption on $\HB$ ensures that the function in~\eqref{intro_eq:CC_distance} is a finite distance --- this is the celebrated Chow--Rashevskii Theorem, see~\cite{ABB20}*{Theorem~3.31} for example.

A sub-Riemannian manifold $(\M,\HB,\scalar*{\cdot,\cdot}_{\HB})$ can be endowed with the Hausdorff measure associated to~$\dcc$.
However, differently from the Riemannian case, the Hausdorff measure does not coincide in general with the volume induced by the  distribution~$\HB$, the \emph{Popp measure}~$\mu_{\HB}$. 
We refer the reader to~\cite{ABB20}*{Chapter~20} for the precise definitions and constructions. 
 
The $\CD(K,N)$ condition fails for the metric-measure space $(\M,\dcc,\mu_{\HB})$ and, for this reason, sub-Riemannian manifolds are said to have Ricci curvature unbounded from below. 
The first result in this direction was obtained by Juillet in~\cite{J09} for the \emph{Heisenberg group} $\mathbb H^N$, building upon some results on optimal transportation in $\mathbb H^N$ established by Ambrosio--Rigot in the seminal paper~\cite{AR04}.
Later, by exploiting a result of~\cite{BB16}, Ambrosio and the author in~\cite{AS19} showed that any non-commutative \emph{Carnot group} is not a $\CD(K,\infty)$ space.
Carnot groups, of which~$\R^N$ and~$\mathbb H^N$ are the simplest examples, are nilpotent Lie groups that, in some sense, capture the local infinitesimal behavior of sub-Riemannian manifolds. 
Precisely, by a famous result of Mitchell~\cite{M85}, Carnot groups are the \emph{tangent metric cones} to sub-Riemannian manifolds, see~\cite{LeD17}*{Section~4.1} and the references therein.
Recently, among other non-embedding results, Huang--Sun~\cite{HS20} proved that equiregular sub-Riemannian manifolds do not satisfy the $\CD(K,N)$ condition, and Juillet~\cite{J20} treated the case of rank-varying distributions.

Despite their intrinsic wild nature of non-$\CD$ spaces, in the last fifteen years sub-Riemanninan manifolds have become an active and promising research topic for the study of Optimal Transport, heat and entropy flows and generalized curvature notions beyond the well-established Riemannian and $\CD$ frameworks.

After the pioneering work~\cite{AR04} of Ambrosio--Rigot, the well-posedness of optimal transportation was studied in Heisenberg and $H$-type groups~\cites{FJ08,DR11,R05}, for non-holonomic distributions~\cites{AL09,KL09} and in more general sub-Riemannian manifolds~\cites{B19,FR10}.

The identification between the heat and the entropy semigroups was established first in the Heisenberg groups by Juillet in~\cite{J14} and then in all Carnot groups by Ambrosio and the author in~\cite{AS19}.

Several notions of \emph{curvature} in sub-Riemannian manifolds have been introduced in recent years, following either the \emph{Lagrangian} or the \emph{Eulerian} approach.

The \emph{Lagrangian} point of view has its roots in the study of Jacobi vector fields initiated in the fundamental works of Agrachev--Li--Zelenko~\cites{AZ02-I,AZ02-II,ZL09} (ALZ for short) and later developed in~\cites{ABR18,BR16,BR17,Grong18,Grong20,M17}.
Besides the numerous applications inspired by some classical results of Riemannian Geometry, see~\cites{ABR17,AL14,BI19,BR16,LLZ16,RS19} for example, a deep and powerful byproduct of the Lagrangian approach --- in the original spirit of Cordero-Erausquin--McCann--Schmucken\-schl\"ager~\cite{CMS01} --- is a precise control of the \emph{distorsion coefficients} in the \emph{sub-Riemannian interpolation inequalities}. 
These results were obtained for the first time by Balogh--Krist\'aly--Sipos in Heisenberg and corank-$1$ Carnot groups~\cites{BKS18,BKS19} via direct methods based on the special structure of these spaces.
The link with the ALZ theory of sub-Riemannian Jacobi fields was made manifest shortly after by Barilari--Rizzi in the more general context of \emph{ideal structures}~\cites{BR19,BR20}.

Sub-Riemannian interpolation inequalities have two interesting consequences: the \emph{Measure Contraction Property}, $\MCP$ for short, and the \emph{distorted} displacement convexity of the entropy along Wasserstein geodesics.

The $\MCP(K,N)$ condition, introduced for the first time by Ohta~\cite{O07}, keeps track of the distortion of the volume of a set when it is transported to a Dirac delta. 
Although for a Riemannian manifold the $\MCP(K,N)$ and the $\CD(K,N)$ conditions are equivalent (with~$N$ the \emph{topological} dimension of the manifold), the $\MCP(K,N)$ condition is in general weaker than the $\CD(K,N)$ condition. 
The first result in this direction was obtained by Juillet~\cite{J09} for the Heisenberg group $\mathbb H^N$ (see also~\cite{CM16}).
The same property was then proved for other Carnot groups~\cites{BR18,R16} and later established for more general sub-Riemannian manifolds in~\cites{BR19,BR20}.

Concerning the applications of the $\MCP(K,N)$ condition and the interpolation inequalities in the sub-Riemannian framework, we refer the reader to the work~\cite{Milman19}, where the author obtains $\Leb^p$-Poincaré and $\log$-Sobolev inequalities on domains via a quasi-convex relaxation of the $\CD(K,N)$ condition, namely, the \emph{quasi curvature-dimension condition $\mathsf{QCD}(Q,K,N)$}, where $Q\in[1,+\infty)$ is an additional `slack' parameter.

The entropy $\Ent_{\HB}$ with respect to the Popp measure $\mu_{\HB}$ of the sub-Riemannian manifolds $(\M,\dcc,\mu_{\HB})$ considered in~\cites{BKS18,BKS19,BR19,BR20} satisfies a \emph{distorted} displacement convexity inequality in the following sense: 
if $[0,1]\in s\mapsto\mu_s\in\Prob_2(\M)$ is the geodesic connecting two probability measures $\mu_0,\mu_1\in\Prob_2(\M)$ with compact support, then
\begin{equation}\label{intro_eq:strange-displacement}
\Ent_{\HB}(\mu_s)
\le
(1-s)\,\Ent_{\HB}(\mu_0)
+
s\,\Ent_{\HB}(\mu_1)
+w(s)
\end{equation}
for all $s\in[0,1]$, where $w\colon[0,1]\to[0,+\infty)$ is a function, concave and such that $w(0)=w(1)=1$, depending only on the lower bounds on the distorsion coefficients of $(\M,\dcc,\mu_{\HB})$ and compensating the lack of $K$-con\-ve\-xi\-ty of the function $s\mapsto\Ent_{\HB}(\mu_s)$.

Although not strictly related to the present work, for the sake of completeness we warn the reader that there are other lines of research in the Lagrangian direction for the definition of curvature in the sub-Riemannian context besides the ALZ approach. We refer the interested reader to~\cites{B12,BF13,H12} for generalizations of the notion of connection and to~\cite{Bere17} for the so-called \emph{Solov'ev method}.

The \emph{Eulerian} point of view arises from the fundamental work of Baudoin--Garofalo~\cite{BG17} (BG for short) and relies on a clever adaptation of the Bakry-\'Emery curvature-dimension inequality~\eqref{intro_eq:CD_ineq} to sub-Riemannian manifolds with \emph{transverse symmetries}.
Roughly said, the tangent space of the sub-Riemannian manifolds considered in~\cite{BG17} splits into the aforementioned subspace of \emph{horizontal} directions~$\HB$ and a subspace of \emph{vertical} directions~$\VB$.
To this splitting, it is possible to associate two $\Gamma$-operators, the usual horizontal one~$\Gamma^\HB$ associated to the CC distance and the horizontal Laplacian $\Delta_\HB$, and a new vertical one~$\Gamma^\VB$ which satisfies the commutation property
\begin{equation}
\label{intro_eq:commutation_prop_Gamma}
\Gamma^\HB(f,\Gamma^\VB(f))
=
\Gamma^\VB(f,\Gamma^\HB(f))
\end{equation}
for all $f\in\Cont^\infty(\M)$.
Property~\eqref{intro_eq:commutation_prop_Gamma} is typical of step~2 distributions~$\HB$, where $\VB=[\HB,\HB]$ and $[\VB,\HB]=0$.
Defining 
\begin{align*}
\Gamma_2^\HB(f,g)
&=
\frac12\big(
\Delta_\HB\Gamma^\HB(f,g)-\Gamma^\HB(f,\Delta_\HB g)-\Gamma^\HB(\Delta_\HB f,g)\big),\\
\Gamma_2^\VB(f,g)
&=
\frac12\big(
\Delta_\HB\Gamma^\VB(f,g)-\Gamma^\VB(f,\Delta_\HB g)-\Gamma^\VB(\Delta_\HB f,g)\big),
\end{align*}
for all $f,g\in\Cont^\infty(\M)$, as in the Riemannian case~\eqref{intro_eq:Gamma_def}, the \emph{generalized BG curvature-dimension inequality}, $\CD(K_\HB,K_\VB,\kappa,N)$ for short, with $K_\HB\in\R$, $K_\VB>0$, $\kappa\ge0$ and $N\in(0,+\infty]$, amounts to say that
\begin{equation}
\label{intro_eq:BG_CD}
\Gamma_2^\HB(f)
+
\eps\,\Gamma_2^\VB(f)
\ge
\frac1N\,(\Delta_\HB f)^2
+
\left(K_\HB-\frac\kappa\eps\right)\,\Gamma^\HB(f)
+
K_\VB\,\Gamma^\VB(f)
\end{equation} 
holds for all $f\in\Cont^\infty(\M)$ and all $\eps>0$.
In~\eqref{intro_eq:BG_CD}, the parameter $K_\HB\in\R$ plays the role of the lower bound on the `generalized Ricci tensor' and, if $\kappa>0$, then $\left(K_\HB-\frac\kappa\eps\right)\to-\infty$ as~$\eps\to0^+$, coherently with the non-$\CD$ nature of sub-Riemannian spaces.
The usual $\CD(K,N)$ condition is thus recovered when~$\Gamma^\VB=0$ and~$\kappa=0$, with~$K=K_\HB$.

The BG theory has been developed in several directions, see the numerous applications obtained in~\cites{B17,BB15,BBG14,BBGM14,BC15,BK14,BW14} and the generalizations made in~\cites{GT16-I,GT16-II}.
We refer the reader also to the recent works~\cites{FengLi20-1,FengLi20-2} where condition~\eqref{intro_eq:commutation_prop_Gamma} is dropped.
A simple but interesting consequence of~\eqref{intro_eq:BG_CD} is the following pointwise gradient bound for the heat flow~$(\P_t)_{t>0}$ associated to the horizontal Laplacian~$\Delta_\HB$, in analogy with~\eqref{intro_eq:P_t_contractivity}: there exists $\alpha\in\R$ such that
\begin{equation}\label{intro_eq:BG_P_t_contractivity}
\Gamma^\HB(\P_tf)
+
\Gamma^\VB(\P_tf)
\le
e^{-2\alpha t}\,
\big(
\P_t\Gamma^\HB(f)
+
\P_t\Gamma^\VB(f)
\big)
\end{equation} 
for all~$f\in\Cont^\infty_c(\M)$ and~$t\ge0$, see~\cite{BG17}*{Corollary~4.6}.

Pointwise gradient bounds for the heat flow~$(\P_t)_{t>0}$ associated to the horizontal Laplacian~$\Delta_\HB$, similar to~\eqref{intro_eq:BG_P_t_contractivity} but closer to the Riemannian one~\eqref{intro_eq:P_t_contractivity}, were proved for the first time by Driver--Melcher for the Heisenberg groups~\cite{DM05} and later generalized to all Carnot groups by Melcher~\cite{M08} (see also~\cite{GT19} for a different proof). 
Baudoin--Bonnefont obtained similar inequalities for the $\SU(2)$ group in~\cite{BB09}.
Stronger inequalities have been proved for the Heisenberg groups~\cites{BBBC08,L06}, $H$-type groups~\cite{HL10} and the Grushin plane~\cite{W14} with different techniques, and very recently Baudoin--Kelleher treated the case of metric graphs via the theory of differential forms on Dirichlet spaces~\cite{BK19} (concerning Dirichlet spaces, we also refer the reader to~\cite{CJKS20} for a strictly related, although slightly weaker, pointwise gradient bound).

In all the spaces quoted above, the heat flow~$(\P_t)_{t>0}$ satisfies an inequality of the form 
\begin{equation}\label{intro_eq:P_t_contractivity_weak}
\Gamma(\P_t f)
\le
Ce^{-2Kt}\,\P_t\Gamma(f)
\end{equation}
for all $f\in\Cont^\infty_c(\M)$ and~$t\ge0$, for some constants $C\ge1$ and $K\in\R$ (with $K=0$ for Carnot groups, coherently with their homogeneous nature).
Since~\eqref{intro_eq:P_t_contractivity_weak} reduces to~\eqref{intro_eq:P_t_contractivity} when $C=1$, by analogy with the $\CD$~framework the (optimal) parameter $K\in\R$ in~\eqref{intro_eq:P_t_contractivity_weak} can be thought as a lower bound on the `generalized Ricci tensor'.

Accordingly to this interpretation, thanks to the celebrated work~\cite{K10} of Kuwada, the pointwise gradient bound~\eqref{intro_eq:P_t_contractivity_weak} is equivalent to the following contractivity property of the dual heat flow~$(\H_t)_{t>0}$ with respect to the $2$-Wasserstein distance: if $\mu,\nu\in\Prob_2(\M)$, then
\begin{equation}
\label{intro_eq:H_t_contractivity_weak}
W_2(\H_t\mu,\H_t\nu)
\le
\sqrt{C}\,e^{-Kt}\,
W_2(\mu,\nu)
\end{equation}
for all~$t\ge0$.
Contractivity properties like~\eqref{intro_eq:H_t_contractivity_weak} have been established also for the Markovian diffusion semigroup associated to the operator $L=\Delta+Z$ on a (connected and complete) Riemannian manifold $(\M,\g)$, where~$\Delta$ is the usual Laplace--Beltrami operator and $Z$ is a $\Cont^1$-regular vector field, see~\cites{CTZ21,W20}.

In view of the equivalence between~\eqref{intro_eq:P_t_contractivity_weak} and~\eqref{intro_eq:H_t_contractivity_weak}, and in analogy with the $\CD$~framework, the study of~\eqref{intro_eq:P_t_contractivity_weak}, or equivalently of~\eqref{intro_eq:H_t_contractivity_weak}, belongs to the Eulerian side of the approach to the definition of curvature in the sub-Riemannian setting.

\subsection{Sub-Riemannian groups as weak \texorpdfstring{$\RCD$}{RCD} spaces}

At the present moment, no link is known between the Lagrangian and the Eulerian approach presented above, in the sense that no relation has been shown between the distorted convexity of the entropy~\eqref{intro_eq:strange-displacement} and the inequalities~\eqref{intro_eq:BG_CD} and~\eqref{intro_eq:P_t_contractivity_weak} satisfied by the $\Gamma$~operator and the heat flow~$(\P_t)_{t>0}$, in the same manner of the $\CD$~framework.

The main contribution of this work is to make a partial step towards the connection between the Langrangian and the Eulerian approach in the sub-Riemannian context.
We show that the pointwise gradient bound~\eqref{intro_eq:P_t_contractivity_weak} is equivalent to a \emph{heated} version of the displacement convexity inequality~\eqref{intro_eq:K-displacement} and an \emph{almost-integrated} form of the Evolution Variational Inequality~\eqref{intro_eq:EVI} in the context of metric-measure spaces with a group structure, extending to this non-$\CD$ setting the dimension-free results obtained by Ambrosio--Gigli--Savaré~\cites{AGS13,AGS14-2} and Ambrosio--Gigli--Mondino--Rajala~\cite{AGMR15}.
Since these inequalities naturally embed the corresponding ones of the $\CD$~framework, our main equivalence result can be seen as an attempt to understand the problem of the \emph{grande unification synthétique} proposed by Villani~\cite{V19} for the special case of metric-measure groups.
The present work was also motivated by some questions raised by Balogh--Krist\'aly--Sipos in~\cite{BKS18}*{Section~5}.

Let us give a sketch of our idea.
We start by assuming that, in a metric-measure space $(X,\d,\m)$, the metric heat flow $(\P_t)_{t>0}$, i.e., the gradient flow of the Cheeger energy associated to the distance~$\d$ (recall~\eqref{intro_eq:Cheeger_energy} for the definition) is linear and satisfies the pointwise gradient bound  
\begin{equation}\label{intro_eq:wBE}
\Gamma(\P_t f)
\le
\c^2(t)\,
\P_t\Gamma(f)
\quad
\text{$\m$-a.e.\ in $X$}
\end{equation}
for all $t\ge0$ and all $f\in\Dom(\Ch)$, for some function $\c\colon[0,+\infty)\to(0,+\infty)$ locally positively bounded from above and below.

In this non-smooth context, we precisely have $\Gamma(f)=|\nabla f|_w^2$ for all $f\in\Dom(\Ch)$, where $|\nabla f|_w\in\Leb^2(X,\m)$ is the so-called \emph{minimal relaxed gradient of~$f$} in the sense of~\cite{AGS14}*{Definition~4.2} and represents the Cheeger energy~\eqref{intro_eq:Cheeger_energy} as
\begin{equation*}
\Ch(f)
=
\frac12\,\int_X |\nabla f|_w^2\di\m
\end{equation*}  
for all $f\in\Dom(\Ch)$.
However, to avoid technicalities, in what follows we simply consider~$X$ as a sub-Riemannian manifold and~$\Gamma$ as the squared modulus of the gradient, $\Gamma(f)=|\nabla f|^2$.

We can think of the function~$\c$ in~\eqref{intro_eq:wBE} as the \emph{curvature function} of the space $(X,\d,\m)$ replacing the function $t\mapsto e^{-Kt}$ of the standard pointwise gradient~\eqref{intro_eq:P_t_contractivity}.
Actually, thanks to the Fekete Lemma for sub-additive functions, the \emph{optimal} curvature function~$\c$ in~\eqref{intro_eq:wBE} does satisfy $\c(t)\le Ce^{-Kt}$ for all $t\ge0$ for some $C\ge1$ and $K\in\R$, as for the pointwise gradient bound~\eqref{intro_eq:P_t_contractivity_weak}, provided that 
$\limsup\limits_{t\to0^+}\c(t)<+\infty$.
The (optimal) constant $K\in\R$ plays the role of the lower bound on the `generalized Ricci tensor' in this situation.
We call inequality~\eqref{intro_eq:wBE} the \emph{weak Bakry--\'Emery curvature condition} with respect to the curvature function~$\c$, $\BE_w(\c,\infty)$ for short.

In this general framework, the equivalence between the pointwise gradient bound~\eqref{intro_eq:wBE} and the $W_2$-contractivity property of the dual heat flow $(\H_t)_{t>0}$,
\begin{equation}\label{intro_eq:w_Kuwada}
W_2(\H_t\mu,\H_t\nu)
\le
\c(t)\,
W_2(\mu,\nu)
\end{equation}
for all $\mu,\nu\in\Prob_2(X)$ and $t\ge0$, has already been addressed by Ambrosio--Gigli--Savaré in~\cite{AGS13}*{Section~3.2} adapting the original idea of Kuwada~\cite{K10}.
Actually, in~\cite{AGS13} only the implication \eqref{intro_eq:wBE}~$\Rightarrow$~\eqref{intro_eq:w_Kuwada} is proved in detail, while the other implication \eqref{intro_eq:w_Kuwada}~$\Rightarrow$~\eqref{intro_eq:wBE} --- of no need for the scopes of~\cite{AGS13} --- is only stated with a sketch of its proof.
However, the line suggested in~\cite{AGS13} for the proof of this implication is not completely correct, see \cref{rem:AGS_pernacchia} below for the technical details.
Our first task is thus to amend the strategy of~\cite{AGS13} and to give a self-contained and complete proof of the equivalence between~\eqref{intro_eq:wBE} and~\eqref{intro_eq:w_Kuwada}.

Having the correspondence between the pointwise gradient bound~\eqref{intro_eq:wBE} and the $W_2$-contractivity property~\eqref{intro_eq:w_Kuwada} of the dual heat semigroup at hand, we can focus on the proof of the \emph{almost-integrated} form of the~$\EVI$.
The following (formal) computations are a sketch of the \emph{action estimates} performed in~\cite{AGS14}*{Section~4.3} for the dimension-free case $N=+\infty$. 
Actually, our approach takes advantage of the more general point of view assumed in~\cite{EKS15}*{Section~4.2} for the finite dimensional case $N<+\infty$.
For the presentation, we also took inspiration from~\cite{BGL15}*{Section~6}.

Let $s\mapsto\mu_s=f_s\m$, $s\in[0,1]$, be a curve in the $2$-Wasserstein space joining two measures $\mu_0,\mu_1\in\Prob_2(X)$ and let us define a new curve $s\mapsto\tilde\mu_s=\tilde f_s\m$, $s\in[0,1]$, by setting
\begin{equation*}
\tilde\mu_s=\H_{\eta(s)}\mu_{\theta(s)},\
\text{so that}\
\tilde f_s=\P_{\eta(s)}f_{\theta(s)},\
\text{for all}\
s\in[0,1],
\end{equation*}
where $\eta\in\Cont^2([0,1];[0,+\infty))$ and $\theta\in\Cont^1([0,1];[0,1])$ with $\theta(0)=0$ and $\theta(1)=1$.
At least formally, we can compute
\begin{align*}
\frac{\di}{\di s}\,\tilde f_s
=
\dot\eta(s)\,\Delta\P_{\eta(s)}f_{\theta(s)}
+
\dot\theta(s)\,\P_{\eta(s)}\dot f_{\theta(s)}
\end{align*} 
for $s\in(0,1)$, where $\Delta$ is the infinitesimal generator of the heat flow, the (metric-measure) Laplacian operator.

On the one hand, integrating by parts, we get
\begin{align*}
\frac{\di}{\di s}\,\Ent_\m(\tilde\mu_s)
&=
\frac{\di}{\di s}\int_X\tilde f_s\log\tilde f_s\di\m\\
&=
\int_X(1+\log\tilde f_s)\,\frac{\di}{\di s}\,\tilde f_s\di\m\\
&=
-\dot\eta(s)\int_X p'(\tilde f_s)\,\Gamma(\tilde f_s)\di\m
+
\dot\theta(s)\int_X p(\tilde f_s)\,\P_{\eta(s)}\dot f_{\theta(s)}
\di\m
\end{align*}
for $s\in(0,1)$, where $p(r)=1+\log r$ for all $r>0$.
Observing that $p'(r)=r(p'(r))^2$ for all~$r>0$, by the \emph{chain rule} 
$\Gamma(\phi(f))=(\phi'(f))^2\,\Gamma(f)$
valid for all $\phi\colon\R\to\R$ sufficiently smooth and all $f\in\Dom(\Ch)$, we can write
\begin{equation}\label{intro_eq:Ent_estimate}
\frac{\di}{\di s}\,\Ent_\m(\tilde\mu_s)
=
-\dot\eta(s)\int_X \Gamma(g_s)\di\tilde\mu_s
+
\dot\theta(s)\int_X \dot f_{\theta(s)}\,\P_{\eta(s)}g_s
\di\m
\end{equation}
for $s\in(0,1)$, where we have set $g_s=p(\tilde f_s)$ for brevity. 

On the other hand, by Kantorovich duality, we can write
\begin{equation}\label{intro_eq:kantorovich}
\frac12\,W_2^2(\mu,\nu)
=
\sup\set*{
\int_X Q_1\phi\di\mu-\int_X\phi\di\nu : \phi\in\Lip(X)\ \text{with bounded support}},
\end{equation}
where
\begin{equation*}
Q_s\phi(x)
=
\inf_{y\in X}\phi(y)+\frac{\d^2(y,x)}{2s},
\quad 
\text{for}\ x\in X\ \text{and}\ s>0,
\end{equation*}
is the \emph{Hopf--Lax infimum-convolution semigroup}.
Recalling that $\phi_s=Q_s\phi$ solves the Hamilton--Jacobi equation  
$\de_s\phi_s+\frac12\,|\nabla\phi_s|^2=0$,
again integrating by parts we can compute
\begin{equation}\label{intro_eq:action_etimate}
\begin{split}
\frac{\di}{\di s}\int_X\phi_s\,\tilde f_s\di\m
&=
\int_X\de_s\phi_s\di\tilde\mu_s
+
\int_X\phi_s\,\frac{\di}{\di s}\tilde f_s\di\m\\
&=
-\frac12\int_X\Gamma(\phi_s)\di\tilde\mu_s-\dot\eta(s)\int_X\Gamma(\phi_s,\tilde f_s)\di\m+\dot\theta(s)\int_X\dot f_{\theta(s)}\,\P_{\eta(s)}\phi_s\di\m
\end{split}
\end{equation}
for $s\in(0,1)$.
We can combine~\eqref{intro_eq:Ent_estimate} and~\eqref{intro_eq:action_etimate} to get
\begin{equation}
\label{intro_eq:total_action}
\begin{split}
\frac{\di}{\di s}\int_X\phi_s\,\tilde f_s\di\m
+
\dot\eta(s)\,\frac{\di}{\di s}\,\Ent_\m(\tilde\mu_s)
&\le
-\frac12\int_X\big(\Gamma(\phi_s)+\dot\eta(s)^2\,\Gamma(g_s)\big)\di\tilde\mu_s\\
&\quad
-\dot\eta(s)\int_X\Gamma(\phi_s,\tilde f_s)\di\m\\
&\quad
+\dot\theta(s)\int_X\dot f_{\theta(s)}\,\P_{\eta(s)}(\phi_s+\dot\eta(s)\,g_s)\di\m
\end{split}
\end{equation}
for $s\in(0,1)$, forgetting the term $-\frac{\dot\eta(s)^2}{2}\int_X \Gamma(g_s)\di\tilde\mu_s\le0$ in~\eqref{intro_eq:Ent_estimate}.
Now
\begin{align*}
\Gamma(\phi_s+\dot\eta(s)\,g_s)
=\Gamma(\phi_s)
+
2\,\dot\eta(s)\,\Gamma(\phi_s,g_s)
+\dot\eta(s)^2\,\Gamma(g_s)
\end{align*}
and, by the chain rule,
\begin{align*}
\Gamma(\phi_s,g_s)
=
\Gamma(\phi_s,p(\tilde f_s))
=
p'(\tilde f_s)\,\Gamma(\phi_s,\tilde f_s).
\end{align*}
Since $r\,p'(r)=1$, we have
\begin{align*}
\int_X\Gamma(\phi_s,g_s)\di\tilde\mu_s
=
\int_X \tilde f_s\,p'(\tilde f_s)\,\Gamma(\phi_s,\tilde f_s)\di\m
=
\int_X \Gamma(\phi_s,\tilde f_s)\di\m,
\end{align*}
and thus~\eqref{intro_eq:total_action} simplifies to
\begin{equation}
\label{intro_eq:total_action_bis}
\begin{split}
\frac{\di}{\di s}\int_X\phi_s\,\tilde f_s\di\m
+
\dot\eta(s)\,\frac{\di}{\di s}\,\Ent_\m(\tilde\mu_s)
&\le
-\frac12\int_X\Gamma(\phi_s+\dot\eta(s)\,g_s)\di\tilde\mu_s\\
&\quad
+\dot\theta(s)\int_X\dot f_{\theta(s)}\,\P_{\eta(s)}(\phi_s+\dot\eta(s)\,g_s)\di\m
\end{split}
\end{equation}
for $s\in(0,1)$. 
At this point, the crucial information we need to know about the chosen curve $s\mapsto\mu_s=f_s\m$ is that
\begin{equation}
\label{intro_eq:lisini}
\int_X \dot f_s\,\psi\di\m
\le
|\dot\mu_s|\,\left(\int_X\Gamma(\psi)\di\mu_s\right)^{\frac12}
\end{equation}
for all sufficiently `nice' functions~$\psi\in\Dom(\Ch)$, where 
$|\dot\mu_s|=\lim\limits_{h\to0}\frac{W_2(\mu_{s+h},\mu_s)}h$, 
$s\in(0,1)$, is the \emph{metric velocity} of the curve $s\mapsto\mu_s$ with respect to the $2$-Wasserstein distance.
With~\eqref{intro_eq:lisini} at disposal, we may choose $\psi=\P_{\eta(s)}(\phi_s+\dot\eta(s)\,g_s)$ and estimate
\begin{equation}
\label{intro_eq:AGS_pernacchietta}
\begin{split}
\dot\theta(s)\int_X\dot f_{\theta(s)}\,\P_{\eta(s)}(\phi_s&+\dot\eta(s)\,g_s)\di\m
=
\int_X\left(\frac{\d}{\di s}\,f_{\theta(s)}\right)\,\P_{\eta(s)}(\phi_s+\dot\eta(s)\,g_s)\di\m\\
&\le
|\dot\theta(s)|\,|\dot\mu_{\theta(s)}|
\left(\int_X\Gamma(\P_{\eta(s)}(\phi_s+\dot\eta(s)\,g_s))\di\mu_s\right)^{\frac12}\\
&\le
\frac{\c^2(\eta(s))}2\,\dot\theta(s)^2\,|\dot\mu_{\theta(s)}|^2
+
\frac{\c^{-2}(\eta(s))}2\,
\int_X\Gamma(\P_{\eta(s)}(\phi_s+\dot\eta(s)\,g_s))\di\mu_s\\
&\le
\frac{\c^2(\eta(s))}2\,\dot\theta(s)^2\,|\dot\mu_{\theta(s)}|^2
+
\frac12\,
\int_X\Gamma(\phi_s+\dot\eta(s)\,g_s)\di\tilde\mu_s
\end{split}
\end{equation}
by Young inequality, in virtue of~\eqref{intro_eq:wBE}.
By combining~\eqref{intro_eq:total_action_bis} with~\eqref{intro_eq:AGS_pernacchietta}, we conclude that
\begin{equation*}
\begin{split}
\frac{\di}{\di s}\int_X\phi_s\,\tilde f_s\di\m
+
\dot\eta(s)\,\frac{\di}{\di s}\,\Ent_\m(\tilde\mu_s)
&\le
\frac{\c^2(\eta(s))}2\,\dot\theta(s)^2\,|\dot\mu_{\theta(s)}|^2
\end{split}
\end{equation*}
for $s\in(0,1)$.
If we choose $\dot\theta(s)=\c^{-2}(\eta(s))$, then we can integrate in $s\in(0,1)$ so that, by Kantorovich duality~\eqref{intro_eq:kantorovich}, we finally get
\begin{equation}
\label{intro_eq:EVI_eta}
\begin{split}
\frac12\,W_2^2(\H_{\eta(1)}\mu_1,\H_{\eta(0)}\mu_0)
-
\frac1{2\RI(\eta)}&\,W_2^2(\mu_1,\mu_0)
+
\dot\eta(1)\,\Ent_\m(\H_{\eta(1)}\mu_1)\\
&\le
\dot\eta(0)\,\Ent_\m(\H_{\eta(0)}\mu_0)
+
\int_0^1\ddot\eta(s)\,\Ent_\m(\H_{\eta(s)}\mu_{\theta(s)})\di s,
\end{split}
\end{equation}
where $\RI(\eta)=\displaystyle\int_0^1\c^{-2}(\eta(s))\di s$.

Note that~\eqref{intro_eq:EVI_eta} is actually equivalent to the pointwise gradient bound~\eqref{intro_eq:wBE}.
Indeed, the choice of the constant function $\eta(s)=t$ for all $s\in[0,1]$ immediately gives~\eqref{intro_eq:w_Kuwada}.
Moreover, if $\c(t)=e^{-Kt}$ for some $K\in\R$, then we recover~\eqref{intro_eq:EVI} by choosing $\eta(s)=st$ for all $s\in[0,1]$.
Indeed, in this case, we obtain
\begin{align*}
\frac12\,W_2^2(\H_t\mu_1,\mu_0)
-
\frac{Kt}{e^{2Kt}-1}&\,W_2^2(\mu_1,\mu_0)
+
t\big(\Ent_\m(\H_t\mu_1)
-
\Ent_\m(\mu_0)\big)
\le0	
\end{align*}
and~\eqref{intro_eq:EVI} follows (for $t=0$, which is enough thanks to the semigroup property) by observing that $\frac{Kt}{e^{2Kt}-1}=\frac12(1-Kt+o(t))$ as~$t\to0^+$.
For this reason, and adopting the same terminology of~\cite{DS08}*{Proposition~3.1}, we may think of~\eqref{intro_eq:EVI_eta} as an \emph{almost-integrated} form of the~$\EVI$~\eqref{intro_eq:EVI}.

Since we have no information about the behavior of the function $s\mapsto\Ent_\m(\H_{\eta(s)}\mu_{\theta(s)})$, to simplify~\eqref{intro_eq:EVI_eta} it is convenient to choose $\eta(s)=(1-s)t_0+st_1$ for $s\in[0,1]$, where $0\le t_0\le t_1$ are fixed. 
With this choice, inequality~\eqref{intro_eq:EVI_eta} reduces to
\begin{equation}
\label{intro_eq:EVI_eta_simple}
\begin{split}
W_2^2(\H_{t_1}\mu_1,\H_{t_0}\mu_0)
+
2(t_1-t_0)\big(\Ent_\m(\H_{t_1}\mu_1)
-
\Ent_\m(\H_{t_0}\mu_0)\big)
\le
\mathsf A[t_0;t_1]^{-1}\,W_2^2(\mu_1,\mu_0),
\end{split}
\end{equation}
where $\mathsf A[t_0;t_1]=\displaystyle\aint_{t_0}^{t_1}\c^{-2}(s)\di s$.
In analogy with~\eqref{intro_eq:EVI}, we call the above inequality~\eqref{intro_eq:EVI_eta_simple} the \emph{weak Evolution Variational Inequality} with respect to the curvature function~$\c$, $\EVI_w(\c)$ for short.

Arguing exactly as in the proof of~\cite{DS08}*{Theorem~3.2}, from~\eqref{intro_eq:EVI_eta_simple} we deduce that, if $s\mapsto\mu_s$ is a $2$-Wasserstein (constant speed) geodesic, then
\begin{equation}
\label{intro_eq:Ent_heated_convex}
\begin{split}
\Ent_\m(\H_{t+h}\mu_s)
&\le 
(1-s)\,\Ent_m(\H_t\mu_0)+s\,\Ent_\m(\H_t\mu_1)\\
&\quad+\frac{s(1-s)}{2h}\left(\mathsf A[t;t+h]^{-1}\,W_2^2(\mu_1,\mu_0)-W_2^2(\H_t\mu_1,\H_t\mu_0)\right)
\end{split}
\end{equation}
for all~$t\ge0$ and~$h>0$.
Note that~\eqref{intro_eq:Ent_heated_convex} is still equivalent to the pointwise gradient bound~\eqref{intro_eq:wBE} since, by multiplying both of its sides by~$h>0$ and then letting~$h\to0^+$, we again recover~\eqref{intro_eq:w_Kuwada}.
Moreover, if we choose $t=0$ in~\eqref{intro_eq:Ent_heated_convex}, then we obtain
\begin{equation}
\label{intro_eq:Ent_heated_convex_t=0}
\begin{split}
\Ent_\m(\H_{h}\mu_s)
\le 
(1-s)\,\Ent_m(\mu_0)+s\,\Ent_\m(\mu_1)
+\frac{\mathsf B[h]}{2}\,s(1-s)\,W_2^2(\mu_1,\mu_0)
\end{split}
\end{equation}
for all~$h>0$, where $\mathsf B[h]=\dfrac{\mathsf A[0,h]^{-1}-1}{h}$.
In particular, if $\c(t)=e^{-Kt}$ then $\mathsf B[h]=-K+o(1)$ as $h\to0^+$, so that we immediately recover the displacement $K$-convexity~\eqref{intro_eq:K-displacement}.
For this reason, we may think of~\eqref{intro_eq:Ent_heated_convex} as a \emph{heated} version of the displacement convexity of the entropy and we call it the (dimension-free) \emph{weak Riemannian Curvature-Dimension condition} with respect to the curvature function~$\c$, $\RCD_w(\c,\infty)$ for short.

Inequality~\eqref{intro_eq:Ent_heated_convex_t=0} is very close to the distorted convexity inequality~\eqref{intro_eq:strange-displacement} but, at the same time, it reflects the idea behind the generalized Baudoin--Garofalo $\CD$~condition~\eqref{intro_eq:BG_CD}, in the sense that $\mathsf B[h]\to+\infty$ as~$h\to0^+$ in the non-$\CD$ setting, coherently with the fact that sub-Riemannian spaces have Ricci curvature unbounded from below.

The explosion of the right-hand side of~\eqref{intro_eq:Ent_heated_convex_t=0} as~$h\to0^+$ can be interpreted also in the light of the \emph{singularity problem} of $2$-Wasserstein geodesics, still open for a general sub-Riemannian manifold: if $\mu_0\ll\m$, then does any $2$-Wasserstein geodesic $s\mapsto\mu_s$, $s\in[0,1]$, joining $\mu_0,\mu_1\in\Prob_2(X)$ still satisfy $\mu_s\ll\m$?
This problem was posed for the first time by Ambrosio--Rigot~\cite{AR04} for the Heisenberg group and positively answered by Figalli--Juillet~\cite{FJ08}.
Later Figalli--Rifford~\cite{FR10} gave an affirmative answer also for more general sub-Riemannian manifolds (see also~\cite{BR19}).
As pointed out by Cavalletti--Mondino~\cite{CM17}, the answer is still positive if the ambient metric-measure space satisfies the $\MCP(K,N)$ condition and is \emph{essentially non-branching}, a condition roughly saying that \emph{branching} geodesics, i.e., geodesics splitting at intermediate times, are not too many. 
Note that some sub-Riemannian spaces do have branching geodesics, see~\cite{MR20}.
Thus, in this sense, considering the heated version of the $2$-Wasserstein geodesic in~\eqref{intro_eq:Ent_heated_convex} can be seen as a way to bypass its possible singularity.

All in all, apart from technicalities, if we can construct sufficiently good $2$-Wasserstein curves $s\mapsto\mu_s=f_s\m$ satisfying~\eqref{intro_eq:lisini}, then, under the linearity of the heat flow, for the metric-measure space $(X,\d,\m)$ we have the equivalences
\begin{equation*}
\BE_w(\c,\infty)
\iff
\EVI_w(\c)
\iff
\RCD_w(\c,\infty).
\end{equation*}
Therefore, in analogy with the $\CD$ setting, we may call such a metric-measure space $(X,\d,\m)$ a \emph{weak $\RCD$-space}.

Property~\eqref{intro_eq:lisini} can be obtained from a celebrated result of Lisini~\cite{L07}, so that the central problem we need to face for the construction of the curve $s\mapsto\mu_s$ is the absolute continuity property $\mu_s\ll\m$.
Due to the aforementioned singularity problem in this general framework, we cannot choose $s\mapsto\mu_s$ to be a geodesic.
Nevertheless, we may choose $s\mapsto\mu_s$ to be a suitable \emph{regularization} $s\mapsto\mu_s^\eps$, for all~$\eps>0$, of a geodesic (or of any other probability curve realizing the $2$-Wasserstein distance between $\mu_0$ and $\mu_1$ up to a smaller and smaller error).
However, since we need the Lisini inequality~\eqref{intro_eq:lisini}, the regularized curve $s\mapsto\mu_s^\eps$ has to have $2$-Wasserstein metric velocity controlled by the velocity of the original curve.

In~\cites{AGS13,EKS15}, the regularized procedure takes advantage of the smoothing property of the heat flow and, precisely, leads to the choice $\mu_s^\eps=\H_\eps\mu_s$.
Indeed, on the one hand, the pointwise gradient bound~\eqref{intro_eq:P_t_contractivity} implies the \emph{instantaneous diffusion property} $\H_t\mu\ll\m$ for all $t>0$ and $\mu\in\Prob_2(X)$.
On the other hand, the $W_2$-contractivity property~\eqref{intro_eq:H_t_contractivity} immediately gives $|\dot\mu_s^\eps|\le e^{-K\eps}\,|\mu_s|$. 

Under the weaker $\BE(\c,\infty)$ condition~\eqref{intro_eq:wBE}, it is still possible to prove the instantaneous diffusion property. 
However, for the choice $\mu_s^\eps=\H_\eps\mu_s$, the $W_2$-contractivity property~\eqref{intro_eq:w_Kuwada} only gives the weaker estimate $|\dot\mu_s^\eps|\le\c(\eps)\,|\mu_s|$, which is of no use unless $\lim\limits_{t\to0^+}\c(t)=1$, a property the curvature function does not satisfy for the sub-Riemannian manifolds under consideration.

Since we cannot rely on the sole properties of the heat flow, it is at this point that we assume that the ambient space~$X$ has a group operation left-compatible with the metric-measure structure and exploit the property of convolution.
In fact, under this additional assumption, we may choose the regularized curve as the left-convoluted curve $\mu_s^\eps=(\rho_\eps\star\mu_s)\m$, where $(\rho_\eps)_{\eps>0}\subset\Leb^1(X,\m)$ is a suitable family of convolution kernels.
Since the group operation is left-compatible with the metric, it is not difficult to prove that the $2$-Wasserstein metric velocity of the left-convoluted curve does not increase, i.e.\ $|\dot\mu_s^\eps|\le|\dot\mu_s|$ (a property not expected for right-convoluted curves, see \cref{rem:right_convolution} below).
We can thus perform all the above computations on the left-convoluted curve $s\mapsto\mu_s^\eps$ and obtain the desired entropic inequalities by passing to the limit as~$\eps\to0^+$ at the end.

Although the present work is focused only on (metric-measure)  groups, we believe that our results may be valid also for other non-group spaces, such as the Grushin plane~\cite{W14} and metric graphs~\cite{BK19}, where the regularization of $2$-Wasserstein curves could possibly be performed by exploiting the particular structure of the underlying space. 
Another interesting problem is whether some sub-Riemmanian manifolds (possibly, with a group structure) may satisfy a more precise form of the pointwise gradient bound~\eqref{intro_eq:P_t_contractivity_weak} also taking into account a dimensional parameter $N\in(0,+\infty)$.
Finally, from a purely metric-measure point of view, on the one side we do not know if the weak $\RCD$~condition~\eqref{intro_eq:Ent_heated_convex} may imply (a weaker form of) the $\MCP$ condition and, on the other side, if the weak~$\EVI$~\eqref{intro_eq:EVI_eta_simple} may be useful for the definition of a weaker notion of metric gradient flow and/or for proving stability properties of weak $\RCD$ spaces under Gromov--Hausdorff convergence.
We will hopefully come again over these and related topics in a future work.

\subsection{Organization of the paper}

The paper is organized as follows.

In \cref{sec:preliminaries} we recall all the known definitions and results in the metric-measure setting we will use in the sequel, in order to keep the paper the most self-contained as possible. 
Almost all the theorems are stated without proofs, but we give the reader precise references to the existing literature where to find the needed technical details. 

In \cref{sec:wBE} we introduce the $\BE_w(\c,\infty)$ condition in the metric-measure framework and study its consequences, such as Poincaré inequalities and the definition and the properties of the dual and the pointwise version of the heat semigroup.
The main result of this part is the equivalence between the $\BE_w(\c,\infty)$ condition and the weak $W_2$-contractivity property of the dual heat flow, the so-called Kuwada duality.

In \cref{sec:Fisher_and_LlogL} we deal with the Fisher information and the $\Leb\log\Leb$-regularization property.
The results, which will be frequently used in the remaining part of the paper, are technical and provide a generalization of the known theory to the context of the $\BE_w(\c,\infty)$ condition.

Finally, in \cref{sec:proof_equivalence}, we prove our main equivalence result for metric-measure groups, see \cref{res:th_equivalence} for the precise statement.
In the last part of this section, we show how this theorem applies to Carnot groups and to the $\SU(2)$~group.
Also, we briefly compare the heated displacement convexity~\eqref{intro_eq:Ent_heated_convex_t=0} we obtain with the distorted displacement convexity~\eqref{intro_eq:strange-displacement} in the case of $2$-Wasserstein geodesics induced by right-translation optimal transport maps.

\section{Preliminaries}
\label{sec:preliminaries}

In this section, we recall the main technical tools we will use throughout the paper. For a more detailed exposition of the results presented below, we refer the reader to~\cites{ACDiM15,AG13,AGMR15,AGS08,AGS13,AGS14,AGS14-2,AGS15,AMS15,B94,BGL14,B73,C99,G14,G15,V03,V09}. At the end of this section, we summarize the main assumptions we will use in the rest of the present work. For the reader's convenience, we will try to keep the paper the most self-contained as possible.

\subsection{AC curves}

Let $(X,\d)$ be a metric space.
Let $I\subset\R$ be a closed interval and let $p\in[1,+\infty]$. We say that a curve $\gamma\colon I\to X$ belongs to $\AC^p(I;X)$ if 
\begin{equation}\label{eq:def_AC_curve}
\d(\gamma_s,\gamma_t)\le\int_s^t g(r)\di r 
\qquad
s,t\in I,\ s<t, 
\end{equation} 
for some $g\in\Leb^p(I)$. The space $\AC^p_{\loc}(I;X)$ is defined analogously. The exponent \mbox{$p=1$} corresponds to \emph{absolutely continuous curves} and is simply denoted by $\AC(I;X)$. It turns out that, if $\gamma\in \AC^p(I;X)$, then there is a minimal function $g\in\Leb^p(I)$ satisfying~\eqref{eq:def_AC_curve}, called \emph{metric derivative} of the curve~$\gamma$, which is given by
\begin{equation*}
|\dot{\gamma}_t|=\lim_{s\to t}\frac{\d(\gamma_s,\gamma_t)}{|s-t|}
\qquad
\text{for $\leb^1$-a.e.}\ t\in I,
\end{equation*} 
see~\cite{AGS08}*{Theorem~1.1.2} for the simple proof. We thus say that an absolutely continuous curve~$\gamma$ has \emph{constant speed} if $t\mapsto|\dot{\gamma}_t|$ is (equivalent to) a constant.

We say that $(X,\d)$ is a \emph{length} (metric) space if for all $x,y\in X$ we have
\begin{equation*}
\d(x,y)=\inf\set*{\int_0^1|\dot{\gamma}_t|\di t : \gamma\in\AC([0,1];X),\ \gamma_0=x,\ \gamma_1=y}.
\end{equation*} 
In addition, we call $(X,\d)$ a \emph{geodesic} metric space if for every $x,y\in X$ there exists a curve $\gamma\colon[0,1]\to X$ such that $\gamma_0=x$, $\gamma_1=y$ and
\begin{equation*}
\d(\gamma_s,\gamma_t)=|s-t|\,\d(\gamma_0,\gamma_1) 
\qquad
\forall s,t\in[0,1].
\end{equation*}
In this case, we say that the curve $\gamma\colon[0,1]\to X$ is a (\emph{constant}) \emph{unit-speed geodesic} and we write
$s\mapsto\gamma_s\in\mathrm{Geo}([0,1];X)$.

\subsection{Slopes}

Let $(X,\d)$ be a metric space.
Let $\overline{\R}=\R\cup\set{-\infty,+\infty}$ and let $f\colon X\to\overline{\R}$ be a function. We define the \emph{effective domain} of~$f$ as
\begin{equation*}
\Dom(f)=\set*{x\in X : f(x)\in\R}.
\end{equation*} 
Given $x\in\Dom(f)$, we define the \emph{slope} and the \emph{asymptotic Lipschitz constant} of~$f$ at~$x$ by
\begin{equation}\label{eq:def_loc_Lip_constant}
|\D f|(x)=\limsup_{y\to x}\frac{|f(y)-f(x)|}{\d(x,y)},
\qquad
|\D^*f|(x)=\limsup_{\substack{y,z\to x\\ y\ne z}}\frac{|f(y)-f(z)|}{\d(y,z)}
\end{equation}
The \emph{descending slope} and the \emph{ascending slope} of~$f$ at~$x$ are respectively given by
\begin{equation*}
|\D^-f|(x)=\limsup_{y\to x}\frac{[f(y)-f(x)]^-}{\d(x,y)},
\qquad
|\D^+f|(x)=\limsup_{y\to x}\frac{[f(y)-f(x)]^+}{\d(x,y)}.
\end{equation*}
Here and in the following, $a^+$ and $a^-$ denote the positive and negative part of $a\in\R$ respectively. 
When $x\in\Dom(f)$ is an isolated point of $X$, we set $|\D f|(x)=|\D^*f|(x)=|\D^-f|(x)=|\D^+f|(x)=0$. By convention, we set $|\D f|(x)=|\D^*f|(x)=|\D^-f|(x)=|\D^+f|(x)=+\infty$ for all $x\in X\setminus\Dom(f)$.
Clearly, $|\D f|\le|\D^*f|$ on~$X$ and the asymptotic Lipschitz constant $|\D^* f|\colon X\to[0,+\infty]$ is an upper semicontinuous function.
Note that the slopes of a Borel function $f\colon X\to\overline{\R}$ are \emph{universally} measurable, see~\cite{AGS14}*{Lemma~2.6}.

According to~\cite{C99} (see also~\cite{AGS14}*{Section~2.3}), we say that a function $g\colon X\to[0,+\infty]$ is an \emph{upper gradient} of $f\colon X\to\overline{\R}$ if, for any curve $\gamma\in\AC([0,1];(\Dom(f),\d))$, $s\mapsto g(\gamma_s)\,|\dot{\gamma}_s|$ is measurable in $[0,1]$ (with the convention $0\cdot\infty=0$) and
\begin{equation}\label{eq:upper_gradient_def}
\abs*{f(\gamma_1)-f(\gamma_0)}\le\int_0^1 g(\gamma_s)\,|\dot{\gamma}_s|\di s.
\end{equation}
If $f\in\Lip(X)$, then $|\D f|$, $|\D^*f|$, $|\D^-f|$ and $|\D^+f|$ are upper gradients of~$f$, see~\cite{AGS14}*{Remark~2.8}. In addition, if $(X,\d)$ is a length space, then
\begin{equation*}
|\D^*f|(x)=\limsup_{y\to x}|\D f|(y),
\qquad
\Lip(f)=\sup_{x\in X}|\D f|(x)=\sup_{x\in X}|\D^*f|(x),
\end{equation*} 
see~\cite{AGS15}*{Section~3.1}. In particular, $|\D^*f|$ is the upper semicontinuous envelope of $|\D f|$.

\subsection{Hopf--Lax semigroup}

Let $(X,\d)$ be a length space.
For all $s>0$, the \emph{Hopf--Lax semigroup} $Q_s\colon\Cont_b(X)\to\Cont_b(X)$ is given by
\begin{equation}
\label{eq:hopf-lax_def}
Q_sf(x)
=
\inf_{y\in X} f(y)+\frac{\d^2(y,x)}{2s}
\quad 
\text{for all $x\in X$ and $f\in\Cont_b(X)$}.
\end{equation}
By convention, we set $Q_0f=f$ for all $f\in\Cont_b(X)$.
If $f\in\Cont_b(X)$, then
\begin{equation}\label{eq:hopf-lax_HJ}
\frac{\di^{\,+}}{\di s}\,Q_sf(x)+\frac12\,|\D Q_sf|^2(x)=0
\end{equation}
for all $s>0$ and $x\in X$, see~\cite{AGS14}*{Theorem~3.6}. 
If $f\in\Lip_b(X)$, then we also have
\begin{equation}\label{eq:hopf-lax_lip}
\Lip(Q_s f)\le 2\Lip(f)
\quad\text{and}\quad
\Lip(Q_{\cdot}f(x))\le2\Lip(f)^2
\end{equation}
for all $s\ge0$ and $x\in X$, see the discussion in~\cite{AGS15}*{Section~3.1}. 
In addition, by~\cite{AGS14}*{Propositon~3.2 and Theorem~3.6}, for all $s>0$ the slope $x\mapsto|\D Q_s f|(x)$ is upper semicontinuous, so that 
\begin{equation}\label{eq:hopf-lax_slope_is_usc}
|\D Q_s f|(x)=|\D^* Q_s f|(x)
\end{equation}
for all $s>0$ and $x\in X$.

\subsection{Wasserstein space}

Let $(X,d)$ be a complete and separable (\emph{Polish}, for short) metric space. We denote by $\Prob(X)$ the set of probability Borel measures on~$X$. Given $p\in[1,+\infty)$, the \emph{$p$-Wasserstein (extended) distance} between $\mu,\nu\in\Prob(X)$ is given by
\begin{equation}\label{eq:def_W_2}
W_p^p(\mu,\nu)
=
\inf\set*{\int_{X\times X} \d^p(x,y)\di\pi : \pi\in\mathsf{Plan}(\mu,\nu)}
\in[0,+\infty],
\end{equation}
where
\begin{equation}\label{eq:def_coupling}
\mathsf{Plan}(\mu,\nu)=\set*{\pi\in\Prob(X\times X) : (p_1)_\sharp\pi=\mu,\ (p_2)_\sharp\pi=\nu}.
\end{equation}
Here $p_i\colon X\times X\to X$, $i=1,2$, denote the the canonical projections on the components. 
As usual, if $\mu\in\Prob(X)$ and $T\colon X\to Y$ is a $\mu$-measurable map with values in the topological space~$Y$, the \emph{push-forward measure} $T_\sharp(\mu)\in\Prob(Y)$ is defined by $T_\sharp(\mu)(B)=\mu(T^{-1}(B))$ for every Borel set $B\subset Y$. 
The set $\mathsf{Plan}(\mu,\nu)$ introduced in~\eqref{eq:def_coupling} is call the set of \emph{admissible plans} or \emph{couplings} for the pair~$(\mu,\nu)$. Since the metric space $(X,\d)$ is complete and separable, there exist optimal couplings where the infimum in~\eqref{eq:def_W_2} is achieved.

The function~$W_p$ is a finite distance on the so-called \emph{$p$-Wasserstein space} $(\Prob_p(X),W_p)$, where
\begin{equation*}
\Prob_p(X)=\set*{\mu\in\Prob(X) : \int_X \d^p(x,x_0)\di\mu(x)<+\infty\ 
\text{for some, and thus any,}\ x_0\in X}.
\end{equation*}
The space $(\Prob_p(X),W_p)$ is complete and separable. 
If $(X,\d)$ is geodesic, then $(\Prob_p(X),W_p)$ is geodesic as well. Moreover, $\mu_n\xrightarrow{W_p}\mu$ as $n\to+\infty$ if and only if $\mu_n\weakto\mu$ as $n\to+\infty$ and
\begin{equation*}
\lim_{n\to+\infty}
\int_X \d^p(x,x_0)\di\mu_n(x)
=
\int_X \d^p(x,x_0)\di\mu(x) 
\quad
\text{for some}\ x_0\in X.
\end{equation*}
As usual, we write $\mu_n\weakto\mu$ as $n\to+\infty$, and we say that \emph{$\mu_n$ weakly converges to~$\mu$} as $n\to+\infty$, if we have 
\begin{equation*}
\lim_{n\to+\infty}
\int_X\phi\di\mu_n
=
\int_X\phi\di\mu
\quad
\text{for all}\ 
\phi\in\Cont_b(X).
\end{equation*}

Given $p\in[1,+\infty)$, the $p$-Wasserstein distance can be equivalently obtained via the \emph{Kantorovich duality formula}
\begin{equation}\label{eq:kantorovich_duality}
\frac1p\,W_p^p(\mu,\nu)
=
\sup\set*{
\int_X\phi^c\di\mu-\int_X\phi\di\nu
:
\phi\in\Lip_b(X)
}
\in[0,+\infty]
\end{equation}
for all $\mu,\nu\in\Prob(X)$, 
where 
\begin{equation}\label{eq:def_c_conjugate}
\phi^c(x)
=
\inf_{y\in X}
\phi(y)+\frac{\d^p(y,x)}p,
\quad
\text{for all}\ x\in X, 
\end{equation}
is the \emph{$c$-conjugate} of $\phi\in\Lip_b(X)$ with respect to the cost function $c=\d^p/p$.  
In particular, if $p=1$ then~\eqref{eq:def_c_conjugate} immediately gives $\phi^c=\phi$ and thus we can rewrite~\eqref{eq:kantorovich_duality} as
\begin{equation}\label{eq:kantorovich_duality_p=1}
W_1(\mu,\nu)
=
\sup\set*{
\int_X\phi\di\,(\mu-\nu)
:
\phi\in\Lip(X)\ 
\text{with $\Lip(\phi)\le1$ and bounded support}
}
\end{equation} 
for all $\mu,\nu\in\Prob_1(X)$, the so-called \emph{Kantorovich--Rubinstein formula}, see~\cite{V09}*{Particular Case~5.16}.
If instead $p=2$, then by~\eqref{eq:hopf-lax_def} we have $\phi^c=Q_1\phi$ and thus we can rewrite~\eqref{eq:kantorovich_duality} as
\begin{equation}\label{eq:kantorovich_duality_p=2}
\frac12\,W_2^2(\mu,\nu)
=
\sup\set*{
\int_XQ_1\phi\di\mu-\int_X\phi\di\nu
:
\phi\in\Lip(X)\ 
\text{with bounded support}
}
\end{equation} 
for all $\mu,\nu\in\Prob(X)$. 
Note that the integral expressions appearing in the right-hand sides of~\eqref{eq:kantorovich_duality_p=1} and~\eqref{eq:kantorovich_duality_p=2} are invariant by adding constants to~$\phi$, so that we can additionally assume $\phi\ge0$ without changing the suprema.

For an account on Kantorovich duality, we refer the reader to~\cite{V03}*{Section~1.1.2} and~\cite{V09}*{Theorem~5.10} (see also the discussion in~\cite{K10}*{Remark~3.6}).

Finally, given a non-negative Borel reference measure~$\m$ on~$X$, finite on bounded sets and such that $\supp(\m)=X$, for $p\in[1,+\infty)$ we let
\begin{equation*}
\Prob^{\rm ac}(X)=\set*{\mu\in\Prob(X) : \mu\ll\m},
\quad
\Prob_p^{\rm ac}(X)=\set*{\mu\in\Prob^{\rm ac}(X) : \mu\in\Prob_p(X)}.
\end{equation*}
Thanks to~\cite{V09}*{Theorem~6.18}, $\Prob_p^{\rm ac}(X)$ is a $W_p$-dense subset of~$\Prob_p(X)$.

For a proof of the above results and as well as for an agile introduction to the Wasserstein distance, we refer the reader to~\cite{AG13}*{Section~3}, \cite{V03}*{Chapter~1} and~\cite{V09}*{Chapters~4--6}.

\subsection{Entropy}

Let $(X,\d,\m)$ be a metric-measure space, i.e.\ $(X,\d)$ is a Polish metric space and 
\begin{equation}\label{eq:supp_m=X}
\text{
\begin{minipage}{0.55\textwidth}
\centering
$\m$ is a non-negative Borel-regular measure, finite on bounded sets and such that $\supp\m=X$.	
\end{minipage}
}
\end{equation}
Note that, in particular, $\m$ is a Radon measure on~$X$, see~\cite{HKST15}*{Proposition~3.3.44}.
In addition, assume that 
\begin{equation}\label{eq:exp_growth}
\exists x_0\in X\ 
\exists A,B>0\
\text{such that}\ 
\m\left(B_r(x_0)\right)\le A\exp(Br^2)\
\text{for all}\ r>0.
\end{equation}
The functional $\Ent_\m\colon\Prob_2(X)\to(-\infty,+\infty]$ given by
\begin{equation}\label{eq:def_entropy}
\Ent_\m(\mu)=
\begin{cases}
\displaystyle\int_X f\log f\di\m 
& \text{if}\ \mu=f\m\in\Prob_2(X),
\\[3mm]
+\infty 
& \text{otherwise},
\end{cases}
\end{equation}
is called the (\emph{relative}) \emph{entropy} of $\mu\in\Prob_2(X)$. According to our definition, $\mu\in\Dom(\Ent_\m)$ implies that $\mu\in\Prob_2(X)$ and that the effective domain $\Dom(\Ent_\m)$ is convex. 

As pointed out in~\cite{AGS14}*{Section~7.1}, the growth condition~\eqref{eq:exp_growth} guarantees that in fact $\Ent_\m(\mu)>-\infty$ for all $\mu\in\Prob_2(X)$, see~\cite{E10}*{Lemma~4.1}.
Hence, if $\mu=f\m\in\Prob_2(X)$ for some $f\in\Leb^p(X,\m)$ with $p\in(1,+\infty]$, then $f|\log f|\in\Leb^1(X,\m)$ and $\mu\in\Dom(\Ent_\m)$.

When $\m\in\Prob(X)$, the entropy functional~$\Ent_\m$ naturally extends to~$\Prob(X)$, is lower semicontinuous with respect to the weak convergence in~$\Prob(X)$ and positive by Jensen inequality. 
In addition, if $F\colon X\to Y$ is a Borel map, then
\begin{equation}\label{eq:ent_sharp_formula}
\Ent_{F_\sharp\m}(F_\sharp\mu)
\le
\Ent_\m(\mu)
\quad
\text{for all}\
\mu\in\Prob(X),
\end{equation}
with equality if $F$ is injective, see~\cite{AGS08}*{Lemma~9.4.5}.

When $\m(X)=+\infty$, if we set $\mathfrak{n}=e^{-c\,\d(\cdot,x_0)^2}\m$, where $x_0\in X$ is as in~\eqref{eq:exp_growth} and $c>0$ is chosen so that $\mathfrak{n}(X)<+\infty$ (the existence of such $c>0$ is ensured by~\eqref{eq:exp_growth}), then we obtain the useful formula
\begin{equation}\label{eq:ent_useful_formula}
\Ent_\m(\mu)=\Ent_{\mathfrak{n}}(\mu)-c\int_X\d^2(x,x_0)\di\mu
\quad
\text{for all}\ \mu\in\Prob_2(X).
\end{equation}
This shows that $\Ent_\m$ is lower semicontinuous in $(\Prob_2(X),W_2)$.

\subsection{Cheeger energy}

Let $(X,\d,\m)$ be a metric-measure space with $(X,\d)$ a Polish metric space and~$\m$ as in~\eqref{eq:supp_m=X}.
The functional $\Ch\colon\!\!\Leb^2(X,\m)\to[0,+\infty]$ given by
\begin{equation}\label{eq:def_Ch}
\Ch(f)=\inf\set*{\liminf_n\int_X |\D f_n|^2\di\m : f_n\to f\ \text{in}\ \Leb^2(X,\m),\ f_n\in\Lip(X)},
\end{equation}
for all $f\in\Leb^2(X,\m)$, is called \emph{Cheeger energy}. Here $|\D f|$ denotes the slope of $f\in\Lip(X)$ as defined in~\eqref{eq:def_loc_Lip_constant}. We let 
\begin{equation*}
\Sob^{1,2}(X,\d,\m)=\Dom(\Ch)=\set*{f\in\Leb^2(X,\m) : \Ch(f)<+\infty}
\end{equation*}
be the \emph{Sobolev space} naturally associated to~$\Ch$ endowed with the norm given by
\begin{equation}\label{eq:Sob_norm}
\|f\|_{\Sob^{1,2}(X,\d,\m)}^2
=
\|f\|_{\Leb^2(X,\m)}^2
+
2\,\Ch(f).
\end{equation}
The space $(\Sob^{1,2}(X,\d,\m),\|\cdot\|_{\Sob^{1,2}(X,\d,\m)})$ is a separable Banach space but can fail to be a Hilbert space in general, see~\cite{AGS14}*{Remark~4.6}.

\subsection{Minimal weak gradient}

Let $(X,\d,\m)$ be a metric-measure space with $(X,\d)$ a Polish metric space and~$\m$ as in~\eqref{eq:supp_m=X}. 
If $f\in\Leb^2(X,\m)$, then
\begin{equation}\label{eq:weak_gradients_collection}
{\sf Grad}(f)
=
\set*{
G\in\Leb^2(X,\m)
:
\bigg[
\begin{minipage}{0.44\textwidth}
$\exists f_n\in\Lip_b(X,\m)$
such that
$f_n\to f$ 
and 
$|\D f_n|\weakto G$
in
$\Leb^2(X,\m)$
as
$n\to+\infty$
\end{minipage}
}
\end{equation} 
is a convex set, possibly empty (see~\cite{AGS14}*{Definition~4.2} or~\cite{S14}*{Section~4.1}). 
If $f\in\Sob^{1,2}(X,\d,\m)$, then it is possible to show that ${\rm Grad}(f)\ne\emptyset$ and thus, by the reflexivity of $\Leb^2(X,\m)$, ${\rm Grad}(f)$ has a unique element of minimal $\Leb^2$-norm, the \emph{minimal weak (upper or relaxed) gradient} of $f$, $|\D f|_w\in\Leb^2(X,\m)$, that is also minimal with respect to the order structure, i.e.\
\begin{equation}\label{eq:min_weak_grad_min_property}
G\in{\rm Grad}(f)
\implies
|\D f|_w\le G\
\m\text{-a.e.\ in}\ X.
\end{equation}
Thanks to~\cite{AGS14}*{Theorems~6.2 and~6.3} (see also~\cite{AGS14}*{Remark~4.7}), if $f\in\Sob^{1,2}(X,\d,\m)$, then the minimal weak gradient $|\D f|_w\in\Leb^2(X,\m)$ provides an integral representation of the Cheeger energy, so that
\begin{equation*}
\Ch(f)
=
\frac12\int_X |\D f|_w^2\di\m
\quad
\text{for all}\
f\in\Sob^{1,2}(X,\d,\m).
\end{equation*}

The minimal weak gradient is a \emph{local} operator, i.e.\
\begin{equation}\label{eq:locality_min_weak_grad}
f,g\in\Sob^{1,2}(X,\d,\m)
\implies
|\D f|_w=|\D g|_w\ \text{$\m$-a.e.\ on}\ 
\set{f-g=c},
\end{equation}
for all $c\in\R$, obeys a \emph{Leibniz-rule estimate}, in the sense that if $f,g\in\Sob^{1,2}(X,\d,\m)\cap\Leb^\infty(X,\m)$ then $fg\in\Sob^{1,2}(X,\d,\m)$ with
\begin{equation}\label{eq:Leibniz_rule}
|\D(fg)|_w\le|f|\,|\D g|_w+|\D f|_w\,|g|,	
\end{equation} 
and satisfies the following \emph{chain rule}
\begin{equation}\label{eq:chain_rule_min_weak_grad}
f\in\Sob^{1,2}(X,\d,\m)
\implies
\phi(f)\in\Sob^{1,2}(X,\d,\m)\
\text{with}\
|\D\phi(f)|_w\le|\phi'(f)|\,|\D f|_w	
\end{equation}
for any Lipschitz function $\phi\colon I\to\R$ defined on an interval $I\subset\R$ containing the image of~$f$ (with $0\in I$ and $\phi(0)=0$ if $\m(X)=+\infty$), with 
$|\D\phi(f)|_w=\phi'(f)\,|\D f|_w$
if $\phi$ is non-decreasing. 
Also, if $f,g\in\Sob^{1,2}(X,\d,\m)$, then $f\wedge g,f\vee g\in\Sob^{1,2}(X,\d,\m)$ with
\begin{equation}\label{eq:min_is_sob}
|\D(f\wedge g)|_w
=
\begin{cases}
|\D f|_w & \text{$\m$-a.e.\ on}\ \set*{f\le g}\\
|\D g|_w & \text{$\m$-a.e.\ on}\ \set*{f\ge g}
\end{cases}
\end{equation}
and
\begin{equation}\label{eq:max_is_sob}
|\D(f\vee g)|_w
=
\begin{cases}
|\D f|_w & \text{$\m$-a.e.\ on}\ \set*{f\ge g}\\
|\D g|_w & \text{$\m$-a.e.\ on}\ \set*{f\le g},
\end{cases}
\end{equation}
see~\cite{AGS14}*{Lemma~2.5 and Proposition~4.8(e)} and their proofs.

In addition, by~\cite{AGS14}*{Theorem~6.3 and Lemma~4.3(c)}, bounded Lipschitz functions are dense in energy in $\Sob^{1,2}(X,\d,\m)$, i.e.\ 
\begin{equation}\label{eq:approx_Sob_by_Lip_b}
f\in\Sob^{1,2}(X,\d,\m)
\implies
\bigg[
\begin{minipage}{0.5\textwidth}
$\exists f_n\in\Lip_b(X)\cap\Leb^2(X,\m)$
such that
$f_n\to f$
and
$|\D f_n|\to|\D f|_w$ 
in
$\Leb^2(X,\m)$
as
$n\to+\infty$.
\end{minipage}
\end{equation}
In particular, we have that 
\begin{equation}\label{eq:wD_less_slope_for_Lip}
f\in\Lip_b(X,\d)\cap\Sob^{1,2}(X,\d,\m)
\implies
|\D f|_w\le|\D f|\ \text{$\m$-a.e.\ in $X$},
\end{equation} 
see also~\cite{AGS14}*{Remark~5.5}. As observed in~\cite{AGS13}*{Section~8.3}, the approximation~\eqref{eq:approx_Sob_by_Lip_b} can be enforced by replacing slopes with asymptotic Lipschitz constants (recall~\eqref{eq:def_loc_Lip_constant} for the definition), so that
\begin{equation}\label{eq:approx_asymp_Sob_by_Lip_b}
f\in\Sob^{1,2}(X,\d,\m)
\implies
\bigg[
\begin{minipage}{0.51\textwidth}
$\exists f_n\in\Lip_b(X)\cap\Leb^2(X,\m)$
such that
$f_n\to f$
and
$|\D^*f_n|\to|\D f|_w$ 
in
$\Leb^2(X,\m)$
as
$n\to+\infty$,
\end{minipage}
\end{equation}
see~\cite{G15}*{Theorem~2.8} and~\cite{ACDiM15}*{Section~4} for a more detailed discussion.

\subsection{Heat semigroup} 

Let $(X,\d,\m)$ be a metric-measure space with $(X,\d)$ a Polish metric space and~$\m$ as in~\eqref{eq:supp_m=X}.
By~\cite{AGS14}*{Theorem~4.5}, $\Ch$ is convex, lower semicontinuous and $2$-homogeneous. The effective domain of the Cheeger energy, which we denote by $\Sob^{1,2}(X,\d,\m)$, is dense in $\Leb^2(X,\m)$. Thus, by the Hilbertian theory of gradient flows, see~\cites{AGS08,B73} for the general theory and~\cite{AG13}*{Theorem~3.1} for a plain exposition of the main results, for each given $f\in\Leb^2(X,\m)$ there exists a curve
\begin{equation}\label{eq:def_semigroup_map}
t\mapsto f_t=\P_t f\in\AC_{\loc}((0,+\infty);\Leb^2(X,\m)),	
\end{equation}
called \emph{heat semigroup}, such that 
\begin{equation}\label{eq:def_HF}
\begin{cases}
\dfrac{\di}{\di t}\,f_t\in-\de^-\Ch(f_t)
& \text{for $\leb^1$-a.e.}\ t\in(0,+\infty)\\[4mm]
\lim\limits_{t\to0^+}f_t=f &\text{in $\Leb^2(X,\m)$}.
\end{cases}
\end{equation}   
Here and in the following, $\de^-\Ch(f)\subset\Leb^2(X,\m)$ denotes the subdifferential of~$\Ch$ at $f\in\Sob^{1,2}(X,\d,\m)$ and is defined by
\begin{equation*}
\ell\in\de^-\Ch(f)
\iff
\Ch(g)
\ge
\Ch(f)
+\int_X\ell\,(g-f)\di\m\
\text{for all}\ g\in\Leb^2(X,\m).
\end{equation*}

The heat flow~\eqref{eq:def_semigroup_map} is uniquely determined by~\eqref{eq:def_HF}, is $1$-homogeneous, i.e.\ 
\begin{equation*}
f\in\Leb^2(X,\m),\ \lambda\in\R
\implies
\P_t(\lambda f)=\lambda\P_t f\ 
\text{for all}\ t\ge0,
\end{equation*}
and defines a strongly continuous semigroup of contractions in $\Leb^2(X,\m)$, meaning that
\begin{equation}\label{eq:heat_L2_contraction}
\|f_t\|_{\Leb^2(X,\m)}
\le
\|f\|_{\Leb^2(X,\m)}\
\text{for all}\ t>0\ \text{and}\ f\in\Leb^2(X,\m).
\end{equation} 
By~\cite{AGS14}*{Theorem~4.16(a)}, the heat semigroup preserves one-side essential bounds (\emph{maximum principle}). Precisely, for $C\in\R$ it holds
\begin{equation}\label{eq:sub-Markov}
f\le C\
\text{(resp.\ $f\ge C$)}
\implies
f_t\le C\
\text{(resp.\ $f_t\ge C$) for all $t\ge0$}
\end{equation}
and, moreover,
\begin{equation}\label{eq:sub-Markov_func}
f\le g+C
\implies
f_t\le g_t+C\
\text{for all $t\ge0$},
\end{equation}
whenever $f,g\in\Leb^2(X,\m)$.
By~\cite{AGS14}*{Theorem~4.16(b)}, 
the heat semigroup satisfies the contraction property
\begin{equation}\label{eq:heat_Lp_contraction}
\|f_t-g_t\|_{\Leb^p(X,\m)}
\le
\|f-g\|_{\Leb^p(X,\m)}
\quad
\text{for all}\ f,g\in\Leb^2(X,\m)\cap\Leb^p(X,\m),	
\end{equation}
whenever $p\in[1,+\infty]$. 
Since $\Leb^p(X,\m)\cap\Leb^2(X,\m)$ is $\Leb^p$-dense in $\Leb^p(X,\m)$ for all $p\in[1,+\infty)$, we can uniquely extend the heat semigroup to a strongly continuous semigroup of contractions in $\Leb^p(X,\m)$, $p\in[1,+\infty)$, for which we retain the same notation. 
The heat semigroup can thus be extended to a weakly*-continuous semigroup of contractions in $\Leb^\infty(X,\m)$ by duality, i.e.\
\begin{equation}\label{eq:def_heat_L_infty}
\int_X\phi\,\P_tf\di\m
=
\int_X f\,\P_t\phi\di\m
\quad
\text{for every $f\in\Leb^1(X,\m)$ and $\phi\in\Leb^\infty(X,\m)$}.
\end{equation}
By~\cite{AGS14}*{Theorem~4.20}, thanks to~\eqref{eq:heat_Lp_contraction}, if $\m$  satisfies the growth condition~\eqref{eq:exp_growth}, then the heat semigroup satisfies the
\emph{mass preservation property}
\begin{equation}\label{eq:heat_mass_preservation}
\int_X f_t\di\m=\int_X f\di\m
\quad
\text{for all}\ t\ge0\ \text{and}\ f\in\Leb^1(X,\m).
\end{equation}

The heat semigroup is regularizing as stated in \cref{res:heat_flow_sob_cont} below. This result is well known to experts, but we quickly prove it here for the reader's convenience.

\begin{lemma}[Heat flow regularization]
\label{res:heat_flow_sob_cont}
Let $f\in\Leb^2(X,\m)$. 
Then 
\begin{equation}\label{eq:Ch_f_t_is_AC_loc}
t\mapsto\Ch(f_t)\in\AC_{\loc}((0,+\infty);[0,+\infty))	
\end{equation}
with
\begin{equation}\label{eq:Ch_f_t_inf_formula}
\Ch(f_t)
\le
\inf\set*{\Ch(g)+\frac1{2t}\int_X|f-g|^2\di\m : g\in\Sob^{1,2}(X,\d,\m)}
\end{equation}
for all $t>0$, and
\begin{equation}\label{eq:min_weak_grad_f_t_is_cont}
t\mapsto|\D f_t|_w\in\Cont((0,+\infty);\Leb^2(X,\m)).
\end{equation}
Moreover, if $f\in\Sob^{1,2}(X,\d,\m)$, then the continuity of the maps in~\eqref{eq:Ch_f_t_is_AC_loc} and~\eqref{eq:min_weak_grad_f_t_is_cont} extends to $t=0$.
\end{lemma}

\begin{proof}
We divide the proof in two steps.

\smallskip

\textit{Step~1: regularity of Cheeger energy along the heat flow}.
The $\AC_{\loc}$-regularity of the Chegeer energy along the heat flow in~\eqref{eq:Ch_f_t_is_AC_loc} and the $\inf$-formula in~\eqref{eq:Ch_f_t_inf_formula} follow from the theory of Hilbertian gradient flows, see~\cite{AG13}*{Theorem~3.1}.
If moreover $f\in\Sob^{1,2}(X,\d,\m)$, then $\Ch(f_t)\le\Ch(f)$ for all $t>0$ by~\eqref{eq:Ch_f_t_inf_formula}, so that $\Ch(f_t)\to\Ch(f)$ as $t\to0^+$ by the lower semicontinuity of the Cheeger energy.

\smallskip

\textit{Step~2: regularity of the minimal weak gradient along the heat flow}.
Fix $t>0$ and let $t_n>0$, $n\in\N$, be such that $t_n\to t$ as $n\to+\infty$.
By~\eqref{eq:Ch_f_t_is_AC_loc}, the sequence $(|\D f_{t_n}|_w)_{n\in\N}$ is bounded in $\Leb^2(X,\m)$.
We can thus find a subsequence $(|\D f_{t_{n_k}}|_w)_{k\in\N}$ and a function $G\in\Leb^2(X,\m)$ such that $|\D f_{t_k}|_w\weakto G$ in $\Leb^2(X,\m)$ as $k\to+\infty$. 
By the weak lower semicontinuity of the $\Leb^2$-norm and again by~\eqref{eq:Ch_f_t_is_AC_loc}, we must have $\|G\|_{\Leb^2(X,\m)}\le\||\D f_t|_w\|_{\Leb^2(X,\m)}$. 
By definition of minimal weak gradient and~\cite{AGS14}*{Lemma~4.3(b)}, we must also have that $|\D f_t|_w\le G$ $\m$-a.e.\ in~$X$.
Hence $G=|\D f_t|_w$ $\m$-a.e.\ in~$X$ and thus $|\D f_{t_{n_k}}|_w\to|\D f_t|_w$ in $\Leb^2(X,\m)$ as $k\to+\infty$ by the uniform convexity of $\Leb^2(X,\m)$ (see~\cite{B11}*{Proposition~3.32} for example). 
Hence~\eqref{eq:min_weak_grad_f_t_is_cont} readily follows.
If moreover $f\in\Sob^{1,2}(X,\d,\m)$, then we can take $t=0$ in the above argument and use the continuity in $t=0$ of the map in~\eqref{eq:Ch_f_t_is_AC_loc} to extend the $\Leb^2$-continuity of the minimal weak gradient along the heat flow to~$t=0$.
\end{proof}

\subsection{Metric-measure Laplacian}

Let $(X,\d,\m)$ be a metric-measure space with $(X,\d)$ a Polish metric space and~$\m$ as in~\eqref{eq:supp_m=X}.
If $f\in\Leb^2(X,\m)$ and $\de^-\Ch(f)\neq\emptyset$, then the element of minimal $\Leb^2$-norm in~$-\de^-\Ch(f)$ is called the (\emph{metric-measure}) \emph{Laplacian} of the function~$f$ and is denoted by $\Delta_{\d,\m}f$, see~\cite{AGS14}*{Definition~4.13}. 
The effective domain $\Dom(\Delta_{\d,\m})$ of the Laplacian is a $\Leb^2$-dense subset of $\Sob^{1,2}(X,\d,\m)$ (and thus, in particular, a $\Leb^2$-dense subset of $\Leb^2(X,\m)$), see~\cite{B73}*{Proposition~2.11}. Note that the operator $\Delta_{\d,\m}$ is not linear in general, but is $1$-homogeneous, in the sense that
\begin{equation*}
f\in\Dom(\Delta_{\d,\m}),\ \lambda\in\R
\implies
\lambda f\in\Dom(\Delta_{\d,\m})\ 
\text{with}\
\Delta_{\d,\m}(\lambda f)=\lambda\Delta_{\d,\m}f.
\end{equation*}

By the regularizing properties of gradient flows in Hilbert spaces (see~\cite{AG13}*{Theorem~3.1}), for every $t>0$ the right time-derivative $\frac{\di^{\,+}}{\di t}f_t$ exists and it is actually the element with minimal $\Leb^2$-norm in $-\de^-\Ch(f)$, so that $f_t\in\Dom(\Delta_{\d,\m})$ for all $t>0$ and we can rewrite~\eqref{eq:def_HF} as
\begin{equation*}
\begin{cases}
\dfrac{\di^{\,+}}{\di t}\,f_t=\Delta_{\d,\m}f_t 
& \text{for every}\ t\in(0,+\infty)\\[4mm]
\lim\limits_{t\to0^+}f_t=f &\text{in $\Leb^2(X,\m)$}
\end{cases}
\end{equation*}  
whenever $f\in\Leb^2(X,\m)$ is given. Moreover, by the integration-by-part formula provided by~\cite{AGS14}*{Proposition~4.15}, it holds that
\begin{equation*}
\frac{\di^{\,+}}{\di t}\,\Ch(f_t)
=
-\|\Delta_{\d,\m}f_t\|_{L^2(X,\m)}^2\
\text{for all}\ t>0,
\end{equation*}
whenever $f\in\Leb^2(x,\m)$.

\subsection{Quadratic Cheeger energy}

Let $(X,\d,\m)$ be a metric-measure space with $(X,\d)$ a Polish metric space and~$\m$ as in~\eqref{eq:supp_m=X} and~\eqref{eq:exp_growth}.
As in~\cite{AGS14-2}*{Section4.3}, we say that the Cheeger energy is \emph{quadratic} if it satisfies the \emph{parallelogram identity}
\begin{equation}\label{eq:quadratic_Ch}
\Ch(f+g)+\Ch(f-g)=2\Ch(f)+2\Ch(g)
\quad
\text{for all}\ f,g\in\Sob^{1,2}(X,\d,\m).
\end{equation}
In this case, $\Ch$ is a quadratic form on $\Sob^{1,2}(X,\d,\m)$, the functional $\E\colon\Sob^{1,2}(X,\d,\m)\to[0,+\infty)$ defined by the formula 
\begin{equation}\label{eq:bil_Dir_form_E}
\E(f,g)=\Ch(f+g)-\Ch(f)-\Ch(g)
\quad
\text{for all}\ f,g\in\Sob^{1,2}(X,\d,\m)
\end{equation}
is a symmetric bilinear form on $\Sob^{1,2}(X,\d,\m)$ and $(\Sob^{1,2}(X,\d,\m),\|\cdot\|_{\Sob^{1,2}(X,\d,\m)})$ is a Hilbert space, see~\cite{G15}*{Proposition~4.22}. 
In particular, thanks to~\eqref{eq:approx_Sob_by_Lip_b}, the set $\Lip_b(X)\cap\Sob^{1,2}(X,\d,\m)$ is $\Sob^{1,2}$-dense in $\Sob^{1,2}(X,\d,\m)$, see~\cite{G15}*{Corollary~2.9}.

For simplicity, we set $\E(f)=\E(f,f)$ for all $f\in\Sob^{1,2}(X,\d,\m)$.   
The chain rule~\eqref{eq:chain_rule_min_weak_grad} for the minimal weak gradient proves that $\E$ is \emph{Markovian}, i.e.\ 
\begin{equation*}
\E(\phi\circ f)\le\E(f)\
\text{for all}\ f\in\Sob^{1,2}(X,\m),
\end{equation*}
whenever $\phi\colon\R\to\R$ is $1$-Lipschitz with $\phi(0)=0$. Since $\Ch$ is lower semicontinuous, the form $\E$ is also closed. Thus, thanks to the density of $\Sob^{1,2}(X,\d,\m)$ in $\Leb^2(X,\m)$, we can extend the form $\E$ given in~\eqref{eq:bil_Dir_form_E} to a symmetric bilinear form on $\Leb^2(X,\m)$, for which we retain the same notation. 
By the locality property of the minimal weak gradient~\eqref{eq:locality_min_weak_grad}, the form $\E$ is \emph{strongly local}, meaning that
\begin{equation*}
f,g\in\Sob^{1,2}(X,\d,\m),\ 
(f+c)g=0\ \text{$\m$-a.e.\ in $X$}
\implies
\E(f,g)=0.
\end{equation*}

By~\cite{G15}*{Proposition~4.24}, the Laplacian $\Delta_{\d,\m}$ coincides with the generator of $\E$ and hence satisfies the \emph{integration-by-part formula}
\begin{equation}\label{eq:int_by_part_formula}
\E(f,g)=-\int_X g\,\Delta_{\d,\m}f\di\m
\quad
\text{for all}\
f\in\Dom(\Delta_{\d,\m}),\
g\in\Sob^{1,2}(X,\d,\m).	
\end{equation}
Thus, since $\E$ is symmetric, the Laplacian $\Delta_{\d,\m}$ is a self-adjoint operator in $\Leb^2(X,\m)$.
In addition, by~\cite{G15}*{Proposition~4.23}, the Laplacian is a linear operator.

Therefore, if the Cheeger energy is quadratic, the heat semigroup $(\P_t)_{t\ge0}$ is a linear \emph{analytic Markov semigroup} in $\Leb^2(X,\m)$, is a self-adjoint operator in $\Leb^2(X,\m)$ and the map~\eqref{eq:def_semigroup_map} is the unique $\Cont^1$ map with values in $\Dom(\Delta_{\d,\m})$ satisfying 
\begin{equation*}
\begin{cases}
\dfrac{\di}{\di t}\,f_t=\Delta_{\d,\m}f_t & \text{for}\ t\in(0,+\infty),\\[3mm]
\lim\limits_{t\to 0^+} f_t=f & \text{in}\ \Leb^2(X,\m).
\end{cases}
\end{equation*}
Because of this, $\Delta_{\d,\m}$ can be equivalently characterized in terms of the strong convergence 
\begin{equation}\label{eq:Laplacian_char_strong_conv}
\frac{\P_tf-f}{t}\to\Delta_{\d,\m}f\quad 
\text{in}\ \Leb^2(X,\m)\
\text{as}\
t\to0^+.
\end{equation}

By~\cite{G15}*{Proposition~4.21} (see also~\cite{AGS14-2}*{Theorem~4.18}), if the Cheeger energy is quadratic, then the parallelogram identity~\eqref{eq:quadratic_Ch} can be localized at the level of minimal weak gradient, in the sense that 
\begin{equation*}
|\D(f+g)|_w^2+|\D(f-g)|_w^2=2|\D f|_w^2+2|\D g|_w^2\quad
\text{$\m$-a.e.\ in~$X$}
\end{equation*} 
for all $f,g\in\Sob^{1,2}(X,\d,\m)$.
Thus, the naturally associated $\Gamma$ operator, given by
\begin{equation*}
\Gamma(f,g)=|\D(f+g)|_w^2-|\D f|_w^2-|\D g|_w^2
\quad 
\text{for all}\ f,g\in\Sob^{1,2}(X,\d,\m),
\end{equation*}
defines a strongly-continuous, symmetric and bilinear map from $\Sob^{1,2}(X,\d,\m)$ to $\Leb^1(X,\m)$ which 
represents the form $\E$, i.e.\
\begin{equation*}
\E(f,g)
=
\int_X\Gamma(f,g)\di\m
\quad
\text{for all}\ f,g\in\Sob^{1,2}(X,\d,\m).
\end{equation*}
The $\Gamma$ operator satisfies the pointwise estimate 
\begin{equation}
\label{eq:Gamma_ineq_prod_min_weak_grads}
\Gamma(f,g)\le|\D f|_w\,|\D g|_w
\quad
\text{for all}\
f,g\in\Sob^{1,2}(X,\d,\m),
\end{equation}
the chain rule
\begin{equation}\label{eq:Gamma_chain_rule}
\Gamma(\phi(f),g)=\phi'(f)\,\Gamma(f,g)	
\quad
\text{for all}\
f,g\in\Sob^{1,2}(X,\d,\m),
\end{equation}
whenever $\phi\in\Lip(\R)$ with $\phi(0)=0$, and the Leibiniz rule
\begin{equation*}
\Gamma(fg,h)
=
g\,\Gamma(f,h)
+
f\,\Gamma(g,h)
\quad
\text{for all}\ 
f,g,h\in\Sob^{1,2}(X,\d,\m),\ f,g\in\Leb^\infty(x,\m),
\end{equation*}
see the discussion in~\cite{G15}*{Chapters~3 and~4}.
The operator 
\begin{equation*}
\Gamma(f)=|\D f|_w^2,
\quad
\text{defined for}\ f\in\Sob^{1,2}(X,\d,\m),
\end{equation*}
is therefore the \emph{carré du champ} associated to~$\E$ and obeys the rules of \emph{$\Gamma$-Calculus}. 
With these notations, the Laplacian satisfies the following \emph{chain rule}, see~\cite{G15}*{Proposition~4.28}: if 
$f\in\Dom(\Delta_{\d,\m})\cap\Sob^{1,2}(X,\d,\m)$ 
and 
$\phi\in\Cont^2(\R)$ with $\phi(0)=0$, 
then 
$\phi(f)\in\Dom(\Delta_{\d,\m})$
with
\begin{equation}\label{eq:chain_rule_Laplacian}
\Delta_{\d,\m}(\phi\circ f)
=
\phi'(f)\,\Delta_{\d,\m}f
+
\phi''(f)\,\Gamma(f).	
\end{equation}

For an account on $\Gamma$-Calculus in the present and related frameworks, we refer the reader to~\cites{AGS15,AMS15,B94,G14,G15} and to the monograph~\cite{BGL14}.

\subsection{Main assumptions and length property}
\label{subsec:main_ass}

We conclude this section summarizing the main assumptions we are going to use throughout this paper. We assume that $(X,\d,\m)$ is a metric-measure space satisfying the following properties:
\begin{enumerate}[label=(P.\arabic*)]
\item\label{assumption:Polish} $(X,\d)$ is a complete and separable metric space;
\item\label{assumption:measure} $\m$ is a non-negative Borel-regular measure on~$X$, finite on bounded sets and such that $\supp\m=X$;
\item\label{assumption:exp_growth} there exist $x_0\in X$ and $A,B>0$ such that $\m\left(B_r(x_0)\right)\le A\exp(Br^2)$ for all $r>0$;
\item\label{assumption:quadratic_Ch} the Cheeger energy~$\Ch$ is quadratic, i.e.\ $\Ch(f+g)+\Ch(f-g)=2\Ch(f)+2\Ch(g)$ for all $f,g\in\Sob^{1,2}(X,\d,\m)$;
\item\label{assumption:Lip_Rademacher} 
if $L\in[0,+\infty)$ and $f\in\Sob^{1,2}(X,\d,\m)$ satisfies $|\D f|_w\le L$ $\m$-a.e.\ in~$X$, then $f=\tilde{f}$ $\m$-a.e.\ in~$X$ for some $\tilde{f}\in\Lip(X)$ with $\Lip(\tilde{f})\le L$. 
\end{enumerate}
We say that a metric-measure space $(X,\d,\m)$ is \emph{admissible} if it satisfies the properties \ref{assumption:Polish}-\ref{assumption:Lip_Rademacher} listed above.

Let us briefly comment on these assumptions.
As we have already seen, assumptions \ref{assumption:Polish}--\ref{assumption:quadratic_Ch} ensure that $(X,\d,\m)$ satisfies all the properties we have recalled in this section.
The additional assumption \ref{assumption:Lip_Rademacher}, instead, allows to identify the metric-measure structure of $(X,\d,\m)$ with the energetic-measure structure of $(X,\mathscr B,\E,\m)$, where $\mathcal B$ is the Borel $\sigma$-algebra generate by the topology of open $\d$-balls, thus making the metric-measure and the energetic-measure approaches equivalent.
More precisely, once the Dirichlet--Cheeger energy $\E$ is available, it is relevant to check if the function
\begin{equation*}
\d_\E(x,y)=\sup\set*{|f(x)-f(y)| : f\in\Sob^{1,2}(X,\d,\m)\cap\Cont(X)\ \text{with}\ \Gamma(f)\le1\ \text{$\m$-a.e.}},
\end{equation*}
defined for all $x,y\in X$, actually coincides with the starting distance function~$\d$. 
The function $\d_\E$ is known as the Biroli--Mosco distance, see~\cite{BM95} and also~\cites{CKS87,S10,S95,S98}.
From~\eqref{eq:wD_less_slope_for_Lip}, we immediately get that~$\d\le\d_\E$.
The opposite inequality follows if any function $f\in\Sob^{1,2}(X,\d,\m)\cap\Cont(X)$ with $\Gamma(f)\le1$ $\m$-a.e.\ in~$X$ is $1$-Lipschitz, which is precisely~\ref{assumption:Lip_Rademacher}. 

In $\RCD(K,\infty)$ spaces, assumption~\ref{assumption:Lip_Rademacher} is an important consequence of the $\BE(K,\infty)$ condition. 
Namely, since $\Gamma(\P_t f)\le e^{-Kt}\,\P_t\Gamma(f)$ for all $f\in\Sob^{1,2}(X,\d,\m)$, the $1$-Lipschitz regularity of $f\in\Sob^{1,2}(X,\d,\m)\cap\Cont(X)$ can be obtained passing to the limit as $t\to0^+$ from the $e^{-Kt}$-Lipschitz regularity of the (more regular) function~$\P_tf$, see~\cite{AGS15}*{Remark~3.8} and the proof of (v)$\,\Rightarrow\,$(ii) in~\cite{AGS15}*{Theorem~3.17}.  

In the more general situation in which the function $t\mapsto e^{-Kt}$ is replaced by a function $t\mapsto\c(t)$ such that $\lim\limits_{t\to0^+}\c(t)>1$ (as it happens in sub-Riemannian spaces), then property~\ref{assumption:Lip_Rademacher} cannot be inferred from the contractivity property of the heat semigroup.

One useful consequence of the identification $\d_\E=\d$ provided by~\ref{assumption:Lip_Rademacher} that will be employed several times in the sequel is the length property of the metric space $(X,\d)$, see~\cite{AGS15}*{Theorems~3.10} (and also~\cite{S10}).

\begin{proposition}[Length property]\label{res:length_space}
If $(X,\d,\m)$ satisfies properties~\ref{assumption:Polish}, \ref{assumption:measure}, \ref{assumption:quadratic_Ch} and~\ref{assumption:Lip_Rademacher} (and thus, in particular, if it is admissible),
then $(X,\d)$ is a length space.
\end{proposition}

\section{Weak Bakry--\'{E}mery curvature condition and Kuwada duality}
\label{sec:wBE}

In this section, we introduce and study a generalization of the \emph{Bakry--\'{E}mery curvature condition} for Sobolev functions and its equivalence with the Wasserstein contraction property of the dual heat semigroup.
The presentation of the results will be close in spirit to that of~\cites{AGS15,K10,K13}. 
For the reader's ease, we adopt the notation of~\cite{AGS15}. 

If not otherwise stated, from now on we assume that $(X,\d,\m)$ is an admissible metric-measure space as in \cref{subsec:main_ass}.

\subsection{Semigroup mollification}

We begin this section by recalling an useful technical tool that we will use in the following.
Let $\kappa\in\Cont^\infty_c((0,+\infty))$ be such that
\begin{equation}
\label{eq:kappa_kernel}
\kappa\ge0
\quad
\text{and}
\quad
\int_0^{+\infty}\kappa(r)\di r=1.
\end{equation}
Let $p\in[1,+\infty]$. For every $f\in\Leb^p(X,\m)$, let us set
\begin{equation}
\label{eq:semigroup_moll_def}
\mathfrak h^\eps f
=
\frac1\eps
\int_0^{+\infty}\P_rf\,\kappa\left(\tfrac r\eps\right)\di r,
\quad
\text{for all}\ \eps>0,	
\end{equation}
be the \emph{semigroup mollification operator}, where the integral is intended in the Bochner sense if $p<+\infty$ and by duality with any function in $\Leb^1(X,\m)$ if $p=+\infty$. 

Obviously, by the semigroup property, we have $\P_t(\mathfrak h^\eps f)=\mathfrak h^\eps(\P_tf)$ for all $t,\eps>0$. Since, by a simple change of variable,
\begin{equation*}
\mathfrak h^\eps f
=
\int_0^{+\infty}\P_{\eps r}f\,\kappa(r)\di r,
\quad
\text{for all}\ \eps>0,	
\end{equation*}
by~\eqref{eq:kappa_kernel} we immediately deduce that $\mathfrak h^\eps f$ converges to~$f$ as $\eps\to0^+$ strongly in $\Leb^p(X,\m)$ if $p<+\infty$ and weakly* in $\Leb^\infty(X,\m)$.
The semigroup mollification operator satisfies the following natural $\Sob^{1,2}$-approximation property.  

\begin{lemma}[$\Sob^{1,2}$-approximation via $(\mathfrak h^\eps)_{\eps>0}$]
\label{res:semigroup_moll_Sob_approx}
If $f\in\Sob^{1,2}(X,\d,\m)$, then $\mathfrak h^\eps f\to f$ in $\Sob^{1,2}(X,\d,\m)$ as $\eps\to0^+$.
\end{lemma}

\begin{proof}
From the definition in~\eqref{eq:weak_gradients_collection}, it follows that
\begin{equation*}
\Gamma(\mathfrak h^\eps f-f)
\le
\int_0^{+\infty}
\Gamma(\P_{\eps r}f-f)\,
\kappa(r)\di r
\quad
\text{$\m$-a.e.\ in}\ X, 	
\end{equation*}
so that the conclusion follows by \cref{res:heat_flow_sob_cont} and the Dominated Convergence Theorem. 
\end{proof}

For the reader's convenience, we briefly prove the following regularity result for the Laplacian of the semigroup mollification operator.

\begin{lemma}[Laplacian of $(\mathfrak h^\eps)_{\eps>0}$]
\label{res:pazy}
Let $p\in[1,+\infty]$. If $f\in\Leb^2(X,\m)\cap\Leb^p(X,\m)$, then 
\begin{equation*}
-\Delta_{\d,\m}(\mathfrak h^\eps f)
=
\frac1{\eps^2}\int_0^{+\infty}\P_rf\,\kappa'\left(\tfrac r\eps\right)\di r
\in
\Leb^2(X,\m)\cap\Leb^p(X,\m)
\end{equation*}
for all $\eps>0$.
\end{lemma}

\begin{proof}
We argue as in the proof of~\cite{P83}*{Theorem~2.7}. Without loss of generality, we can assume $\eps=1$. If $t>0$, then
\begin{align*}
\frac{\P_t(\mathfrak h^1 f)-\mathfrak h^1 f}{t}
=
\int_0^{+\infty}\frac{\P_{r+t}f-\P_rf}{t}\,\kappa(r)\di r
=
\int_0^{+\infty}\P_rf\,\frac{\kappa(r-t)-\kappa(r)}{t}\di r.
\end{align*}
As $t\to0^+$, the integrand in the last term converges to $-\P_rf\kappa'(r)$ uniformly for $r\in[0,+\infty)$. Thus, in virtue of~\eqref{eq:Laplacian_char_strong_conv}, we get that $\mathfrak h^1 f\in\Dom(\Delta_{\d,\m})$ and
\begin{align*}
\Delta_{\d,\m}(\mathfrak h^1 f)
=
\lim_{t\to0^+}
\frac{\P_t(\mathfrak h^1 f)-\mathfrak h^1 f}{t}
=
-\int_0^{+\infty}\P_rf\,\kappa'(r)\di r
\end{align*} 
and the conclusion follows.  
\end{proof}

\subsection{A differentiation formula}

We now prove \cref{res:diff_formula} below. 
This result was proved for the first time in~\cite{AGS15}*{Lemma~2.1} to provide a very general formulation, in the weak sense and with minimal requirements on the regularity of the functions involved, of the simple differentiation formula
\begin{equation}\label{eq:diff_formula_simple}
\frac{\de}{\de s}\P_s(\P_{t-s}f)^2
=
\P_s\left\{\Delta(\P_{t-s}f)^2-2(\P_{t-s}f)(\Delta\P_{t-s}f)\right\}
=2\P_s|\nabla\P_{t-s}f|^2_\g
\end{equation}
valid for all $f\in\Cont^\infty(\M)$ on a Riemannian manifold $(\M,\g)$, see~\cites{B06,BGL14,BL06,W11} for an account. 

Note that the differentiation formula~\eqref{eq:diff_formula_simple}, as well as \cref{res:diff_formula} below, does not require any information about the curvature of the ambient space.
For the reader's convenience and in order to keep this work the most self-contained as possible, we provide a proof of this result in our setting.   

\begin{lemma}[Differentiation formula]\label{res:diff_formula}
Let $f\in\Leb^2(X,\m)$ and $\phi\in\Leb^2(X,\d,\m)\cap\Leb^\infty(X,\m)$. 
If $t>0$, then
\begin{equation}\label{eq:bakry_functions}
\begin{split}
s\mapsto\mathsf{A}_t[f;\phi](s)
=\frac12
\int_X(\P_{t-s}f)^2\,\P_s\phi\di\m
&
\in\Cont([0,t])\cap\Cont^1([0,t)),
\\
s\mapsto\mathsf{B}_t[f;\phi](s)
=
\int_X\Gamma(\P_{t-s}f)\,\P_s\phi\di\m
&
\in\Cont([0,t))
\end{split}
\end{equation}
and
\begin{equation}\label{eq:diff_formula}
\frac{\de}{\de s}
\mathsf{A}_t[f;\phi](s)
=
\mathsf{B}_t[f;\phi](s)
\quad
\text{for all}\ s\in[0,t).
\end{equation}
The regularity of the functions $\mathsf A$ and $\mathsf B$ in~\eqref{eq:bakry_functions} and the differentiation formula in~\eqref{eq:diff_formula} extend to $s=t$ if $f\in\Sob^{1,2}(X,\d,\m)$. 
\end{lemma}

\begin{proof}
We divide the proof in four steps.

\smallskip

\textit{Step~1: continuity of $\mathsf A$}.
The function $s\mapsto\P_{t-s}f$, $s\in[0,t]$, is strongly continuous in $\Leb^2(X,\m)$ by the definition of the heat flow.
Thanks to the $\Leb^1$-contraction property of the heat semigroup, by a simple approximation argument we easily get that $s\mapsto\P_s\phi$, $s\in[0,t]$, is weakly* continuous in $\Leb^\infty(X,\m)$.
This prove that $s\mapsto\mathsf A_t[f;\phi](s)\in\Cont([0,t])$.

\smallskip

\textit{Step~2: continuity of $\mathsf B$}.
Since the function $s\mapsto\Gamma(\P_{t-s}f)$, $s\in[0,t)$, is strongly continuous in $\Leb^1(X,\m)$ by~\eqref{eq:min_weak_grad_f_t_is_cont}, and since the function $s\mapsto\P_s\phi$, $s\in[0,t]$, is weakly* continuous in $\Leb^\infty(X,\m)$ by Step~1, we easily deduce that $s\mapsto\mathsf B_t[f;\phi](s)\in\Cont([0,t))$.

\smallskip

\textit{Step~3: proof of the differentiation formula~\eqref{eq:diff_formula}}.
Let us first assume that 
\begin{equation}\label{eq:diff_formula_assumption_step3}
f\in\Leb^2(X,\m)\cap\Leb^\infty(X,\m),\
\phi\in\Leb^\infty(X,\m)\cap\Dom(\Delta_{\d,\m})\
\text{with}\
\Delta_{\d,\m}\phi\in\Leb^\infty(X,\m).	
\end{equation}
Then we have
\begin{equation*}
\lim_{h\to0}\frac{\P_{t-(s+h)}f-\P_{t-s}f}{h}
=
-\Delta_{\d,\m}\P_{t-s}f
\end{equation*}
strongly in $\Leb^2(X,\m)$ for all $s\in[0,t)$ by~\eqref{eq:Laplacian_char_strong_conv} and
\begin{equation*}
\lim_{h\to0}\frac{\P_{s+h}\phi-\P_s\phi}{h}
=
\Delta_{\d,\m}\P_s\phi
\end{equation*} 
weakly* in $\Leb^\infty(X,\m)$ for all $s\in[0,t]$ by~\eqref{eq:def_heat_L_infty} and again by~\eqref{eq:Laplacian_char_strong_conv}. 
Hence we get that 
\begin{equation*}
\frac{\de}{\de s}
\mathsf{A}_t[f;\phi](s)
=
\int_X\bigg(
-\P_{t-s}f\,\Delta_{\d,\m}\P_{t-s}f\,\P_s\phi
+
\frac12
(\P_{t-s}f)^2\,\Delta_{\d,\m}\P_s\phi
\bigg)\di\m
\end{equation*}
for all $s\in[0,t)$. Since $\P_{t-s}f\in\Leb^\infty(X,\m)$ by~\eqref{eq:diff_formula_assumption_step3} according to~\eqref{eq:sub-Markov}, we have $(\P_{t-s}f)^2\in\Sob^{1,2}(X,\d,\m)$ and thus, thanks to the integration-by-part formula~\eqref{eq:int_by_part_formula} and the chain rule~\eqref{eq:chain_rule_Laplacian} for the Laplacian, we can compute
\begin{align*}
\int_X(\P_{t-s}f)^2\,\Delta_{\d,\m}\P_s\phi\di\m
&=
\int_X\Delta_{\d,\m}(\P_{t-s}f)^2\,\P_s\phi\di\m\\
&=
-2\int_X\P_{t-s}f\,\Delta_{\d,\m}\P_{t-s}f\,\P_s\phi\di\m
+
2\int_X\Gamma(\P_{t-s}f)\,\P_s\phi\di\m
\end{align*}
for all $s\in[0,t)$ and~\eqref{eq:diff_formula} follows.

Now let $f\in\Leb^2(X,\m)$ and keep $\phi$ as in~\eqref{eq:diff_formula_assumption_step3}. Let $f_n=-n\vee f\wedge n\in\Leb^2(X,\m)\cap\Leb^\infty(X,\m)$ for all $n\in\N$ and note that $f_n\to f$ in $\Leb^2(X,\m)$ as $n\to+\infty$. By~\eqref{eq:diff_formula} applied to~$f_n$ and~$\phi$, we know that
\begin{equation}\label{eq:diff_formula_spet3_integrated}
\mathsf A_t[f_n;\phi](s_1)
-
\mathsf A_t[f_n;\phi](s_0)
=
\int_{s_0}^{s_1}\mathsf B_t[f_n;\phi](s)\di s
\end{equation}
for all $0\le s_0<s_1<t$.
Since 
\begin{equation*}
||\D\P_{t-s}f_n|_w-|\D\P_{t-s}f|_w|^2
\le
\Gamma(\P_{t-s}(f_n-f))
\end{equation*}
for all $n\in\N$ and $s\in[0,t)$ by~\eqref{eq:Gamma_ineq_prod_min_weak_grads}, 
and since
\begin{equation*}
\int_X\Gamma(\P_{t-s}(f_n-f))\di\m
=
2\Ch(\P_{t-s}(f_n-f))
\le
\frac1{t-s}
\int_X|f_n-f|^2\di\m
\end{equation*}
for all $n\in\N$ and $s\in[0,t)$ by~\eqref{eq:Ch_f_t_inf_formula}, 
we have that 
$\Gamma(\P_{t-s}f_n)\to\Gamma(\P_{t-s}f)$
in $\Leb^1(X,\m)$ as $n\to+\infty$ for all $s\in[0,t)$. Since also $\P_{t-s}f_n\to\P_{t-s}f$ in $\Leb^2(X,\m)$ as $n\to+\infty$ for all $s\in[0,t)$, we can pass to the limit as $n\to+\infty$ in~\eqref{eq:diff_formula_spet3_integrated} and prove~\eqref{eq:diff_formula} for all $f\in\Leb^2(X,\m)$ and~$\phi$ as in~\eqref{eq:diff_formula_assumption_step3}.

Finally, let $f\in\Leb^2(X,\m)$ and $\phi\in\Leb^2(X,\m)\cap\Leb^\infty(X,\m)$. For all $\eps>0$, we set $\phi_\eps=\mathfrak h^\eps\phi$. By \cref{res:pazy}, we know that $\phi_\eps$ is as in~\eqref{eq:diff_formula_assumption_step3} for all $\eps>0$ and moreover $\phi_\eps\to\phi$ weakly* in $\Leb^\infty(X,\m)$ as $\eps\to0^+$. By applying~\eqref{eq:diff_formula} to $f$ and $\phi_\eps$ in its integrated form and then passing to the limit as $\eps\to0^+$, we prove~\eqref{eq:diff_formula} for all $f\in\Leb^2(X,\m)$ and $\phi\in\Leb^2(X,\m)\cap\Leb^\infty(X,\m)$.

\smallskip

\textit{Step~4: the limit case $s=t$}. Let $f\in\Sob^{1,2}(X,\d,\m)$. By the Mean Value Theorem, we just need to prove that the continuity of the function $\mathsf B$ extends to $s=t$. This immediately follows since $\Gamma(\P_{t-s}f)\to\Gamma(f)$ in $\Leb^1(X,\m)$ as $s\to t^-$ thanks to \cref{res:heat_flow_sob_cont}.  
\end{proof}

\subsection{\texorpdfstring{$\wBE(\c,\infty)$}{BEw(c,infty)} condition}

We now come to the central definition of our paper. Here and in the following, we let
\begin{equation}\label{eq:def_c}
\c\colon[0,+\infty)\to(0,+\infty) 
\text{ be such that }
\c,\c^{-1}\in\Leb^\infty([0,T])
\text{ for all } T>0.
\end{equation} 

\begin{definition}[$\wBE(\c,\infty)$ condition]\label{def:wBE}
We say that $(X,\d,\m)$ satisfies the \emph{weak Bakry--\'{E}mery curvature condition} with respect to the function $\c\colon[0,+\infty)\to(0,+\infty)$ in~\eqref{eq:def_c}, $\wBE(\c,\infty)$ for short, if for all $f\in\Sob^{1,2}(X,\d,\m)$ and $t\ge0$ the function $\P_t f\in\Sob^{1,2}(X,\d,\m)$ satisfies
\begin{equation}\label{eq:def_wBE}
\Gamma(\P_t f)\le\c^2(t)\,\P_t\Gamma(f)
\quad
\text{$\m$-a.e.\ in $X$}.
\end{equation}
Although not strictly necessary, we always assume that $\c(0)=1$ for simplicity. 
\end{definition}

Clearly, if $\c(t)=e^{-Kt}$ for $t\ge0$, then~\eqref{eq:def_wBE} states that $\Gamma(\P_tf)\le e^{-2Kt}\,\P_t\Gamma(f)$ $\m$-a.e.\ in~$X$ for all $f\in\Sob^{1,2}(X,\d,\m)$, which is precisely the standard Bakry--\'Emery curvature condition $\BE(K,\infty)$.
We also observe that~\eqref{eq:def_wBE} naturally rephrases condition $(G_2)$ in~\cite{K10}*{Theorem~2.2(ii)} in our more general framework for $\tilde d=\c(t)\,d$ whenever $t\ge0$.

Note that, if~\eqref{eq:def_wBE} holds for some everywhere finite measurable function $\c\colon[0,+\infty)\to[0,+\infty)$ (and so not necessarily locally positively bounded from above and below as in~\eqref{eq:def_c}), then we can replace it with another measurable function $\c_\star\colon[0,+\infty)\to[0,+\infty)$ which is optimal in the following sense: if $t>0$ and $\c_\star(t)>0$, then for all $\eps>0$ there exists $f_\eps\in\Sob^{1,2}(X,\d,\m)$ such that 
\begin{equation*}
\m\left(\set*{x\in X : \Gamma(\P_t f_\eps)(x)\ge(\c_\star(t)-\eps)^2\,\P_t\Gamma(f_\eps)(x)}\right)>0.
\end{equation*}
By~\ref{assumption:Lip_Rademacher}, we immediately get that 
\begin{equation*}
\m(\set*{x\in X : \Gamma(f)(x)>0)}>0
\end{equation*} 
whenever $f\in\Sob^{1,2}(X,\d,\m)$ is not $\m$-equivalent to a constant function, and so $\c_\star(0)=1$. 
In the following result we collect the elementary properties of~$\c_\star$.

\begin{lemma}[Properties of $\c_\star$]
\label{res:c_star_props}
The function $\c_\star\colon[0,+\infty]\to[0,+\infty)$ satisfies:
\begin{enumerate}[label=(\roman*)]
\item\label{item:c_star_submult} 
$\c_\star(s+t)\le\c_\star(s)\,\c_\star(t)$ for all $s,t\ge0$;
\item\label{item:c_star_lsc} 
$\c_\star$ is lower semicontinuous;
\item\label{item:c_star_positive} 
$\c_\star(t)>0$ for all $t\ge0$.
\end{enumerate}
\end{lemma}

\begin{proof}
Property~\ref{item:c_star_submult} follows from the semigroup property of the heat flow and the optimality of~$\c_\star$. 
Property~\ref{item:c_star_lsc} is a consequence of \cref{res:heat_flow_sob_cont} and again of the optimality of~$\c_\star$.
By~\ref{item:c_star_submult}, if $\c(t)=0$ for some $t>0$, then $\c(t')=0$ for all $t'>t$.
So let us set $t_0=\inf\set*{t>0 : \c_\star(t)=0}$. 
By~\ref{item:c_star_lsc}, we get that $\c_\star(t_0)=0$. 
Since $\c_\star(0)=1$, we must have that $t_0>0$.
Now let $f\in\Sob^{1,2}(X,\d,\m)\cap\Leb^1(X,\m)$ be non-negative, non-constant and such that $\int_X f \di\m=1$.
Since $\c_\star(t_0)=0$, we must have $\Gamma(\P_{t_0}f)=0$ $\m$-a.e.\ in~$X$, so that $\P_{t_0}f(x)=a$ for all $x\in X$, for some $a\in\R$, by~\ref{assumption:Lip_Rademacher}.
By~\eqref{eq:heat_mass_preservation}, we thus get that $a=1$ and so we must have that $\m(X)<+\infty$. 
Without loss of generality, we can assume $\m(X)=1$.
Then, by~\eqref{eq:heat_mass_preservation} again and Jensen inequality, we get
\begin{equation*}
1
=
\left(\int_X f\di\m\right)^2
=
\left(\int_X \P_{t_0/2}f\di\m\right)^2
\le
\int_X(\P_{t_0/2}f)^2\di\m
=
\int_X f\,\P_{t_0}f\di\m
=1.
\end{equation*}
By the strict convexity of the square function, we thus get that $\P_{t_0/2}f=1$ $\m$-a.e.\ in~$X$.
Hence again $\int_X f\,\P_{t_0/2}f\di\m=1$, so that $\P_{t_0/2^n}f=1$ $\m$-a.e.\ in~$X$ for all $n\in\N$ by iterating the argument above.
Thus $f=1$ $\m$-a.e.\ in~$X$, contradicting the fact that the function~$f$ was taken non-constant.
The proof is thus complete.
\end{proof}

The following result is a simple consequence of well-known properties of subadditive functions (see~\cite{HP74}*{Chapter~VII} for example), but we briefly sketch its proof here for the reader's convenience. 

\begin{lemma}[Local boundedness of $\c_\star$]
\label{res:c_star_loc_boundedness}
There exist $t_\star\ge0$ and $K\in\R$ such that
\begin{equation}\label{eq:c_star_fekete}
\c_\star(t)\le e^{-Kt}
\quad
\text{for all}\ t\ge t_\star.
\end{equation}
In addition, $\c_\star\in\Leb^\infty_{\loc}((0,+\infty))$ and $\c_\star^{-1}\in\Leb^\infty_{\loc}([0,+\infty))$.
\end{lemma}

\begin{proof}
We define $\phi\colon[0,+\infty)\to\R$ by setting $\phi(t)=\log\c_\star(t)$ for all $t\ge0$.
By \cref{res:c_star_props}, $\phi$ is well posed, lower semicontinuous and subadditive.
By Fekete Lemma (see~\cite{HP74}*{Theorem~7.6.1} for example), we have that
\begin{equation}\label{eq:phi_fekete}
\exists
\lim_{t\to+\infty}\frac{\phi(t)}t
=
\inf_{t>0}\frac{\phi(t)}t<+\infty,
\end{equation} 
from which we immediately deduce~\eqref{eq:c_star_fekete}.
By~\cite{HP74}*{Theorem~7.4.1}, we have that $\phi\in\Leb^\infty_{\loc}((0,+\infty))$. 
Since $\liminf\limits_{t\to0^+}\c_\star(t)\ge1$ by \cref{res:c_star_props}\ref{item:c_star_lsc}, we also get that $\c_\star(t)\ge M$ for all $t\in[0,\delta]$ for some $\delta,M>0$, concluding the proof.
\end{proof}

From \cref{res:c_star_loc_boundedness}, we easily deduce the following exponential upper bound for the optimal function~$\c_\star$.

\begin{corollary}[Exponential bound for $\c_\star$]
\label{res:c_star_exp_bound}
If~\eqref{eq:def_wBE} holds for some everywhere finite measurable function 
$\c\colon[0,+\infty)\to[0,+\infty)$ 
such that 
\begin{equation*}
\limsup\limits_{t\to0^+}\c(t)<+\infty,
\end{equation*} 
then the optimal function $\c_\star\colon[0,+\infty)\to(0,+\infty)$ fulfills~\eqref{eq:def_c} and is such that
\begin{equation}\label{eq:c_star_exp_bound}
\c_\star(t)
\le
M e^{-Kt}
\quad
\text{for all}\ t\ge0
\end{equation}
for some $M\ge1$ and $K\in\R$.
\end{corollary}

By \cref{res:c_star_exp_bound}, in analogy with the classical Bakry--\'Emery condition, we may think of the (best) constant $K\in\R$ appearing in~\eqref{eq:c_star_exp_bound} as a bound from below of the generalized metric-measure Ricci curvature of the space $(X,\d,\m)$.
In analogy with the usual $\RCD$ framework, we may say that $(X,\d,\m)$ is \emph{negatively}/\emph{zero}/\emph{positively curved} if we can choose $K<0$/$K=0$/$K>0$ in~\eqref{eq:c_star_exp_bound}. 
By comparing~\eqref{eq:c_star_exp_bound} with~\eqref{eq:c_star_fekete} and~\eqref{eq:phi_fekete}, $\BE_w(\c_\star,+\infty)$ condition behaves like $\BE(K,+\infty)$ for some limit $K\in\R$ as $t\to+\infty$. 

\subsection{Poincaré inequalities}

Exploiting the differentiation formula in~\eqref{eq:diff_formula}, we can prove the following consequence of \cref{def:wBE} in analogy with~\cite{AGS15}*{Corollary~2.3}. See also~\cite{W11}*{Theorem~1.1(3) and (4)} for the same inequalities in the Riemannian setting.

\begin{proposition}[Poincaré inequalities]\label{res:poincare_ineqs}
Assume $(X,\d,\m)$ satisfies $\wBE(\c,\infty)$.
\begin{enumerate}[label=(\roman*)]
\item\label{item:poincare_ineq_lower} 
If $f\in\Leb^2(X,\m)$ and $t>0$, then $\P_t f\in\Sob^{1,2}(X,\d,\m)$ with
\begin{equation}\label{eq:poincare_ineq_lower}
2\I_{-2}(t)\,\Gamma(\P_t f)
\le\P_t(f^2)-(\P_t f)^2
\quad
\text{$\m$-a.e.\ in $X$}.
\end{equation}
\item\label{item:poincare_ineq_upper} 
If $f\in\Sob^{1,2}(X,\d,\m)$ and $t>0$, then 
\begin{equation}\label{eq:poincare_ineq_upper}
\P_t(f^2)-(\P_t f)^2
\le2\I_2(t)\,\P_t\Gamma(f)
\quad
\text{$\m$-a.e.\ in $X$}.
\end{equation}
\end{enumerate}
Here and in the following, we let 
\begin{equation}\label{eq:def_I_p}
\I_p(t)=\int_0^t\c^p(s)\di s	
\end{equation}
for all $t\ge0$ and $p\in\R$.
\end{proposition}

\begin{proof}
Fix $t>0$ and $f\in\Sob^{1,2}(X,\d,\m)$. 
By \cref{res:diff_formula}, we have
\begin{equation*}
\begin{split}
\frac{1}{2}\int_X(\P_t(f^2)-(\P_t f)^2)\,\phi\di\m
&=\frac{1}{2}\int_0^1\frac{\di}{\di s}\int_X(\P_{t-s} f)^2\,\P_s\phi\di\m\di s\\
&=\int_0^t\int_X \Gamma(\P_{t-s} f)\,\P_s\phi\di\m\di s\\
&=\int_0^t\int_X \P_s\Gamma(\P_{t-s} f)\,\phi\di\m\di s
\end{split}
\end{equation*}
for all $\phi\in\Leb^2(X,\m)\cap\Leb^\infty(X,\m)$.
Now assume $\phi\ge0$. 
By the weak Bakry-\'Emery condition~\eqref{eq:def_wBE} and the semigroup property of the heat flow, we can estimate
\begin{equation*}
\c^{-2}(s)\int_X\Gamma(\P_t f)\,\phi\di\m
\le\int_X\P_s\Gamma(\P_{t-s} f)\,\phi\di\m
\le\c^2(t-s)\int_X\P_t\Gamma(f)\,\phi\di\m
\end{equation*}
for all $s\in[0,t]$. 
Thus
\begin{equation}\label{eq:daniell}
\I_{-2}(t)\int_X\Gamma(\P_t f)\,\phi\di\m
\le
\frac{1}{2}\int_X(\P_t(f^2)-(\P_t f)^2)\,\phi\di\m
\le
\I_2(t)\int_X\P_t\Gamma(f)\,\phi\di\m
\end{equation}
for all $\phi\in\Leb^2(X,\m)\cap\Leb^\infty(X,\m)$ such that $\phi\ge0$.
In particular, we can choose $\phi=\chi_E$ for any set $E\subset X$ with finite $\m$-measure, so that inequalities~\eqref{eq:poincare_ineq_lower} and~\eqref{eq:poincare_ineq_upper} follow for all $f\in\Sob^{1,2}(X,\d,\m)$. 

Now assume $f\in\Leb^2(X,\m)$. 
By the density of $\Sob^{1,2}(X,\d,\m)$ in $\Leb^2(X,\m)$, there exists $(f_n)_{n\in\N}\subset\Sob^{1,2}(X,\d,\m)$ such that $f_n\to f$ in $\Leb^2(X,\m)$ as $n\to+\infty$. As a consequence, we have that $\P_t(f_n^2)\to\P_t(f^2)$ and $(\P_tf_n)^2\to(\P_tf)^2$ in $\Leb^1(X,\m)$ as $n\to+\infty$.
Moreover, by~\eqref{eq:Ch_f_t_inf_formula} and~\eqref{eq:Gamma_ineq_prod_min_weak_grads}, we can estimate
\begin{align*}
\int_X||\D\P_tf_n|_w-|\D\P_tf|_w|^2\di\m
&\le
\int_X\Gamma(\P_t(f_n-f))\di\m\\
&=
2\Ch(\P_t(f_n-f))\\
&\le
\frac1t\int_X|f_n-f|^2\di\m
\end{align*}
for all $n\in\N$, so that $\Gamma(\P_tf_n)\to\Gamma(\P_tf)$ in $\Leb^1(X,\m)$ as $n\to+\infty$.
By the first inequality in~\eqref{eq:daniell}, we have that
\begin{equation}\label{eq:half-daniell}
\I_{-2}(t)\int_X\Gamma(\P_t f_n)\,\phi\di\m
\le
\frac{1}{2}\int_X(\P_t(f_n^2)-(\P_t f_n)^2)\,\phi\di\m
\end{equation}
for all $\phi\in\Leb^2(X,\m)\cap\Leb^\infty(X,\m)$ such that $\phi\ge0$. Passing to the limit as $n\to+\infty$ in~\eqref{eq:half-daniell} and arguing as before, we get~\eqref{eq:poincare_ineq_lower}.
\end{proof}

\subsection{\texorpdfstring{$\BE_w$}{BEw} inequality for Lipschitz functions}

Thanks to \cref{res:poincare_ineqs}, we can prove that the heat flow of a bounded Lipschitz functions in $\Leb^2(X,\m)$ has a Lipschitz representative with controlled Lipschitz constant.

\begin{proposition}[$\wBE$ for $\Lip$-functions, I]
\label{res:lip_wBE}
Assume $(X,\d,\m)$ satisfies $\wBE(\c,\infty)$. 
If $f\in\Lip_b(X,\d)\cap\Leb^2(X,\m)$, then $\P_t f\in\Lip_b(X,\d)\cap\Leb^2(X,\m)$ with
\begin{equation}\label{eq:lip_wBE}
\Lip(\P_t f)\le\c(t)\Lip(f)
\end{equation}
for all $t\ge0$
\end{proposition}

\begin{proof}
Let $t\ge0$ be fixed.
By~\eqref{eq:sub-Markov} and~\eqref{eq:poincare_ineq_lower}, we know that $\P_t f\in\Sob^{1,2}(X,\d,\m)\cap\Leb^\infty(X,\m)$ with 
\begin{equation*}
2\I_{-2}(t)\,\Gamma(\P_t f)
\le
\P_t(f^2)
\le
\|f\|_{\Leb^\infty(X,\m)}^2
\quad
\text{$\m$-a.e.\ in $X$}.
\end{equation*}
Thus, recalling property~\ref{assumption:Lip_Rademacher}, $\P_t f$ coincides $\m$-a.e.\ in~$X$ with a (bounded) Lipschitz function. 
We now divide the proof in two steps.

\smallskip

\textit{Step~1}. Assume that $\supp f$ is bounded. Then $f\in\Lip_b(X,\d)\cap\Sob^{1,2}(X,\d,\m)$. Hence, by applying the weak Bakry--\'{E}mery condition~\eqref{eq:def_wBE} to $f$, we find that
\begin{equation*}
\Gamma(\P_t f)\le\c^2(t)\P_t\Gamma(f)\le\c^2(t)\Lip(f)^2
\quad
\text{$\m$-a.e.\ in $X$}.
\end{equation*}
Thus, again by property~\ref{assumption:Lip_Rademacher}, we get~\eqref{eq:lip_wBE}.

\smallskip

\textit{Step~2}. Now fix $x_0\in X$ and let $R>0$. We let $\eta_{x_0,R}\colon X\to[0,1]$ be such that
\begin{equation}\label{eq:def_cut_off_eta_x_0_R}
\eta_{x_0,R}(x)=\left(1-\frac{\d(x,x_0)}{R}\right)^+
\quad
\text{for all}\ x\in X.
\end{equation}
Note that $\eta_{x_0,R}\in\Lip_b(X,\d)$ with 
\begin{equation*}
\supp\eta_{x_0,R}\subset\closure{B_R(x_0)},
\quad
|\D\eta_{x_0,R}|\le\frac1R\,\chi_{\closure{B_R(x_0)}},
\quad
\Lip(\eta_{x_0,R})\le\frac1R
\end{equation*}
for all $R>0$. 
Let us set $f_n=f\eta_{x_0,n}$ for all $n\in\N$. 
Then $f_n\in\Lip_b(X,\d)\cap\Leb^2(X,\m)$ has bounded support and is such that 
\begin{equation*}
\|f_n\|_{\Leb^\infty(X,\d)}\le\|f\|_{\Leb^\infty(X,\d)},
\quad
\Lip(f_n)\le\Lip(f)+\frac1n\|f\|_{\Leb^\infty(X,\d)}
\end{equation*}
for all $n\in\N$.
By Step~1, we get that
\begin{equation}\label{eq:lip_wBE_n}
\Lip(\P_t f_n)
\le
\c(t)\Lip(f_n)
\end{equation}
for all $n\in\N$.
Since $(f_n)_{n\in\N}$ is equi-bounded and equi-Lipschitz, by~\eqref{eq:sub-Markov} and~\eqref{eq:lip_wBE_n} also $(\P_t f_n)_{n\in\N}$ is equi-bounded and equi-Lipschitz. Hence $(\P_t f_n)_{n\in\N}$ converges locally uniformly to a (bounded) Lipschitz function~$g_t$ with $\Lip(g_t)\le\c(t)\Lip(f)$. Since $f_n\to f$ in $\Leb^2(X,\m)$ as $n\to+\infty$, then $\P_tf_n\to \P_tf$ in $\Leb^2(X,\m)$ as $n\to+\infty$ and thus, up to subsequences, $\P_tf_n\to \P_tf$ $\m$-a.e.\ in~$X$. But then $g_t=\P_t f$ $\m$-a.e.\ in~$X$ and thus~\eqref{eq:lip_wBE} readily follows.  
\end{proof}

\subsection{Dual heat semigroup}

Thanks to~\eqref{eq:heat_mass_preservation}, we can define the \emph{dual heat semigroup} $\H_t\colon\Prob^{\rm ac}(X)\to\Prob^{\rm ac}(X)$ for all $t\ge0$ on absolutely continuous probability measures by setting
\begin{equation}\label{eq:dual_heat_def}
\H_t\mu=(\P_tf)\,\m
\quad
\text{for all}\ \mu=f\m\in\Prob^{\rm ac}(X).
\end{equation}
Note that $(\H_t)_{t\ge0}$ is a \emph{linearly convex} semigroup on $\Prob^{\rm ac}(X)$, in the sense that
\begin{equation}\label{eq:dual_heat_is_semigroup}
\H_{s+t}\mu=\H_s(\H_t\mu)
\end{equation} 
and 
\begin{equation}\label{eq:dual_heat_is_linear}
\H_t((1-\lambda)\mu+\lambda\nu)
=
(1-\lambda)\H_t\mu+\lambda\H_t\nu
\quad
\text{for all}\
\lambda\in[0,1]
\end{equation}
whenever $s,t\ge0$ and $\mu,\nu\in\Prob^{\rm ac}(X)$.
Note that~\eqref{eq:dual_heat_def} is well posed without assuming the $\BE_w(\c,\infty)$ condition.

The following result (which still does not require the $\BE_w(\c,\infty)$ condition) proves that the dual heat semigroup preserves the finiteness of the second moments of the measures in the domain of the entropy.

\begin{lemma}[Second moment estimate]
\label{res:2nd_moment_est_heat}
If $\mu=f\m\in\Dom(\Ent_\m)$,
then $\mu_t=\H_t\mu\in\Prob^{\rm ac}_2(X)$ with
\begin{equation}
\label{eq:2nd_moment_est_heat}
\int_X\d^2(x,x_0)\di\mu_t
\le
e^{4t}
\bigg(
\Ent_\m(\mu)
+
2\int_X\d^2(x,x_0)\di\mu
\bigg)
\end{equation}
for all $t\ge0$, whenever $x_0\in X$ is given.
In particular, 
$\H_t(\Dom(\Ent_\m))\subset\Prob_p^{\rm ac}(X)$ for all $p\in[1,2]$ and $t\ge0$.
\end{lemma}

\begin{proof}
Fix $x_0\in X$ and set $V(x)=\d(x,x_0)$ for all $x\in X$.
Then $V\in\Lip(X)$ with $\Lip(V)\le1$ and~\eqref{eq:2nd_moment_est_heat} follows by~\cite{AGS14}*{Theorem~4.20} if $f\in\Leb^2(X,\m)$. 
The conclusion then follows by considering $f_n=q_n^{-1}(f\wedge n)\in\Leb^2(X,\m)$, $q_n=\int_Xf\wedge n\di\m$, $n\in\N$, and passing to the limit as $n\to+\infty$ by the Monotone Convergence Theorem.
Since $\Prob_2(X)\subset\Prob_p(X)$ by Jensen inequality, the proof is complete.
\end{proof}

The following simple result provides a useful sufficient condition to extend the dual heat semigroup to a $W_p$-Lipschitz map on the whole $p$-Wasserstein space for $p\in[1,2]$.

\begin{lemma}[$\Lip$-extension of $\H_t$ on $\Prob_p(X)$ for $p\in[1,2{]}$]
\label{res:dual_heat_Lip_ext}
Let $p\in[1,2]$ and let $C\colon[0,+\infty)\to(0,+\infty)$ be locally bounded.
If the dual heat semigroup defined in~\eqref{eq:dual_heat_def} satisfies
\begin{equation*}
W_p(\H_t\mu,\H_t\nu)
\le
C(t)\,
W_p(\mu,\nu)
\quad
\text{for all}\ 
\mu,\nu\in\mathscr D
\end{equation*} 
for some $W_p$-dense subset $\mathscr D$ of $\Prob_p(X)$, then it uniquely extends to a $W_p$-Lipschitz map (for which we retain the same notation) 
$\H_t\colon\Prob_p(X)\to\Prob_p(X)$ such that
\begin{equation*}
W_p(\H_t\mu,\H_t\nu)
\le
C(t)\, W_p(\mu,\nu)
\quad\text{for all}\
\mu,\nu\in\Prob_p(X).
\end{equation*}
In addition, the maps $(\H_t)_{t\ge0}\colon\Prob_p(X)\to\Prob_p(X)$ still satisfy~\eqref{eq:dual_heat_is_semigroup} and~\eqref{eq:dual_heat_is_linear} for all $\mu,\nu\in\Prob_p(X)$.
\end{lemma}

\begin{proof}
The extension of the dual heat semigroup readily follows from the $W_p$-density of~$\mathscr D$ in~$\Prob_p(X)$ and the completeness of the $p$-Wasserstein space.
The validity of~\eqref{eq:dual_heat_is_semigroup} and~\eqref{eq:dual_heat_is_linear} for all $\mu,\nu\in\Prob_p(X)$ is a direct consequence of the joint convexity of the $p$-Wasserstein distance.
\end{proof}

In the following result, we prove that the $\BE_w(\c,\infty)$ condition implies that the dual heat semigroup can be extended to a $W_1$-Lipschitz map on the whole $1$-Wasserstein space. See~\cite{AGS15}*{Proposition~3.2(i)}, \cite{K10}*{Proposition~3.7} and~\cite{K13}*{Theorem~2.2(ii)}.

\begin{proposition}[$\wBE(\c,\infty)\Rightarrow\H_t$ is $W_1$-$\Lip$]
\label{res:dual_semigroup_W_1_Lip}
Assume $(X,\d,\m)$ satisfies $\wBE(\c,\infty)$.
For $t\ge0$, the dual heat semigroup~\eqref{eq:dual_heat_def} uniquely extends to a $W_1$-Lipschitz map with
\begin{equation}\label{eq:dual_semigroup_W_1_Lip}
W_1(\H_t\mu,\H_t\nu)
\le
\c(t)\,
W_1(\mu,\nu)
\quad\text{for all}\
\mu,\nu\in\Prob_1(X).
\end{equation}
\end{proposition}

\begin{proof}
Let $t\ge0$ be fixed.
Thanks to \cref{res:dual_heat_Lip_ext}, we just need to prove that
\begin{equation}\label{eq:dual_heat_W_1_proof}
W_1(\H_t\mu,\H_t\nu)
\le
\c(t)\,
W_1(\mu,\nu)
\quad
\text{for all}\
\mu,\nu\in\Dom(\Ent_\m).
\end{equation}   
So let $\mu,\nu\in\Dom(\Ent_\m)$ with $\mu=f\m$ and $\nu=g\m$.
Let $\phi\in\Lip(X)$ with $\Lip(\phi)\le1$ and bounded support.
Then $\phi\in\Lip_b(X)\cap\Leb^2(X,\m)$ and so $\phi_t=\P_t\phi\in\Lip_b(X)\cap\Leb^2(X,\m)$ with $\Lip(\phi_t)\le\c(t)$ for all $t\ge0$ by \cref{res:lip_wBE}. 
Thus, setting $\mu_t=\H_t\mu$ and $\nu_t=\H_t\nu$, we can estimate
\begin{align*}
\int_X\phi\di\,(\mu_t-\nu_t)
&=
\int_X\phi\,f_t\di\m
-
\int_X\phi\,g_t\di\m
=
\int_X\phi_t\,f\di\m
-
\int_X\phi_t\,g\di\m\\
&=
\int_X\phi_t\di\,(\mu-\nu)
\le
\c(t)\,W_1(\mu,\nu)	
\end{align*}  
for all $t\ge0$ by~\eqref{eq:kantorovich_duality}. 
Passing to the supremum on all $\phi\in\Lip(X)$ with $\Lip(\phi)\le1$ and bounded support, by~\eqref{eq:kantorovich_duality_p=1} we get~\eqref{eq:dual_heat_W_1_proof} and the proof is complete.
\end{proof}

\subsection{Pointwise version of the heat semigroup}

Let $(\H_t)_{t\ge0}$ be the dual heat semigroup defined from $\Prob^{\rm ac}(X)$ to itself as in~\eqref{eq:dual_heat_def} and assume that, for some $p\in[1,2]$,  
\begin{equation}\label{eq:dual_heat_p_ext_assumption}
\text{$(\H_t)_{t\ge0}$ admits a unique $W_p$-continuous extension from $\Prob_p(X)$ to itself}.
\end{equation}

If~\eqref{eq:dual_heat_p_ext_assumption} holds, then for all $t\ge0$ we can define the (everywhere defined) \emph{pointwise version} of the heat semigroup 
\begin{equation}\label{eq:pointwise_heat_def}
\tilde\P_t f(x)
=
\int_X f\di\H_t\delta_x,
\quad
x\in X,
\end{equation}
whenever $f\colon X\to\overline{\R}$ is either a bounded or a non-negative Borel function. 
Note that, by the very definition~\eqref{eq:pointwise_heat_def}, $(\tilde\P_t)_{t\ge0}$ defines a linear semigroup of $\Leb^\infty$-contractions on bounded Borel functions, in the sense that
\begin{equation*}
\tilde\P_{s+t}f=\tilde\P_s(\tilde\P_t f)
\end{equation*}
and
\begin{equation}\label{eq:pointwise_heat_L_infty_estim}
\|\tilde\P_tf\|_{\Leb^\infty(X,\m)}\le\|f\|_{\Leb^\infty(X,\m)}
\end{equation}
whenever $f\in\Leb^\infty(X,\m)$ is Borel and $s,t\ge0$.
The following result lists the main properties of $(\tilde\P_t)_{t\ge0}$, see~\cite{AGS15}*{Proposition~3.2(ii) and~(iii)}.

\begin{proposition}[Properties of $\tilde\P_t$]
\label{res:pointwise_heat_props}
Assume~\eqref{eq:dual_heat_p_ext_assumption} holds for some $p\in[1,2]$.
\begin{enumerate}[label=(\roman*)]

\item\label{item:pointwise_heat_on_C_b} 
If $f\in\Leb^\infty(X,\m)$ is lower (resp., upper) semicontinuous, then $\tilde\P_tf$ is lower (resp., upper) semicontinuous for all $t\ge0$. As a consequence, if $f\in\Cont_b(X)$, then $\tilde\P_tf\in\Cont_b(X)$ for all $t\ge0$.

\item\label{item:pointwise_heat_on_L_infty_Borel}
If $f\in\Leb^\infty(X,\m)$ is Borel, then $\tilde\P_tf=\P_tf$ $\m$-a.e.\ in~$X$ for all $t\ge0$.

\item\label{item:pointwise_heat_duality_on_Prob_p}
If $f\in\Leb^\infty(X,\m)$ is Borel and $\mu\in\Prob_p(X)$, then
\begin{equation}
\label{eq:pointwise_heat_duality_on_Prob_p}
\int_X \tilde\P_t f\di\mu
=
\int_X f\di\H_t\mu
\end{equation}
for all $t\ge0$.

\end{enumerate}	
\end{proposition}

\begin{proof}
Let $t\ge0$ be fixed.
We prove the three statements separately.

\smallskip

\textit{Proof of~\ref{item:pointwise_heat_on_C_b}}.
By the linearity of $\tilde\P_t$ and $\P_t$, we can assume that $f\ge0$ without loss of generality. 
Let $f$ be lower semicontinuous.  
If $x_n\to x$ in~$X$ as $n\to+\infty$, then $\delta_{x_n}\to\delta_x$ in $\Prob_p(X)$ as $n\to+\infty$ and thus $\H_t\delta_{x_n}\to\H_t\delta_x$ in~$\Prob_p(X)$ as $n\to+\infty$ by~\eqref{eq:dual_heat_p_ext_assumption}. 
Hence $\H_t\delta_{x_n}\weakto\H_t\delta_x$ in~$\Prob(X)$ as $n\to+\infty$ and thus
\begin{equation*}
\H_t\delta_x(\set*{f>t})
\le
\liminf_{n\to+\infty}
\H_t\delta_{x_n}(\set*{f>t}).
\end{equation*}
Thus
\begin{align*}
\tilde\P_tf(x)
=\int_0^{+\infty}\H_t\delta_x(\set*{f>t})\di t
\le
\liminf_{n\to+\infty}
\int_0^{+\infty}\H_t\delta_{x_n}(\set*{f>t})\di t
=
\liminf_{n\to+\infty}
\tilde\P_tf(x_n)
\end{align*}
by Fatou Lemma. 
The proof is similar in the case $f$ is upper semicontinuous. This proves~\ref{item:pointwise_heat_on_C_b}.

\smallskip

\textit{Proof of~\ref{item:pointwise_heat_on_L_infty_Borel}}.
We divide the proof in two steps.

\smallskip

\textit{Step~1}.
Let $g\in\Leb^1(X,\m)$ have compact support $K=\supp g$ and be such that $\mu=g\m\in\Prob_p(X)$.
We claim that 
\begin{equation}\label{eq:pointwise_heat_on_L_infty_Borel_step_1}
\H_t(g\m)
=
\int_X g(x)\,\H_t\delta_x\di\m(x)
\quad
\text{in}\
\Prob_p(X).
\end{equation} 
Indeed, given $n\in\N$, we can find~$n$ points $p_1^n,\dots,p_n^n\in K$ such that
\begin{equation*}
K\subset\bigcup_{i=1}^n B\left(p_i^n,\tfrac1n\right).
\end{equation*} 
Let us set 
\begin{equation*}
A_1^n
=
B\left(p_1^n,\tfrac1n\right),
\quad
A_i^n
=
B\left(x_i^n,\tfrac1n\right)
\setminus
\bigcup_{j<i}B\left(p_j^n,\tfrac1n\right)\
\text{for all}\ i=2,\dots,n,
\end{equation*} 
so that $\diam A_i^n\le\frac2n$ 
for all $i=1,\dots,n$.
Without loss of generality, we can assume that $A_i^n\ne\emptyset$ 
for all $i=1,\dots,n$, 
so that we can choose some $x_i^n\in A_i^n$ 
for all $i=1,\dots,n$.
Let us set 
\begin{equation*}
a_{i,n}=\int_{A_i^n}g\di\m
\quad
\text{for all $i=1,\dots,n$},
\end{equation*} 
so that $\sum_{i=1}^n a_{i,n}=1$.
We thus define
\begin{equation*}
\mu_n
=
\sum_{i=1}^n a_{i,n}\delta_{x_i^n}
\in
\Prob_p(X)
\quad
\text{for all}\ n\in\N.
\end{equation*}
We have that $\mu_n\overset{W_p}\longto\mu$ as $n\to+\infty$.
Indeed, if $\phi\in\Cont(X)$, then~$\phi$ is uniformly continuous on~$K$ and thus
\begin{align*}
\limsup_{n\to+\infty}
\bigg|\int_X\phi\di\mu_n
-
\int_X\phi\di\mu\,\bigg|
\le
\limsup_{n\to+\infty}
\sum_{i=1}^n
\int_{A_i^n}|\phi(x_i^n)-\phi(x)|\di\mu(x)
=0.
\end{align*}
By~\eqref{eq:dual_heat_p_ext_assumption} we thus get that $\nu_n=\H_t\mu_n\overset{W_p}\longto\nu=\H_t\mu$ as $n\to+\infty$. Let us set
\begin{equation*} 
\tilde\nu=\int_X\H_t\delta_x\di\mu(x)
\in
\Prob_p(X),	
\end{equation*}
which is well posed thanks to~\eqref{eq:dual_heat_p_ext_assumption} and the compactness of~$K$.
By using twice the joint convexity of the $p$-Wasserstein distance (see~\cite{V09}*{Theorem~4.8}), we can estimate
\begin{align*}
W_p(\nu_n,\tilde\nu)
&=
W_p\left(
\sum_{i=1}^n
\int_{A_i^n}\H_t\delta_{x_i^n}\di\m(x),
\sum_{i=1}^n
\int_{A_i^n}\H_t\delta_x\di\m(x)
\right)\\
&\le
\sum_{i=1}^n\m(A_i^n)\,
W_p\left(
\aint_{A_i^n}\H_t\delta_{x_i^n}\di\m(x),
\aint_{A_i^n}\H_t\delta_x\di\m(x)
\right)\\
&\le
\sum_{i=1}^n
\aint_{A_i^n}
W_p\left(
\H_t\delta_{x_i^n},
\H_t\delta_x
\right)
\di\m(x)
\end{align*}
for all $n\in\N$, so that $\nu_n\overset{W_p}{\longto}\tilde\nu$ as $n\to+\infty$ by the uniform $W_p$-continuity of the map $x\mapsto\H_t\delta_x$ on the compact set~$K$. 
Thus $\tilde\nu=\nu$ in $\Prob_p(X)$ and the claim~\eqref{eq:pointwise_heat_on_L_infty_Borel_step_1} follows.

\smallskip

\textit{Step~2}.
Thanks to Step~1 and Fubini Theorem, we can compute
\begin{align*}
\int_X \P_tf\,g\di\m
&=
\int_X f\,\P_tg\di\m\\
&=
\int_X f\di\H_t(g\m)\\
&=
\int_X g(x)\int_X f(y)\di\H_t\delta_x(y)\di\m(x)\\
&=
\int_X \tilde\P_t f\,g\di\m
\end{align*} 
for all $g\in\Leb^1(X,\m)$ with compact support and such that $\mu=g\m\in\Prob_p(X)$.
Hence 
\begin{equation*}
\int_X \P_tf\,g\di\m
=
\int_X \tilde\P_t f\,g\di\m
\end{equation*}
for all non-negative $g\in\Leb^1(X,\m)$ and~\ref{item:pointwise_heat_on_L_infty_Borel} immediately follows.

\smallskip

\textit{Proof of~\ref{item:pointwise_heat_duality_on_Prob_p}}.
We divide the proof in two steps.

\textit{Step~1}.
Assume $f\in\Cont_b(X)$. 
By~\ref{item:pointwise_heat_on_L_infty_Borel}, we thus know that $\tilde\P_tf=\P_tf$ $\m$-a.e.\ in~$X$, so that we can compute
\begin{equation}\label{eq:bruco}
\int_X\tilde\P_t f\di\mu
=
\int_X\P_t f\,g\di\m
=
\int_X f\,\P_t g\di\m
=
\int_X f\di\H_t\mu
\end{equation} 
for all $\mu=g\m\in\Prob^{\rm ac}(X)$.
Now if $\mu\in\Prob_p(X)$, then we can find $\mu_n\in\Prob_p^{\rm ac}(X)$ such that $\mu_n\overset{W_p}\longto\mu$ as $n\to+\infty$. 
By~\eqref{eq:dual_heat_p_ext_assumption}, we also have $\H_t\mu_n\overset{W_p}\longto\H_t\mu$ as $n\to+\infty$, and thus by~\eqref{eq:bruco} and~\ref{item:pointwise_heat_on_C_b} we get
\begin{align*}
\int_X\tilde\P_t f\di\mu
=
\lim_{n\to+\infty}
\int_X\tilde\P_t f\di\mu_n
=
\lim_{n\to+\infty}
\int_X f\di\H_t\mu_n
=
\int_X f\di\H_t\mu,
\end{align*} 
proving~\eqref{eq:pointwise_heat_duality_on_Prob_p} whenever $f\in\Cont_b(X)$.

\smallskip

\textit{Step~2}.
Let $K\subset X$ be a non-empty bounded closed set.
For each $n\in\N$, let us set
\begin{equation*}
f_n(x)=[1-n\dist(x,K)]^+
\quad
\text{for all}\ x\in X.
\end{equation*} 
Then $f_n\in\Cont_b(X)$ and 
$\chi_K\le f_n\le\chi_H$ 
for all $n\in\N$, where 
$H=\set*{x\in X : \dist(x,K)\le1}$, 
and $f_n(x)\downarrow\chi_K(x)$ for all $x\in X$ as $n\to+\infty$.  
Hence $\tilde\P_t f_n(x)\downarrow \tilde\P_t\chi_K(x)$ for all $x\in X$ as $n\to+\infty$ and thus
\begin{equation*}
\int_X\tilde\P_t\chi_K\di\m
=
\lim_{n\to+\infty}
\int_X\tilde\P_t f_n\di\mu
=
\lim_{n\to+\infty}
\int_X f_n\di\H_t\mu
=
\int_X\chi_K\di\H_t\mu
\end{equation*}
by Step~1 and the Monotone Convergence Theorem.
Thus~\ref{item:pointwise_heat_duality_on_Prob_p} follows by the Monotone Class Theorem (see~\cite{D19}*{Theorem~5.2.2} for example).
\end{proof}

\subsection{\texorpdfstring{$\BE_w$}{BEw} inequality for Lipschitz functions, refined}

We now refine \cref{res:lip_wBE} to the following \cref{res:wBE-like}, where we prove an everywhere pointwise gradient bound for the heat flow starting from Lipschitz functions.

\begin{proposition}[$\BE_w$ for $\Lip$-functions, II]
\label{res:wBE-like}
Assume $(X,\d,\m)$ satisfies $\wBE(\c,\infty)$.
If $f\in\Lip_b(X,\d)\cap\Leb^2(X,\m)$ with $|\D^*f|\in\Leb^2(X,\m)$, then
$\tilde{\P}_t f\in\Lip_b(X,\d)\cap\Sob^{1,2}(X,\d,\m)$ with 
\begin{equation}\label{eq:wBE-like}
|\D^*\tilde{\P}_t f|^2(x)
\le
\c^2(t)
\,
\tilde{\P}_t(|\D^*f|^2)(x)
\quad
\text{for all}\ x\in X
\end{equation}
for all $t\ge0$. 
\end{proposition} 

In the proof of \cref{res:wBE-like}, we need the following technical result, see~\cite{AGS15}*{Proposition~3.11}. For the reader's convenience, we recall its short proof here.

\begin{lemma}[Reverse slope estimate]
\label{res:reverse_slope_estim}
Let $f\in\Cont_b(X)\cap\Sob^{1,2}(X,\d,\m)$. If $\Gamma(f)\le G^2$ $\m$-a.e.\ in~$X$ for some upper semicontinuous $G\in\Leb^\infty(X,\m)$, then $f\in\Lip(X)$ and $|\D^* f|(x)\le G(x)$ for all $x\in X$.
\end{lemma}

\begin{proof}
Since $G\in\Leb^\infty(X,\m)$, we immediately get $f\in\Lip(X)$ from property~\ref{assumption:Lip_Rademacher}.
Let $x\in X$ be fixed.
For all $\eps>0$, set $G_\eps=\sup\limits_{y\in B_\eps(x)}\zeta(y)$ and define 
\begin{equation*}
\psi_\eps(y)
=
\min\set*{\max\set*{|f(y)-f(x)|,G_\eps\d(y,x)},G_\eps[\eps-\d(y,x)]^+}
\quad
\text{for all}\ y\in X. 
\end{equation*}
Then $\psi_\eps\in\Lip(X)$ with $\supp\psi\subset\overline{B_\eps(x)}$, so that $\psi_\eps\in\Sob^{1,2}(X,\d,\m)$. 
By~\eqref{eq:min_is_sob} and~\eqref{eq:max_is_sob}, we have $|\D\psi_\eps|_w\le\max\set*{\zeta,G_\eps}$ $\m$-a.e.\ in~$X$.
Since $\psi_\eps(y)=0$ for $\d(y,x)\ge\eps$, we must have that $|\D\psi_\eps|_w\le G_\eps$ $\m$-a.e.\ in~$X$. 
Again by~\ref{assumption:Lip_Rademacher}, we get $\Lip(\psi_\eps)\le G_\eps$.
Since $\psi_\eps(x)=0$, we conclude that 
\begin{equation}\label{eq:cioccolato}
\psi_\eps(y)\le G_\eps\,\d(y,x)
\quad
\text{for all}\ y\in X. 	
\end{equation}
Now, if $\d(y,x)<\frac\eps2$, then $[\eps-\d(y,x)]^+>\frac\eps2$ and $\psi_\eps(y)<\frac{G_\eps\eps}2$ by~\eqref{eq:cioccolato}, so that 
\begin{equation*}
|f(y)-f(x)|\le\psi_\eps(y)\le G_\eps\,\d(y,x).	
\end{equation*}
Hence
\begin{equation*}
|\D f|(x)
=
\limsup_{y\to x}\frac{|f(y)-f(x)|}{\d(y,x)}
\le
\limsup_{y\to x}\frac{G_\eps\,\d(y,x)}{\d(y,x)}
=
G_\eps.
\end{equation*}
Since $\zeta$ is upper semicontinuous, we have $\lim\limits_{\eps\to0^+}G_\eps=\zeta(x)$ and thus $|\D f|(x)\le\zeta(x)$ whenever $x\in X$. 
Since $(X,\d)$ is a length space by \cref{res:length_space}, again by the upper semicontinuity of $\zeta$ we also get $|\D^*f|\le\zeta$.
\end{proof}

\begin{proof}[Proof of \cref{res:wBE-like}]
Let $t\ge0$ be fixed.
By \cref{res:lip_wBE}, we know that $\P_t f\in\Lip_b(X,\d)\cap\Leb^2(X,\m)$. 
By \cref{res:pointwise_heat_props}, we thus know that the continuous representative of~$\P_tf$ coincides with~$\tilde{\P}_tf$. 
Since $|\D^*f|\in\Leb^2(X,\m)$, we have $f\in\Sob^{1,2}(X,\d,\m)$ and thus, by~\eqref{eq:wD_less_slope_for_Lip}, $|\D f|_w\le|\D^*f|$ $\m$-a.e.\ in~$X$.
Therefore, thanks to~\eqref{eq:sub-Markov_func} and~\eqref{eq:lip_wBE}, we can estimate 
\begin{equation*}
\Gamma(\tilde{\P}_t f)
\le
\c^2(t)\,\P_t\Gamma(f)
\le
\c^2(t)\,\P_t(|\D^*f|^2)
\quad
\text{$\m$-a.e.\ in}\ X	
\end{equation*}
and hence, since $\tilde{\P}_t(|\D^*f|^2)=\P_t(|\D^*f|^2)$ $\m$-a.e.\ in~$X$ by \cref{res:pointwise_heat_props}\ref{item:pointwise_heat_on_L_infty_Borel}, we get
\begin{equation*}
\Gamma(\tilde{\P}_t f)
\le
\c^2(t)\,\tilde\P_t(|\D^*f|^2)
\quad
\text{$\m$-a.e.\ in}\ X.	
\end{equation*}
Since the function $x\mapsto\tilde{\P}_t(|\D^*f|^2)$ is bounded and upper semicontinuous by \cref{res:pointwise_heat_props}\ref{item:pointwise_heat_on_C_b}, inequality~\eqref{eq:wBE-like} immediately follows by \cref{res:reverse_slope_estim}.
\end{proof}

\subsection{Kuwada duality}

We now come the the main result of this section, the equivalence between the weak Bakry--\'Emery inequality and the $W_2$-contractivity property of the (dual) heat semigroup.
This duality property is well known for the $\BE(K,\infty)$ condition and is due to Kuwada, see the pioneering works~\cites{K10,K13,K15}. 
This duality is also known for the stronger $\BE(K,N)$ condition (with $N<+\infty$), see~\cite{EKS15}.
In a very general framework, this equivalence has been obatined in~\cite{AGS15}*{Theorem~3.5 and Corollary~3.18}.

\begin{theorem}[Kuwada duality]
\label{res:kuwada_equivalence}
The following are equivalent. 
\begin{enumerate}[label=(\roman*)]

\item\label{item:kuwada_equiv_1} 
$(X,\d,\m)$ satisfies $\wBE(\c,\infty)$.

\item\label{item:kuwada_equiv_2} 
There exists a $W_2$-dense subset~$\mathscr{D}$ of $\Prob_2^{\rm ac}(X)$ such that $\H_t(\mathscr{D})\subset\Prob_2^{\rm ac}(X)$ and
\begin{equation}\label{eq:kuwada_equiv_2}
W_2((\P_t f)\m,(\P_t g)\m)
\le
\c(t)\,W_2(f\m,g\m)
\quad
\text{whenever}\
f\m,g\m\in\mathscr{D}
\end{equation}
for all $t\ge0$.
\end{enumerate}
If either~\ref{item:kuwada_equiv_1} or~\ref{item:kuwada_equiv_2} holds, then for all $t\ge0$ the dual heat semigroup~\eqref{eq:dual_heat_def} uniquely extends to a map $\H_t\colon\Prob_2(X)\to\Prob_2(X)$ such that
\begin{equation}\label{eq:dual_heat_weak_kuwada}
W_2(\H_t\mu,\H_t\nu)
\le
\c(t)\,W_2(\mu,\nu)
\quad
\text{for all}\ \mu,\nu\in\Prob_2(X).
\end{equation}
\end{theorem}

In the proof of \cref{res:kuwada_equivalence}, we will need the following two technical results.

The first one will be use in the proof of the implication \ref{item:kuwada_equiv_1}$\,\Rightarrow\,$\ref{item:kuwada_equiv_2} and was proved for the first time in~\cite{K10}. 
In the present framework, this result was proved in~\cite{AGS15}*{Lemma~3.4}.
For the reader's convenience, we provide a proof of it below. 
Here and in the following, we let
\begin{equation}\label{eq:Lip_star}
\Lip_\star(X)
=
\set*{f\in\Lip(X) : \supp f\ \text{is bounded and}\ f\ge0}.
\end{equation}
Thanks to~\eqref{eq:wD_less_slope_for_Lip}, we immediately see that $\Lip_\star(X)\subset\Sob^{1,2}(X,\d,\m)$.
Moreover, from its very definition~\eqref{eq:hopf-lax_def}, we see that the Hopf--Lax semigroup satisfies 
$Q_s(\Lip_\star(X))\subset\Lip_\star(X)$ for all $s\ge0$, since if $f\in\Lip_\star(X)$ then $\supp(Q_s f)\subset\supp f$ for all $s\ge0$.

\begin{lemma}[Kuwada estimate]
\label{res:heat_hopf-lax_Lip_estimate}
Assume $(X,\d,\m)$ satisfies $\wBE(\c,\infty)$.
If $f\in\Lip_\star(X)$, then
\begin{equation}\label{eq:heat_hopf-lax_Lip_estimate}
\tilde\P_t Q_1 f(y)-\tilde\P_t f(x)
\le
\frac12\,\c^2(t)\,\d^2(y,x)
\end{equation}
for all $x,y\in X$ and $t\ge0$.
\end{lemma}

In the proof of \cref{res:heat_hopf-lax_Lip_estimate}, we will use the following generalization of Fatou Lemma.
For its proof, we refer the reader to~\cite{AGS15}*{Lemma~3.3}.

\begin{lemma}[Generalized Fatou Lemma]
\label{res:fatou_gen}
If $\mu_n\in\Prob(X)$ weakly converges to $\mu\in\Prob(X)$ and $(f_n)_{n\in\N}$ are Borel equi-bounded functions such that
\begin{equation*}
\limsup_{n\to+\infty}
f_n(x_n)
\le
f(x)
\quad
\text{whenever}\ x_n\to x\
\text{as $n\to+\infty$}
\end{equation*}
for some Borel function~$f$, then
\begin{equation*}
\limsup_{n\to+\infty}
\int_X f_n\di\mu_n
\le
\int_X f\di\mu.
\end{equation*}
\end{lemma} 

\begin{proof}[Proof of \cref{res:heat_hopf-lax_Lip_estimate}]
Let $t\ge0$ and $x,y\in X$ be fixed.
Since $Q_sf\in\Lip_\star(X)$, we clearly have $Q_s f\in\Lip_b(X)\cap\Leb^2(X,\m)$ with $|\D^*Q_sf|\in\Leb^2(X,\m)$ for all $s\ge0$.
By \cref{res:pointwise_heat_props}\ref{item:pointwise_heat_on_C_b}, \cref{res:lip_wBE} and \cref{res:wBE-like}, we thus have that 
\begin{equation}\label{eq:easyjet}
\Lip(\tilde\P_t Q_sf)
\le
\c(t)\Lip(Q_s f)
\end{equation}
for all $s\ge0$ and
\begin{equation}\label{eq:ryanair}
|\D^*\tilde\P_tQ_sf|^2(x)
\le\c^2(t)\,\tilde\P_t(|\D^*Q_sf|^2)(x)
\end{equation} 
for all $s\ge0$ and $x\in X$.
Now let $\gamma\in\AC([0,1];X)$ be such that $\gamma_0=x$ and $\gamma_1=y$.
We claim that $s\mapsto\tilde\P_tQ_sf(\gamma_s)\in\AC([0,1];\R)$.
Indeed, by~\eqref{eq:hopf-lax_lip} and~\eqref{eq:easyjet}, we can estimate
\begin{align*}
|\tilde\P_tQ_{s_1}f(\gamma_{s_1})-\tilde\P_tQ_{s_0}f(\gamma_{s_0})|
&\le
|\tilde\P_tQ_{s_1}f(\gamma_{s_1})-\tilde\P_tQ_{s_1}f(\gamma_{s_0})|
+
|\tilde\P_tQ_{s_1}f(\gamma_{s_0})-\tilde\P_tQ_{s_0}f(\gamma_{s_0})|\\
&\le
\Lip(\tilde\P_tQ_{s_1}f)\,\d(\gamma_{s_1},\gamma_{s_0})
+
\int_X|Q_{s_1}f-Q_{s_0}f|\,\di\H_t\gamma_{s_0}\\
&\le
2\c(t)\Lip(f)\int_{s_0}^{s_1}|\dot\gamma_s|\di s
+
2\Lip(f)^2\,(s_1-s_0)
\end{align*}  
for all $0\le s_0<s_1\le1$.
We can now write
\begin{align*}
\frac{\tilde\P_tQ_{s+h}f(\gamma_{s+h})
-
\tilde\P_tQ_sf(\gamma_s)}{h}
&=
\int_X\frac{Q_{s+h}f-Q_sf}{h}\di\H_t\gamma_{s+h}
+
\frac{\tilde\P_tQ_sf(\gamma_{s+h})
-
\tilde\P_tQ_sf(\gamma_s)}{h}
\end{align*}
for all $0\le s<s+h\le1$.
On the one hand, we have
\begin{align*}
\limsup_{h\to0^+}
\int_X\frac{Q_{s+h}f-Q_sf}{h}\di\H_t\gamma_{s+h}
&\le
\int_X\frac{\di^{\,+}}{\di s}\,Q_sf\di\H_t\gamma_s\\
&=
-\frac12\int_X|\D Q_sf|^2\di\H_t\gamma_s\\
&=
-\frac12\tilde\P_t(|\D Q_sf|^2)(\gamma_s)
\end{align*}
by \cref{res:fatou_gen} and~\eqref{eq:hopf-lax_HJ} for all $s\in[0,1]$.
On the other hand, by the upper gradient property of the asymptotic Lipschitz constant for Lipschitz functions, we can estimate  
\begin{align*}
\limsup_{h\to0^+}
\frac{|\tilde\P_tQ_sf(\gamma_{s+h})-\tilde\P_tQ_sf(\gamma_s)|}{h}
&\le
\limsup_{h\to0^+}
\frac1h\int_s^{s+h}|\D^*\tilde\P_tQ_sf|(\gamma_r)\,|\dot\gamma_r|\di r\\
&=
|\D^*\tilde\P_tQ_sf|(\gamma_s)\,|\dot\gamma_s|\end{align*}
for $\L^1$-a.e.\ $s\in[0,1]$.
By~\eqref{eq:ryanair} and Young inequality, we have
\begin{align*}
|\D^*\tilde\P_tQ_sf|(\gamma_s)\,|\dot\gamma_s|
&\le
\c(t)\,|\dot\gamma_s|\,\sqrt{\tilde\P_t(|\D^*Q_sf|^2)(\gamma_s)}\\
&\le
\frac12\,\c^2(t)\,|\dot\gamma_s|^2
+
\frac12\,\tilde\P_t(|\D^*Q_sf|^2)(\gamma_s)
\end{align*}
for all $s\in[0,1]$. 
By~\eqref{eq:hopf-lax_slope_is_usc}, we thus have that
\begin{align*}
\frac{\di}{\di s}\,\tilde\P_t Q_sf(\gamma_s)
&\le
-\frac12\,\tilde\P_t(|\D Q_sf|^2)(\gamma_s)
+
\frac12\,\c^2(t)\,|\dot\gamma_s|^2
+
\frac12\,\tilde\P_t(|\D^*Q_sf|^2)(\gamma_s)\\
&=
\frac12\,\c^2(t)\,|\dot\gamma_s|^2
\end{align*}
for $\L^1$-a.e.\ $s\in[0,1]$, so that
\begin{align*}
\tilde\P_t Q_1 f(y)-\tilde\P_t f(x)
=
\int_0^1\frac{\di}{\di s}\,\tilde\P_t Q_sf(\gamma_s)\di s
\le
\frac12\,\c^2(t)\int_0^1|\dot\gamma_s|^2\di s
\end{align*}
and~\eqref{eq:heat_hopf-lax_Lip_estimate} follows by minimizing with respect to all curves $\gamma\in\AC([0,1];X)$ such that $\gamma_0=x$ and $\gamma_1=y$. This concludes the proof.  
\end{proof}

The second preliminary result is a well known result proved for the first time by Lisini in~\cite{L07}.
In this general framework, this result was proved in~\cite{AGS15}*{Lemma~4.12} (se also~\cite{G15}*{Theorem~2.1}). For the reader's convenience, we provide a proof of it below.

\begin{lemma}[Lisini Theorem for $\Lip$-functions]
\label{res:Lisini_Lip}
If $s\mapsto\mu_s\in\AC^2([0,1],\Prob_2(X))$, then
$s\mapsto
\int_X\phi\di\mu_s\in\AC^2([0,1];\R)$
with
\begin{equation}\label{eq:Lisini_Lip_integral}
\bigg|\int_X\phi\di\mu_1
-\int_X\phi\di\mu_0\,\bigg|
\le
\int_0^1\bigg(\int_X|\D\phi|^2\di\mu_s\bigg)^{1/2}|\dot\mu_s|\di s
\end{equation}	
for all $\phi\in\Lip_b(X)$.
In particular, we have
\begin{equation}\label{eq:Lisini_Lip_derivative}
\bigg|\frac{\di}{\di s}\int_X\phi\di\mu_s\,\bigg|^2
\le
|\dot\mu_s|^2\int_X|\D\phi|^2\di\mu_s
\quad
\text{for $\leb^1$-a.e.}\ s\in[0,1]
\end{equation} 
for all $\phi\in\Lip_b(X)$.
\end{lemma}

\begin{proof}
Let $\phi\in\Lip_b(X)$ be fixed.
By Lisini Theorem, see~\cite{L07}*{Theorem~5} or~\cite{G15}*{Theorem~2.1}, there exists
$\eta\in\Prob(\mathscr C)$, 
$\mathscr C=\Cont([0,1],(X,\d))$, 
concentrated on 
$\AC([0,1],X)$,
such that
\begin{equation}\label{eq:Lisini_1}
(\mathsf{e}_s)_\sharp\eta
=
\mu_s
\quad
\text{for all}\ s\in[0,1],
\end{equation}
where $\mathsf e_s\colon\mathscr C\to X$ is the evaluation map at time $s\in[0,1]$, and
\begin{equation}\label{eq:Lisini_2}
\int_{\mathscr C}|\dot \gamma_s|^2\di\eta(\gamma)
=
|\dot\mu_s|^2
\quad
\text{for $\leb^1$-a.e.}\ s\in[0,1].
\end{equation}
By the upper gradient property of the slope (recall~\eqref{eq:upper_gradient_def}), by~\eqref{eq:Lisini_1}, \eqref{eq:Lisini_2} and H\"older inequality, we can estimate
\begin{align*}
\bigg|\int_X\phi\di\mu_1
-\int_X\phi\di\mu_0\,\bigg|
&=
\abs*{\,\int_{\mathscr C} 
(\phi(\gamma_1)-\phi(\gamma_0))
\di\eta(\gamma)\,}\\
&\le
\int_0^1\int_{\mathscr C}|\D\phi|(\gamma_s)\,|\dot\gamma_s|\di\eta(\gamma)\di s\\
&\le
\int_0^1\left(\int_{\mathscr C}|\D\phi|^2(\gamma_s)\di\eta(\gamma)\right)^{1/2}
\left(\int_{\mathscr C}|\dot\gamma_s|^2\di\eta(\gamma)\right)^{1/2}\di s\\
&=
\int_0^1\left(\int_X|\D\phi|^2\di\mu_s\right)^{1/2}|\dot{\mu}_s|\di s,
\end{align*}
proving~\eqref{eq:Lisini_Lip_integral}. 
Inequality~\eqref{eq:Lisini_Lip_derivative} follows easily.
\end{proof}

We are now ready to prove the main result of this section.

\begin{proof}[Proof of \cref{res:kuwada_equivalence}]
We prove the two implications separately.

\smallskip

\textit{Proof of 
$
\ref{item:kuwada_equiv_1}
\Rightarrow
\ref{item:kuwada_equiv_2}
$}.
Fix $t\ge0$. We divide the proof in three steps.

\smallskip

\textit{Step~1: definition of $\H_t$ and $\tilde\P_t$}.
We define $\H_t\colon\Prob^{\rm ac}(X)\to\Prob^{\rm ac}(X)$ by setting
\begin{equation}\label{eq:dual_heat_def_weak_kuwada}
\H_t\mu=(\P_tf)\,\m
\quad
\text{for all}\ \mu=f\m\in\Prob^{\rm ac}(X),
\end{equation}
as in~\eqref{eq:dual_heat_def}.
By \cref{res:dual_semigroup_W_1_Lip}, this map can be extended to a map $\H_t\colon\Prob_1(X)\to\Prob_1(X)$ which satisfies~\eqref{eq:dual_semigroup_W_1_Lip} with $C(t)=\c(t)$. 
Hence~\eqref{eq:dual_heat_p_ext_assumption} holds with $p=1$ and we can thus define 
\begin{equation*}
\tilde{\P}_tf(x)=\int_X f\di\H_t\delta_x,
\quad
x\in X,
\end{equation*}
whenever $f\colon X\to\overline{\R}$ is either a bounded or a non-negative Borel function, as in~\eqref{eq:pointwise_heat_def}. 

\smallskip

\textit{Step~2: $W_2$-estimate for $\H_t$ on Dirac deltas}.
Let $x,y\in X$. By \cref{res:heat_hopf-lax_Lip_estimate}, we can estimate
\begin{equation}\label{eq:kuwada_heat_hopf-lax_estimate}
\int_X Q_1\phi\di\H_t\delta_y
-
\int_X \phi\di\H_t\delta_x
=
\tilde\P_tQ_1\phi(y)-\tilde\P_t\phi(x)
\le
\frac12\,\c^2(t)\,\d^2(y,x)
\end{equation}
for all $\phi\in\Lip_\star(X)$. 
Hence, by \eqref{eq:kantorovich_duality_p=2} and taking the supremum on all $\phi\in\Lip_\star(X)$ in~\eqref{eq:kuwada_heat_hopf-lax_estimate}, we get
\begin{equation}\label{eq:weak_kuwada_1_to_2_step_2_final}
W_2(\H_t\delta_y,\H_t\delta_x)
\le
\c(t)\,\d(y,x)
\end{equation}
whenever $x,y\in X$.

\smallskip

\textit{Step~3: $W_2$ estimate for $\H_t$ on $\Prob_2(X)$}.
If $\mu\in\Prob_2(X)$, then we can write 
\begin{equation}\label{eq:delta_ride}
\mu=\int_X\delta_x\di\mu(x)
\end{equation}
and thus, by \cref{res:pointwise_heat_props}\ref{item:pointwise_heat_duality_on_Prob_p}, we we can also write
\begin{equation}\label{eq:delta_ride_t}
\H_t\mu=\int_X\H_t\delta_x\di\mu(x).
\end{equation}
Now let $\mu,\nu\in\Prob_2(X)$. 
If $\pi\in\mathsf{Plan}(\mu,\nu)$, then we may use a Measurable Selection Theorem (see~\cite{V09}*{Corollary~5.22} or~\cite{B07}*{Theorem~6.9.2} for example) to select in a $\pi$-measurable way an optimal plan 
\begin{equation}\label{eq:opt_plan_omega_x_y}
\omega_{x,y}^t\in\mathsf{OptPlan}(\H_t\delta_x,\H_t\delta_y)
\quad
\text{for all}\ x,y\in X.	
\end{equation}
By~\eqref{eq:delta_ride} and~\eqref{eq:delta_ride_t}, we thus get that
\begin{equation*}
\Omega_t
=
\int_{X\times X}\omega_{x,y}^t\di\pi(x,y)
\in\mathsf{Plan}(\H_t\mu,\H_t\nu).
\end{equation*}
Hence, by~\eqref{eq:def_W_2}, the optimality of~\eqref{eq:opt_plan_omega_x_y} and by~\eqref{eq:weak_kuwada_1_to_2_step_2_final}, we can estimate 
\begin{align*}
W_2^2(\H_t\mu,\H_t\nu)
&\le
\int_{X\times X}\d^2(u,v)\di\Omega_t(u,v)\\
&=
\int_{X\times X}\int_{X\times X}\d^2(u,v)\di\omega_{x,y}^t(u,v)\di\pi(x,y)\\
&=
\int_{X\times X}W_2^2(\H_t\delta_x,\H_t\delta_y)\di\pi(x,y)\\
&\le\c^2(t)\int_{X\times X}\d^2(x,y)\di\pi(x,y)
\end{align*}
whenever $\pi\in\mathsf{Plan}(\mu,\nu)$. 
Again by~\eqref{eq:def_W_2}, we thus get~\eqref{eq:dual_heat_weak_kuwada}. By~\eqref{eq:dual_heat_def_weak_kuwada} and \cref{res:2nd_moment_est_heat}, this proves~\ref{item:kuwada_equiv_2} with $\mathscr D=\Dom(\Ent_\m)$.

\smallskip

\textit{Proof of 
$
\ref{item:kuwada_equiv_2}
\Rightarrow
\ref{item:kuwada_equiv_1}
$}.
Fix $t\ge0$. We divide the proof in three steps.

\smallskip

\textit{Step~1: definition of $\H_t$ and $\tilde\P_t$}.
We define $\H_t\colon\mathscr D\to\Prob_2^{\rm ac}(X)$ by setting
\begin{equation}\label{eq:kuwada_equiv_2_H_t_def}
\H_t(f\m)=(\P_t f)\,\m
\quad
\text{for all}\ 
f\m\in\mathscr{D}	
\end{equation}
as in~\eqref{eq:dual_heat_def}. 
Thanks to \cref{res:dual_heat_Lip_ext}, by~\eqref{eq:kuwada_equiv_2} we can extend the map~\eqref{eq:kuwada_equiv_2_H_t_def} to a map $\H_t\colon\Prob_2(X)\to\Prob_2(X)$ (for which we retain the same notation) such that
\begin{equation*}
W_2(\H_t\mu,\H_t\nu)
\le
\c(t)
\,
W_2(\mu,\nu)
\quad
\text{for all}\
\mu,\nu\in\Prob_2(X).
\end{equation*} 
Hence~\eqref{eq:dual_heat_p_ext_assumption} holds for $p=2$ and we can thus define 
\begin{equation*}
\tilde{\P}_tf(x)=\int_X f\di\H_t\delta_x,
\quad
x\in X,
\end{equation*}
whenever $f\colon X\to\overline{\R}$ is either a bounded or a non-negative Borel function, as in~\eqref{eq:pointwise_heat_def}. 

\smallskip

\textit{Step~2: $\BE_w$ on $\Lip$-functions via $\tilde\P_t$}.
Let $f\in\Lip_b(X)\cap\Leb^2(X,\m)$ with $|\D^*f|\in\Leb^2(X,\m)$. 
We claim that
\begin{equation}\label{eq:mr_tamburine}
\Gamma(\P_tf)
\le
\c^2(t)\,\P_t|\D^*f|^2
\quad
\text{$\m$-a.e.\ in~$X$}.
\end{equation}
Indeed, thanks to \cref{res:pointwise_heat_props}\ref{item:pointwise_heat_on_L_infty_Borel}, we have $\tilde\P_t f=\P_t f$ $\m$-a.e.\ in~$X$, so that $\tilde\P_tf\in\Leb^2(X,\m)\cap\Leb^\infty(X,\m)$ in particular.
Fix $x,y\in X$ and let $\gamma\in\AC^2([0,1];X)$ such that $\gamma_0=x$ and $\gamma_1=y$.
We thus have that
\begin{equation*}
s\mapsto\mu_s=\H_t\delta_{\gamma_s}
\in
\AC^2([0,1];\Prob_2(X)),
\end{equation*} 
since we can estimate
\begin{align*}
W_2(\mu_{s_1},\mu_{s_0})
\le
\c(t)\,
W_2(\delta_{\gamma_{s_0}},\delta_{\gamma_{s_1}})
\le
\c(t)\,
\d(\gamma_{s_1},\gamma_{s_0})
\le
\c(t)
\int_{s_0}^{s_1}|\dot\gamma_r|\di r
\end{align*}
for all $0\le s_0<s_1\le1$, which immediately gives
\begin{equation}\label{eq:ringo_starr}
|\dot\mu_s|
\le
\c(t)\,
|\dot\gamma_s|
\quad
\text{for $\leb^1$-a.e.}\ s\in[0,1].
\end{equation}
By \cref{res:Lisini_Lip} and~\eqref{eq:ringo_starr}, we thus get
\begin{align*}
|\tilde{\P}_tf(y)-\tilde{\P}_tf(x)|
&=
\bigg|\int_X f\di\mu_1-\int_X f\di\mu_0\,\bigg|\\
&\le
\int_0^1\left(\int_X|\D f|^2\di\mu_s\right)^{1/2}|\dot\mu_s|\di s\\
&\le
\c(t)\int_0^1\left(\int_X|\D f|^2\di\mu_s\right)^{1/2}|\dot{\gamma}_s|\di s\\
&=
\c(t)\int_0^1\left(\tilde{\P}_t|\D f|^2(\gamma_s)\right)^{1/2}|\dot{\gamma}_s|\di s.
\end{align*}
Thanks to~\eqref{eq:pointwise_heat_L_infty_estim}, we thus get
\begin{equation*}
|\tilde{\P}_tf(y)-\tilde{\P}_tf(x)|
\le
\c(t)\Lip(f)\int_0^1|\dot{\gamma}_s|\di s,
\end{equation*} 
so that $\tilde\P_t f\in\Lip(X)$ with $\Lip(\tilde\P_t f)\le\c(t)\Lip(f)$ by the length property of $(X,\d)$ (recall \cref{res:length_space}).
In addition, again by the length property of $(X,\d)$, we have
\begin{equation}\label{eq:elios}
\abs*{\tilde{\P}_tf(y)-\tilde{\P}_tf(x)}
\le
\c(t)\,
\d(y,x)
\,\sup\set*{\left(\tilde{\P}_t|\D^*f|^2(z)\right)^{1/2} : \d(z,y)\le2\d(y,x)}
\end{equation}
for all $x,y\in X$. 
Since $x\mapsto|\D^*f|(x)$ is upper semicontinuous and bounded, by \cref{res:pointwise_heat_props}\ref{item:pointwise_heat_on_C_b} also $x\mapsto\left(\tilde{\P}_t|\D^*f|^2(x)\right)^{1/2}$ is upper semicontinuous. 
Therefore, taking the $\limsup$ as $y\to x$ in~\eqref{eq:elios}, we get
\begin{equation*}
|\D^*\tilde{\P}_tf|(x)\le\c(t)\left(\tilde{\P}_t|\D^*f|^2(x)\right)^{1/2}
\quad
\text{for all}\ x\in X.
\end{equation*}
Since $|\D^* f|\in\Leb^2(X,\m)\cap\Leb^\infty(X,\m)$ is Borel, we have that $\tilde{\P}_t|\D^*f|^2=\P_t|\D^*f|^2$ $\m$-a.e.\ in~$X$ by \cref{res:pointwise_heat_props}\ref{item:pointwise_heat_on_C_b},
and thus $\tilde{\P}_t|\D^*f|^2\in\Leb^2(X,\m)$. 
Hence 
$\tilde\P_tf\in\Lip_b(X)\cap\Leb^2(X,\m)$ 
with 
$|\D^*\tilde\P_tf|\in\Leb^2(X,\m)$, 
so that 
$\tilde\P_t f\in\Sob^{1,2}(X,\d,\m)$ 
by~\eqref{eq:min_weak_grad_min_property}, with 
$\Gamma(\tilde\P_t f)\le|\D^*\tilde\P_tf|^2$ 
$\m$-a.e.\ in~$X$ by~\eqref{eq:wD_less_slope_for_Lip}. 
Since 
$\tilde\P_tf=\P_tf$ $\m$-a.e.\ in~$X$, 
we must have 
$\Gamma(\tilde\P_t f)=\Gamma(\P_t f)$ (recall~\eqref{eq:weak_gradients_collection} and again the definition of minimal weak gradient). 
Claim~\eqref{eq:mr_tamburine} is thus proved. 

\smallskip

\textit{Step~3: approximation}.
Let $f\in\Sob^{1,2}(X,\d,\m)$.
By~\eqref{eq:approx_asymp_Sob_by_Lip_b}, we can find $f_n\in\Lip_b(X)\cap\Leb^2(X,\m)$ such that $f_n\to f$ and $|\D^*f|\to|\D f|_w$ in $\Leb^2(X,\m)$ as $n\to+\infty$.
By Step~2, we have
\begin{equation}\label{eq:jacopo}
\Gamma(\P_t f_n)
\le
\c^2(t)\,
\P_t|\D^*f_n|^2
\quad
\text{$\m$-a.e.\ in~$X$}
\end{equation}
for all $n\in\N$. Since $\P_t|\D^*f_n|^2\to\P_t\Gamma(\P_t f)$ in $\Leb^1(X,\m)$ as $n\to+\infty$ by~\eqref{eq:heat_L2_contraction} and $\Gamma(\P_t f_n)\to\Gamma(\P_t f)$ by~\eqref{eq:Ch_f_t_inf_formula} and again~\eqref{eq:heat_L2_contraction}, up to possibly pass to a subsequence, we can pass to the limit as $n\to+\infty$ in~\eqref{eq:jacopo} and get~\eqref{eq:def_wBE}.
This proves~\ref{item:kuwada_equiv_1}. 
\end{proof}

\begin{remark}[Errata to the proof of~\cite{AGS15}*{Theorem~3.5}]
\label{rem:AGS_pernacchia}
In~\cite{AGS15}*{Section~3.2}, instead of~\eqref{eq:wBE-like}, the authors consider the pointwise inequality (see~\cite{AGS15}*{Equation~(3.16)}) 
\begin{equation}\label{eq:wBE-like_BAD}
|\D\P_tf|^2(x)
\le
\c^2(t)\,\tilde\P_t(|\D f|^2)(x)
\quad
\text{for all}\ x\in X,
\end{equation}
whenever $f\in\Lip_b(X)\cap\Leb^2(X,\m)$.
In the proof of~\cite{AGS15}*{Theorem~3.5}, the authors then state that inequality~\eqref{eq:wBE-like_BAD}, together with inequality~\eqref{eq:lip_wBE} (which corresponds to~\cite{AGS15}*{Equation~(3.15)}) are implied by the $W_2$-contractivity property of the dual heat semigroup. 
Since they do not use this implication in their paper, the authors do not provide a proof of this statement and only refer to~\cite{AGS14-2}*{Theorem~6.2} and to~\cite{K10}.
However, the proof of~\cite{AGS14-2}*{Theorem~6.2} uses the fact that $\tilde\P_t|\D f|^2\in\Cont(X)$ for all $t>0$, i.e.\ the \emph{$\Leb^\infty$-to-$\Cont$-regularization property} of $(\tilde\P_t)_{t>0}$ previously proved in~\cite{AGS14-2}*{Theorem~6.1} thanks to the $\EVI_K$ property of the gradient flow of the entropy.
In the general framework considered in~\cite{AGS15}*{Section~3.2}, as well as in the present one, the $\Leb^\infty$-to-$\Cont$-regularization property of~$(\tilde\P_t)_{t>0}$ is not available, and thus the continuity of the function $x\mapsto\tilde\P_t(|\D f|^2)(x)$ for~$t>0$ is not guaranteed. 
For this reason, the implication 
\begin{equation*}
\text{$W_2$-contractivity of $(\H_t)_{t\ge0}$ $\implies$ \eqref{eq:lip_wBE} and~\eqref{eq:wBE-like_BAD}}
\end{equation*}
stated in~\cite{AGS15}*{Theorem~3.5} is not completely justified.
One can get rid of this problem by replacing~\cite{AGS15}*{Equation~(3.16)} with~\eqref{eq:wBE-like} and arguing as we have done in the proof of the implication \ref{item:kuwada_equiv_2}$\,\Rightarrow\,$\ref{item:kuwada_equiv_1} of \cref{res:kuwada_equivalence} above thanks to the upper semicontinuity of the asymptotic Lipschitz constant (recall its definition in~\eqref{eq:def_loc_Lip_constant}), without affecting the validity of all the other results of~\cite{AGS15}*{Section~3.2}.
We let the interested reader check the details.
\end{remark}

\subsection{Strong Feller property and densities of the dual heat semigroup}

The following result deals with the regularization property of the pointwise heat semigroup $(\tilde\P_t)_{t>0}$ on $\Leb^2\cap\Leb^\infty$-functions, see the proof of the implication \mbox{(i)$\,\Rightarrow\,$(v)} in~\cite{AGS15}*{Theorem~3.17}. 
We briefly provide its proof below for the reader's convenience. 

\begin{corollary}[Strong Feller property]\label{res:S-Feller}
Assume $(X,\d,\m)$ satisfies $\wBE(\c,\infty)$. 
If $f\in\Leb^2(X,\m)\cap\Leb^\infty(X,\m)$ is Borel, then $\tilde{\P}_t f\in\Lip_b(X)$ with
\begin{equation}\label{eq:S-Feller}
\sqrt{2\I_{-2}(t)}\,\Lip(\tilde{\P}_t f)
\le
\|f\|_{\Leb^\infty(X,\m)}
\end{equation}
for all $t>0$. 
\end{corollary}

\begin{proof}
Let $t>0$ be fixed.
Assume $f\in\Cont_b(X)\cap\Leb^2(X,\m)$. 
Then $\P_t f\in\Sob^{1,2}(X,\d,\m)$ by~\cref{res:poincare_ineqs}\ref{item:poincare_ineq_lower} with 
$2\I_{-2}(t)\,\Gamma(\P_tf)
\le
\|f\|_{\Leb^\infty(X,\m)}^2$ $\m$-a.e.\ in~$X$.
Hence $\P_tf$ has a Lipschitz representative by~\ref{assumption:Lip_Rademacher}. Thanks to \cref{res:kuwada_equivalence} and \cref{res:pointwise_heat_props}\ref{item:pointwise_heat_on_C_b}, the Lipschitz representative of $\P_tf$ must coincide with $\tilde\P_tf$.
We thus get that $\tilde\P_t f\in\Lip_b(X)$ satisfies~\eqref{eq:S-Feller}.
Hence, arguing as in Step~2 of the proof of \cref{res:pointwise_heat_props}\ref{item:pointwise_heat_duality_on_Prob_p}, we get that $\tilde\P_t\chi_K\in\Lip_b(X)$ satisfies~\eqref{eq:S-Feller} whenever $K\subset X$ is a non-empty bounded closed set.
The conclusion thus follows by the Monotone Class Theorem (see~\cite{D19}*{Theorem~5.2.2} for example), since~\eqref{eq:S-Feller} allows to convert monotone equibounded convergence of a sequence $(f_n)_{n\in\N}$ into pointwise convergence on~$X$ of the sequence~$(\tilde\P_t f_n)_{n\in\N}$.
\end{proof}

An important consequence of \cref{res:S-Feller} is the absolute continuity property of the dual heat semigroup $(\H_t)_{t>0}$ on measures in $\Prob_2(X)$, see the proof of~\cite{AGS15}*{Theorem~3.17}. 
We provide a sketch of its proof below for the reader's convenience. 

\begin{corollary}[$\H_t(\Prob_2(X))\subset\Prob_2^{\rm ac}(X)$ for $t>0$]
\label{res:dual_heat_ll_m}
Assume $(X,\d,\m)$ satisfies $\wBE(\c,\infty)$. 
If $\mu\in\Prob_2(X)$, then $\H_t\mu\ll\m$ for all $t>0$.
\end{corollary}

\begin{proof}
Let $t>0$ be fixed.
Let $\mu\in\Prob_2(X)$ and let $A\subset X$ be a Borel set with \mbox{$\m(A)=0$}. 
Then $\chi_A\in\Leb^2(X,\m)\cap\Leb^\infty(X,\m)$ and so $\tilde\P_t\chi_A\in\Lip_b(X)$ by \cref{res:S-Feller}. 
By \cref{res:pointwise_heat_props}\ref{item:pointwise_heat_on_L_infty_Borel}, we must have that $\tilde\P_t f=\P_t f=0$ $\m$-a.e.\ in~$X$, and thus $\tilde\P_t f(x)=0$ for all $x\in X$.
Hence
\begin{equation*}
\H_t\mu(A)
=
\int_X\chi_A\di\H_t\mu
=
\int_X\tilde\P_t\chi_A\di\mu
=
0
\end{equation*} 
by \cref{res:kuwada_equivalence} and \cref{res:pointwise_heat_props}\ref{item:pointwise_heat_duality_on_Prob_p}.
The proof is complete.
\end{proof}

\begin{remark}[Extension of $(\H_t)_{t\ge0}$ on $\Prob(X)$]
Although we do not need such a generality for our purposes, it is possible to show that the dual heat semigroup can be extended to a weakly continuous map $(\H_t)_{t\ge0}\colon\Prob(X)\to\Prob(X)$ such that~\eqref{eq:dual_heat_weak_kuwada} holds for all $\mu,\nu\in\Prob(X)$ with $W_2(\mu,\nu)<+\infty$. 
Moreover, the validity of \cref{res:pointwise_heat_props}\ref{item:pointwise_heat_duality_on_Prob_p} and of
\cref{res:dual_heat_ll_m} extends to any $\mu\in\Prob(X)$. 
We refer the interested reader to~\cite{AGS15}*{Section~3.2} for the details.
\end{remark}

By \cref{res:dual_heat_ll_m}, for all $x\in X$ there exists a non-negative density $\p_t[x]\in\Leb^1(X,\m)$ such that
\begin{equation}\label{eq:def_p_t_density}
\H_t\delta_x
=
\p_t[x]\,\m
\quad
\text{for all}\ t>0.
\end{equation}
Therefore, accordingly with~\eqref{eq:pointwise_heat_def}, if $f\colon X\to\overline{\R}$ is either a bounded or a non-negative Borel function, we can then write
\begin{equation}\label{eq:pointwise_heat_p_t_density}
\tilde\P_t f(x)
=
\int_X f(y)\di\H_t\delta_x(y)
=
\int_X f(y)\,\p_t[x](y)\di\m(y),
\end{equation}
for all $t>0$, so that the definition of $(\tilde\P_t f)_{t>0}$ does not depend on the particular choice of the representative of~$f$.
By linearity, $(\tilde\P_t f)_{t>0}$ is thus well defined whenever $f\colon X\to\overline{\R}$ is a one-side bounded measurable function.

\begin{lemma}[Properties of $(\p_t{[}\!\cdot\!{]})_{t>0}$]
\label{res:density_p_t_props}
Assume $(X,\d,\m)$ satisfies $\wBE(\c,\infty)$ and let $t>0$.
The following hold.
\begin{enumerate}[label=(\roman*)]

\item\label{item:density_p_t_semigroup} 
$\tilde\P_s(\p_t[x])=\p_{s+t}[x]$ $\m$-a.e.\ in~$X$ for all $x\in X$ and $s\ge0$.

\item\label{item:density_p_t_symmetry} 
$\p_t[x](y)=\p_t[y](x)$ for $\m$-a.e.\ $x,y\in X$.

\end{enumerate}	
\end{lemma} 

\begin{proof}
Let $t>0$ be fixed.
We prove the two statements separately.

\smallskip

\textit{Proof of~\ref{item:density_p_t_semigroup}}.
Let $\phi\in\Leb^1(X,\m)\cap\Leb^\infty(X,\m)$ be a Borel non-negative function and set $\phi=\bar\phi\,\|\phi\|_{\Leb^1(X,\m)}$.
We can compute
\begin{align*}
\int_X\phi\,\p_{s+t}[y]\di\m
&=
\int_X\phi\di\H_{s+t}\delta_y
&&
\text{by~\eqref{eq:pointwise_heat_p_t_density}, \eqref{eq:dual_heat_is_semigroup} and \cref{res:dual_heat_Lip_ext}}\\
&=
\int_X\tilde\P_s\phi\di\H_t\delta_y
&&
\text{by~\eqref{eq:pointwise_heat_def}}\\
&=
\int_X\P_s\phi\,\p_t[y]\di\m
&&
\text{by \cref{res:pointwise_heat_props}\ref{item:pointwise_heat_on_L_infty_Borel} and~\eqref{eq:pointwise_heat_p_t_density}}\\
&=
\|\phi\|_{\Leb^1(X,\m)}
\int_X\p_t[y]\di\H_s(\bar\phi\,\m)
&&
\text{by~\eqref{eq:dual_heat_def}}\\
&=
\int_X\tilde\P_s(\p_t[y])\,\phi\di\m
&&
\text{by \cref{res:pointwise_heat_props}\ref{item:pointwise_heat_duality_on_Prob_p} and~\eqref{eq:pointwise_heat_p_t_density}}	
\end{align*}
for all $s\ge0$, so that~\ref{item:density_p_t_semigroup} immediately follows.

\smallskip

\textit{Proof of~\ref{item:density_p_t_symmetry}}.
Let $\phi,\psi\in\Leb^1(X,\m)\cap\Leb^\infty(X,\m)$ be two Borel non-negative functions.
By Tonelli Theorem, \eqref{eq:pointwise_heat_p_t_density} and \cref{res:pointwise_heat_props}\ref{item:pointwise_heat_on_L_infty_Borel}, we can compute
\begin{align*}
\int_X\int_X\phi(x)\,\psi(y)\,\p_t[x](y)\di\m(x)\di\m(y)
&=
\int_X\phi(x)\int_X\psi(y)\,\p_t[x](y)\di\m(y)\di\m(x)\\
&=
\int_X\phi(x)\,\tilde\P_t\psi(x)\di\m(x)\\
&=
\int_X\tilde\P_t\phi(y)\,\psi(y)\di\m(y)\\
&=
\int_X\psi(y)\int_X\phi(x)\,\p_t[y](x)\di\m(x)\di\m(y)\\
&=
\int_X\int_X\phi(x)\,\psi(y)\,\p_t[y](x)\di\m(x)\di\m(y),
\end{align*}
so that~\ref{item:density_p_t_symmetry} immediately follows.
\end{proof}

\subsection{\texorpdfstring{$\BE_w$}{BEw} inequality for Lipschitz functions, again}

We conclude this section with the following result, which provides a refined version of \cref{res:lip_wBE} and \cref{res:wBE-like}, see the proof of~\cite{AGS15}*{Theorem~3.17}. 
We give its proof below for the reader's convenience. 

\begin{proposition}[$\BE_w$ for $\Lip$-functions, III]
\label{res:lip_wBE_best}
Assume $(X,\d,\m)$ satisfies $\wBE(\c,\infty)$. 
If $f\in\Lip_b(X)\cap\Sob^{1,2}(X,\d,\m)$, then $\P_tf\in\Lip_b(X)\cap\Sob^{1,2}(X,\d,\m)$ with
\begin{equation*}
|\D^*\P_tf|^2(x)\le\c^2(t)\,\tilde\P_t\Gamma(f)(x)
\quad
\text{for all}\ x\in X
\end{equation*}
whenever $t>0$.	
\end{proposition}

\begin{proof}
Let $t>0$ be fixed.
By \cref{res:lip_wBE} we already know that $\P_tf\in\Lip_b(X)\cap\Sob^{1,2}(X,\d,\m)$.
From~\eqref{eq:def_wBE},
\eqref{eq:wD_less_slope_for_Lip},
\cref{res:pointwise_heat_props}\ref{item:pointwise_heat_on_L_infty_Borel} and \eqref{eq:pointwise_heat_p_t_density} we get that
\begin{equation*}
\Gamma(\P_t f)
\le
\c^2(t)\,\P_t\Gamma(f)
=
\c^2(t)\,\tilde\P_t\Gamma(f)
\quad
\text{$\m$-a.e.\ in $X$}.
\end{equation*} 
Since $\Gamma(f)\in\Leb^1(X,\m)\cap\Leb^\infty(X,\m)$, by \cref{res:S-Feller} and again~\eqref{eq:pointwise_heat_p_t_density} we must have that $\tilde\P_t\Gamma(f)\in\Cont_b(X)$.
The conclusion thus follows from \cref{res:reverse_slope_estim}.
\end{proof}

\begin{corollary}[Weak reverse slope estimate for $(\P_t)_{t>0}$]
\label{res:Fellerazzo}
Assume $(X,\d,\m)$ satisfies $\wBE(\c,\infty)$.
If $f\in\Leb^2(X,\m)\cap\Leb^\infty(X,\m)$, then
\begin{equation}\label{eq:Fellerazzo}
|\D^*\P_t f|^2
\le
\underline\c(0^+)^2
\,
\Gamma(\P_t f)
\quad
\text{$\m$-a.e.\ in~$X$}
\end{equation}	
for all $t>0$, where $\underline\c(0^+)=\liminf\limits_{s\to0^+}\c(s)\ge1$.
\end{corollary}

\begin{proof}
Let $t>0$ be fixed and let $(\eps_n)_{n\in\N}\subset(0,t)$ be such that $\eps_n\downarrow0$ as $n\to+\infty$ and $\underline\c(0^+)=\lim\limits_{n\to+\infty}\c(\eps_n)$.
By \cref{res:S-Feller}, we have $f_n=\P_{t-\eps_n}f\in\Lip_b(X)\cap\Sob^{1,2}(X,\d,\m)$ for all $n\in\N$ and thus, by \cref{res:lip_wBE_best}, \cref{res:pointwise_heat_props}\ref{item:pointwise_heat_on_L_infty_Borel} and~\eqref{eq:pointwise_heat_p_t_density}, we can estimate
\begin{equation*}
|\D^*\P_t f|^2
=
|\D^*\P_{\eps_n} f_n|^2
\le
\c^2(\eps_n)\,\tilde\P_{\eps_n}\Gamma(f_n)
=
\c^2(\eps_n)\,\P_{\eps_n}\Gamma(f_n)
\quad
\text{$\m$-a.e.\ in~$X$}
\end{equation*}
for all $n\in\N$. Since $\Gamma(f_n)\to\Gamma(f)$ in $\Leb^1(X,\m)$ as $n\to+\infty$ by~\eqref{eq:min_weak_grad_f_t_is_cont} in \cref{res:heat_flow_sob_cont}, the conclusion follows by passing to the limit as $n\to+\infty$.
\end{proof}

As a completely natural (although painful) drawback of the weakness of the $\BE_w(\c,\infty)$ property, if the function~$\c$ in~\eqref{eq:def_c} is such that
\begin{equation}\label{eq:Fellerazzo_salto}
\underline\c(0^+)
=
\liminf_{t\to0^+}\c(t)>1,
\end{equation}
then \cref{res:Fellerazzo} provides no useful information, since $|\D^*\P_t f|\ge|\D\P_t f|\ge|\D\P_t f|_w$ $\m$-a.e.\ in~$X$ whenever~$t>0$ by~\eqref{eq:wD_less_slope_for_Lip} and \cref{res:S-Feller}.
In \cref{sec:proof_equivalence} (precisely, in  the proof of \cref{res:action_double_variation}), similarly to~\cite{AGS15}, we will need the following regularization property of the heat semigroup.

\begin{definition}[Heat-smoothing admissible space]
\label{def:heat-smoothing_space}
We say that an admissible metric-measure space is \emph{heat-smoothing} if
\begin{equation}\label{assumption:heat-smoothing}
f\in\Leb^\infty(X,\m)\cap\Sob^{1,2}(X,\d,\m)
\implies
|\D\P_tf|=|\D\P_tf|_w\
\text{$\m$-a.e.\ in~$X$ for all $t>0$}.
\end{equation}
\end{definition}

\noindent
Note that, if $\underline\c(0^+)=1$, then inequality~\eqref{eq:Fellerazzo} in \cref{res:Fellerazzo} immediately implies~\eqref{assumption:heat-smoothing} (actually, in the stronger form assuming $f\in\Leb^2(X,\m)\cap\Leb^\infty(X,\m)$ only).

\section{Fisher information and \texorpdfstring{$\Leb\log\Leb$}{LlogL}-regolarization}
\label{sec:Fisher_and_LlogL}

In this section, we recall some useful properties of the Fisher information and the entropy functional in admissible metric-measure spaces. 
We only detail the proofs of the results which rely on the $\BE_w(\c,\infty)$ condition. 

\subsection{Fisher information, entropy and Kuwada Lemma}   

Let us set 
\begin{equation*}
\Leb^1_+(X,\m)=\set*{f\in\Leb^1(X,\m) : f\ge0\ \text{$\m$-a.e.\ in}\ X}
\end{equation*}
the convex cone of non-negative $\Leb^1$-functions.
As in~\cite{AGS14}*{Definition~4.9},  
the \emph{Fisher information} $\F\colon\Leb^1_+(X,\m)\to[0,+\infty]$ is defined for all $f\in\Leb^1_+(X,\m)$ as
\begin{equation*}
\F(f)
=
\begin{cases}
4\Ch(\sqrt{f}) & \text{if}\ \sqrt{\mathstrut f}\in\Sob^{1,2}(X,\d,\m)\\[2mm]
+\infty & \text{otherwise}.
\end{cases}
\end{equation*}
In particular, we have
\begin{equation*}
\Dom(\F)=\set*{f\in\Leb^1_+(X,\m) : \sqrt{\mathstrut f}\in\Sob^{1,2}(X,\d,\m)}.	
\end{equation*}
Since $f_n\to f$ in~$\Leb^1_+(X,\m)$ as $n\to+\infty$ implies that $\sqrt{\mathstrut f_n}\to\sqrt{\mathstrut f}$ in~$\Leb^2(X,\m)$ as $n\to+\infty$, the Fisher information~$\F$ is lower semicontinuous in~$\Leb^1_+(X,\m)$.
Thanks to the locality property~\eqref{eq:locality_min_weak_grad} and the chain rule~\eqref{eq:chain_rule_min_weak_grad}, if $f\in\Dom(\F)$ then $f_n=f\wedge n\in\Sob^{1,2}(X,\d,\m)$ with
\begin{equation*}
|\D f_n|_w
=
2\sqrt{f_n}\,|\D\sqrt{f}|_w\,\chi_{\set*{f\le n}}
\in\Leb^1(X,\m)
\end{equation*}
for all~$n\in\N$ and 
$|\D f_n|_w\uparrow 2\sqrt{\mathstrut f}\,|\D\sqrt{\mathstrut f}|_w$ 
in~$\Leb^1(X,\m)$ as~$n\to+\infty$.
Hence we can write
\begin{align*}
\F(f)
&=
\lim_{n\to+\infty}
4\int_X|\D\sqrt f|^2_w\,\chi_{\set*{0<f\le n}}\di\m\\
&=
\lim_{n\to+\infty}
\int_X\frac{4f_n\,|\D\sqrt f|^2_w}{f}\,\chi_{\set*{0<f\le n}}\di\m\\
&=
\lim_{n\to+\infty}
\int_{\set*{f>0}}\frac{|\D f_n|^2}{f}\di\m.
\end{align*}
Thus, accordingly with~\cite{AGS14}*{Lemma~4.10}, if $f\in\Dom(\F)$ then we define
\begin{equation}\label{eq:weak_grad_ext}
|\tilde\D f|_w
=
2\sqrt{f}\,|\D\sqrt f|_w
=
\lim_{n\to+\infty}|\D f_n|_w
\end{equation}
and
\begin{equation}\label{eq:Gamma_ext_def_and_Fisher}
\F(f)
=
\int_{\set*{f>0}}\frac{\tilde\Gamma(f)}f\di\m,\
\text{where}\
\tilde\Gamma(f)
=
|\tilde\D f|_w^2.
\end{equation}  
In particular, thanks to the convexity of the weak gradient and the convexity of the function $(x,y)\mapsto y^2/x$ on~$(0,+\infty)\times\R$, we also get that the Fisher information is convex on~$\Leb^1_+(X,\m)$, see~\cite{AGS14}*{Lemma~4.10}.

The following result is a part of the statement of~\cite{AGS15}*{Lemma~4.2} and provides some simple but extremely useful estimates involving the Fisher information, the entropy functional and the second moments of the measure. 
The proof goes as the one of~\cite{AGS15}*{Lemma~4.2} (see also~\cite{AGS14}*{Theorems~4.16(b) and~4.20}) and we thus omit it.

\begin{lemma}[Entropy and Fisher information along $(\H_t)_{t\ge0}$]
\label{res:Fisher_Entropy_along_H_t}
Let $\mu=f\m\in\Dom(\Ent_\m)$ and set $f_t=\P_tf$ and $\mu_t=f_t\m$ for all $t\ge0$.
\begin{enumerate}[label=(\roman*)]

\item\label{item:entropy_decreasing_dual_heat}
For all $t\ge0$, we have
$\Ent_\m(\mu_t)\le\Ent_\m(\mu)$.

\item\label{item:Fisher_moment_est_dual_heat}
If $x_0\in X$ and $T>0$, then 
\begin{equation}
\label{eq:Fisher_moment_est_dual_heat}
\int_0^T\F(f_t)\di t
+
2\int_0^T\int_X\d^2(x,x_0)\di\mu_t(x)\di t
\le
2e^{4T}
\left(
\Ent_\m(\mu)
+
2\int_X\d^2(x,x_0)\di\mu(x)
\right).
\end{equation}
\end{enumerate}
\end{lemma}

Thanks to \cref{res:Fisher_Entropy_along_H_t}, we have the following fundamental result. 
We refer the interested reader to~\cite{AGS14}*{Lemma~6.1} (see also~\cite{GKO13}*{Proposition~3.7}) for its proof. 

\begin{lemma}[Estimate on the $W_2$-velocity]
\label{res:kuwada_lemma_velocity}
If $\mu=f\m\in\Dom(\Ent_\m)$ with $f\in\Leb^2(X,\m)$, then 
$t\mapsto\mu_t=f_t\m\in\AC_{\loc}^2([0,+\infty);\Prob_2(X))$, 
where $f_t=\P_t f$ for all $t\ge0$, with
\begin{equation*}
|\dot\mu_t|^2\le\F(f_t)=\int_{\set*{f_t>0}}\frac{|\D f_t|_w^2}{f_t}\di\m
\end{equation*}
for $\L^1$-a.e.\ $t>0$.
\end{lemma}

As a simple consequence of \cref{res:kuwada_lemma_velocity} and the $\BE_w(\c,\infty)$ condition, we can prove the following $W_2$-continuity property of the dual heat flow.
For the same result under the standard $\BE(K,\infty)$ condition, see the last part of~\cite{AGS15}*{Lemma~4.2}. 

\begin{lemma}[$W_2$-continuity of $t\mapsto\H_t$]
\label{res:dual_heat_W_2_continuity}
Assume $(X,\d,\m)$ satisfies $\wBE(\c,\infty)$.
If $\mu\in\Prob_2(X)$, then $t\mapsto\H_t\mu\in\Cont([0,+\infty);\Prob_2(X))$.
Equivalently, if $\mu\in\Prob_2(X)$ then $t\mapsto\H_t\mu$ is weakly continuous on $[0,+\infty)$ and $t\mapsto\int_X\d^2(x,x_0)\di\H_t\mu(x)$ is continuous on $[0,+\infty)$ whenever $x_0\in X$ is given.
\end{lemma}

\begin{proof}
If $\mu=f\m\in\Dom(\Ent_\m)$ with $f\in\Leb^2(X,\m)$, then the $W_2$-continuity of the map $t\mapsto\H_t\mu$ follows immediately from \cref{res:kuwada_lemma_velocity}.
If $\mu\in\Prob_2(X)$, then we can find $(\mu_n)_{n\in\N}\subset\Prob_2^{\rm ac}(X)$ such that $\mu_n\overset{W_2}{\longto}\mu$ as~$n\to+\infty$.
Possibly performing a truncation argument, we can also assume that $\Ent(\mu_n)<+\infty$ and $\mu_n=f_n\m$ with $f_n\in\Leb^2(X,\m)$ for all $n\in\N$.
If $t\ge0$, then we can estimate
\begin{align*}
\limsup_{s\to t}
W_2(\H_s\mu,\H_t\mu)
&\le
W_2(\H_s\mu_n,\H_s\mu)
+
W_2(\H_t\mu_n,\H_t\mu)
+
\limsup_{s\to t}
W_2(\H_s\mu_n,\H_t\mu_n)\\
&\le
2M\, W_2(\mu_n,\mu),
\end{align*}
where we have set $M=\sup_{s\in[\frac t2,t+1]}\c(s)$. The conclusion thus follows by passing to the limit as~$n\to+\infty$.
\end{proof}

\subsection{Log-Harnack and \texorpdfstring{$\Leb\log\Leb$}{L log L} estimates}

The rest of this section is dedicated to the proof of the following fundamental regularization property of the dual heat semigroup, see~\cite{AGS15}*{Theorem~4.8} for the same result in the standard $\BE(K,\infty)$ setting.

\begin{theorem}[$\Leb\log\Leb$ regularization]\label{res:LlogL_reg}
Assume $(X,\d,\m)$ satisfies $\wBE(\c,\infty)$. 
If $\mu\in\Prob_2(X)$, then
\begin{equation}\label{eq:LlogL_reg}
\Ent_\m(\H_t\mu)
\le 
\frac{1}{2\I_{-2}(t)}
\left(r^2+\int_X\d^2(x,x_0)\di\mu(x)\right)
-\log\m(B_r(x_0))
\end{equation}
for all $x_0\in X$ and $r,t>0$.
In particular, 
$\H_t(\Prob_2(X))\subset\Dom(\Ent_\m)$
for all $t>0$.
\end{theorem}

To prove \cref{res:LlogL_reg}, we follow the same strategy adopted in~\cite{AGS15}*{Section~4.2}.
Before the proof of \cref{res:LlogL_reg}, we need two preliminary results.

The first one is the following generalization of the differentiation formula proved in \cref{res:diff_formula}.
The proof goes as that of~\cite{AGS15}*{Lemma~4.5} with minor modifications, so we omit it.

\begin{lemma}[General differentiation formula]
\label{res:gen_diff_formula_dual_heat}
Assume $(X,\d,\m)$ satisfies $\wBE(\c,\infty)$ 
and let $\omega\in\Cont^2([0,+\infty))$.
If $f\in\Lip_b(X)\cap\Sob^{1,2}(X,\d,\m)$ and $\mu\in\Prob_2(X)$, then for all $t>0$ we have
\begin{equation*}
s\mapsto G(s)
=
\int_X\omega(\P_{t-s}f)\di\H_s\mu
\in\Cont([0,t])\cap\Cont^1((0,t))
\end{equation*}
with
\begin{equation*}
G'(s)
=
\int_X\omega''(\P_{t-s}f)\,\Gamma(\P_{t-s}f)\di\H_s\mu
\end{equation*}
for all $s\in[0,t]$.
\end{lemma}

The second preliminary result is an adaptation to the abstract setting of an inequality proved for the first time in the Riemannian framework by Wang, see~\cite{W11}*{Theorem~1.1(6)}.
\cref{res:wang_eps} below is the reformulation, under the more general $\BE_w(\c,\infty)$ condition, of~\cite{AGS15}*{Lemma~4.6}. 
Although its proof is very similar to that of~\cite{AGS15}*{Lemma~4.6}, we detail it here for the reader's convenience.

\begin{lemma}
[Wang log-Harnack inequality]
\label{res:wang_eps}
Assume $(X,\d,\m)$ satisfies $\wBE(\c,\infty)$ 
and let \mbox{$\eps>0$}.
If $f\in\Leb^1(X,\m)$ is non-negative,
then
\begin{equation}\label{eq:wang_eps}
\tilde{\P}_t(\log(f+\eps))(y)\le\log(\tilde{\P}_tf(x)+\eps)+\frac{\d^2(x,y)}{4\I_{-2}(t)},
\end{equation}
for all $x,y\in X$ and $t>0$, where $\I_{-2}$ is as in \eqref{eq:def_I_p}.
\end{lemma}

\begin{proof}
Let $\eps>0$ and $t>0$ be fixed.
We divide the proof in three steps.

\smallskip

\textit{Step~1}. 
Assume $f\in\Lip_b(X)\cap\Sob^{1,2}(X,\d,\m)$.
Fix $x,y\in X$ and let $\gamma\in\AC([0,1];\R)$ be such that $\gamma_0=x$ and $\gamma_1=y$.
We set
\begin{equation*}
\theta(r)=\frac{\I_{-2}(r)}{\I_{-2}(t)}\in[0,1]
\quad
\text{for all}\ r\in[0,t].
\end{equation*}
Note that $\theta\in\Lip([0,t])$. 
We also set $\omega_\eps(r)=\log(r+\eps)-\log\eps$ for all $r\ge0$.
Note that $\omega\in\Cont^2([0,+\infty))$ with $\omega(0)=0$. 
We claim that 
\begin{equation}\label{eq:macarena}
s\mapsto\int_X\omega_\eps(\P_{t-s}f)\di\H_s\delta_{\gamma_{\theta(s)}}
\in\AC([0,t];\R).
\end{equation}
To this aim, we set
\begin{equation*}
G_\eps(s,r)
=
\int_X\omega_\eps(\P_{t-s}f)\di\H_s\delta_{\gamma_{\theta(r)}}
\quad
\text{for all}\ s,r\in[0,t].
\end{equation*}
On the one hand, by \cref{res:gen_diff_formula_dual_heat} applied with $\mu=\delta_{\gamma_{\theta(r)}}$ for each $r\in[0,t]$, we get that 
\begin{equation*}
|\de_s G_\eps(s,r)|
=
\bigg|
\int_X\omega_\eps''(\P_{t-s}f)\,\Gamma(\P_{t-s}f)\di\H_s\delta_{\gamma_{\theta(r)}}
\bigg|
\le
\frac{\c(t-s)}{\eps^2}\,\Lip(f)
\end{equation*}
for all $s\in(0,t)$ and $r\in[0,t]$ by~\eqref{eq:wD_less_slope_for_Lip} and \cref{res:lip_wBE}, so that $s\mapsto G(s,r)$ is Lipschitz on~$[0,t]$ uniformly in $r\in[0,t]$.
On the other hand, by~\eqref{eq:kantorovich_duality_p=1}, \cref{res:lip_wBE}, Jensen inequality and \cref{res:kuwada_equivalence}, we can estimate
\begin{align*}
|G_\eps(s,r_1)-G_\eps(s,r_0)|
&=
\bigg|
\int_X\omega_\eps(\P_{t-s}f)\di\big(\H_s\delta_{\gamma_{\theta(r_1)}}-\H_s\delta_{\gamma_{\theta(r_0)}}\big)
\bigg|\\
&\le
\Lip(\omega_\eps(\P_{t-s}f))\,
W_1\big(\H_s\delta_{\gamma_{\theta(r_1)}},\H_s\delta_{\gamma_{\theta(r_0)}}\big)\\
&\le
\frac{\c(t-s)}\eps\Lip(f)\,
W_2\big(\H_s\delta_{\gamma_{\theta(r_1)}},\H_s\delta_{\gamma_{\theta(r_0)}}\big)\\
&\le
\frac{\c(t-s)\c(s)}\eps\Lip(f)\,
W_2\big(\delta_{\gamma_{\theta(r_1)}},\delta_{\gamma_{\theta(r_0)}}\big)\\
&\le
\frac{\c(t-s)\c(s)}\eps\Lip(f)\,
\d(\gamma_{\theta(r_1)},\gamma_{\theta(r_0)})
\end{align*}
for all $0\le r_0\le r_1\le t$ and $s\in[0,t]$, so that $r\mapsto G(s,r)\in\AC([0,t];\R)$ uniformly in $s\in[0,t]$.
This prove the claim in~\eqref{eq:macarena}.
Now write
\begin{equation*}
G_\eps(s,r)=F_s^\eps(\gamma_{\theta(r)}),
\quad
F_s^\eps(x)=\tilde\P_s\big(\omega_\eps(\P_{t-s}f)\big)(x)\
\text{for all}\ x\in X.
\end{equation*}
Then, whenever $s\in[0,t]$, we can estimate
\begin{equation*}
\de_r G_\eps(s,r)
\le
|\D F^\eps_s|(\gamma_{\theta(r)})
\,
|\dot\gamma_{\theta(r)}|
\,
|\theta'(r)|
\end{equation*}
for $\L^1$-a.e.\ $r\in[0,t]$.
By \cref{res:lip_wBE_best}, we have
\begin{equation}\label{eq:umpa}
|\D F^\eps_s|^2(\gamma_{\theta(r)})
=
\big|\D\tilde\P_s\big(\omega_\eps(\P_{t-s}f)\big)\big|^2(\gamma_{\theta(r)})
\le
\c^2(s)
\,
\tilde\P_s\Gamma\big(\omega_\eps(\P_{t-s}f)\big)(\gamma_{\theta(r)})
\end{equation}
for all $r\in[0,t]$.
Recalling~\eqref{eq:def_p_t_density}, by the chain rule~\eqref{eq:chain_rule_min_weak_grad} we can write
\begin{equation}\label{eq:lumpa}
\begin{split}
\tilde\P_s\Gamma\big(\omega_\eps(\P_{t-s}f)\big)(\gamma_{\theta(r)})
&=
\int_X\Gamma\big(\omega_\eps(\P_{t-s}f)\big)\di\H_s\gamma_{\theta(r)}\\
&=
\int_X\Gamma\big(\omega_\eps(\P_{t-s}f)\big)\,\p_s[\gamma_{\theta(r)}]\di\m\\
&=
\int_X (\omega_\eps'(\P_{t-s}f))^2\,\Gamma(\P_{t-s}f)\,\p_s[\gamma_{\theta(r)}]\di\m\\
&=
-\int_X \omega_\eps''(\P_{t-s}f)\,\Gamma(\P_{t-s}f)\di\H_s\gamma_{\theta(r)}
\end{split}
\end{equation}
for all $s\in(0,t]$ and $r\in[0,t]$, since $(\omega_\eps')^2=-\omega_\eps''$.
Therefore, by Young inequality, we get that
\begin{align*}
\de_r G_\eps(s,r)
\le
\c^{-2}(s)\,|\D F^\eps_s|^2(\gamma_{\theta(r)})
+
\frac{\c^2(s)}4
\,
|\dot\gamma_{\theta(r)}|^2
\,
|\theta'(r)|^2
\end{align*}
for all $s\in(0,t]$ and $\L^1$-a.e.\ $r\in[0,t]$.
By combining~\eqref{eq:umpa} with~\eqref{eq:lumpa}, we conclude that
\begin{align*}
\frac{\di}{\di s}
\int_X\omega_\eps(\P_{t-s}f)\di\H_s\delta_{\gamma_{\theta(s)}}
&=
\de_s G_\eps(s,s)
+
\de_r G_\eps(s,s)\\
&\le
\int_X\omega_\eps''(\P_{t-s}f)\,\Gamma(\P_{t-s}f)\di\H_s\delta_{\gamma_{\theta(s)}}\\
&\quad+
\c^{-2}(s)\,|\D F^\eps_s|^2(\gamma_{\theta(s)})
+
\frac{\c^2(s)}4
\,
|\dot\gamma_{\theta(s)}|^2
\,
|\theta'(s)|^2\\
&\le
\frac{\c^2(s)}4
\,
|\dot\gamma_{\theta(s)}|^2
\,
|\theta'(s)|^2
\end{align*}
for $\L^1$-a.e.\ $s\in[0,t]$, so that
\begin{align*}
\int_X\omega_\eps(f)\di\H_t\delta_y
-
\int_X\omega_\eps(\P_t f)\di\delta_x
&=
\int_0^t\frac{\di}{\di s}
\int_X\omega_\eps(\P_{t-s}f)\di\H_s\delta_{\gamma_{\theta(s)}}
\di s\\
&\le
\frac14\int_0^t\c^2(s)
\,
|\dot\gamma_{\theta(r)}|^2
\,
|\theta'(s)|^2
\di s\\
&=
\frac1{4\I_{-2}(t)}\int_0^1|\dot\gamma_\tau|^2\di\tau.
\end{align*}
Recalling \cref{res:length_space}, we immediately deduce that
\begin{equation}\label{eq:step_1_wang}
\tilde\P_t(\omega_\eps(f))(y)
\le
\omega_\eps(\tilde\P_tf)(x)
+
\frac{\d^2(x,y)}{4\I_{-2}(t)},
\end{equation}
for all $x,y\in X$, whenever $f\in\Lip_b(X)\cap\Sob^{1,2}(X,\d,\m)$.

\smallskip

\textit{Step~2}. 
Assume $f\in\Leb^1(X,\m)\cap\Leb^\infty(X,\m)$.
Since $\Sob^{1,2}(X,\d,\m)$ is dense in $\Leb^2(X,\m)$, by~\eqref{eq:approx_Sob_by_Lip_b} we can find $(f_n)_{n\in\N}\subset\Lip_b(X)\cap\Sob^{1,2}(X,\d,\m)$ such that $f_n\to f$ $\m$-a.e.\ in~$X$ as~$n\to+\infty$.
Since $f\in\Leb^\infty(X,\m)$ is non-negative, by~\eqref{eq:min_is_sob} and~\eqref{eq:max_is_sob} we can also assume that $0\le f_n\le\|f\|_{\Leb^\infty(X,\m)}$ for all $n\in\N$. 
By~\eqref{eq:pointwise_heat_p_t_density} and the Dominated Convergence Theorem, we thus get that $\tilde\P_t f_n(x)\to\tilde\P_t f(x)$ for all $x\in X$ as $n\to+\infty$. 
Hence, by Fatou Lemma and by~\eqref{eq:step_1_wang} in Step~1, we get
\begin{align*}
\tilde\P_t(\omega_\eps(f))(y)
&\le
\liminf_{n\to+\infty}
\tilde\P_t(\omega_\eps(f_n))(y)\\
&\le
\lim_{n\to+\infty}
\omega_\eps(\tilde\P_tf_n)(x)
+
\frac{\d^2(x,y)}{4\I_{-2}(t)}\\
&=
\omega_\eps(\tilde\P_tf)(x)
+
\frac{\d^2(x,y)}{4\I_{-2}(t)},
\end{align*}
for all $x,y\in X$, proving~\eqref{eq:step_1_wang} whenever $f\in\Leb^1(X,\m)\cap\Leb^\infty(X,\m)$.

\smallskip

\textit{Step~3}. 
Assume $f\in\Leb^1(X,\m)$.
Then $f_n=f\wedge n\in\Leb^1(X,\m)\cap\Leb^\infty(X,\m)$ for all $n\in\N$ and thus the conclusion follows by Step~2 and the Monotone Convergence Theorem. 
\end{proof}

An simple but interesting consequence of \cref{res:wang_eps} is the following result, see~\cite{AGS15}*{Corollary~4.7} for the same result in the standard $\BE(K,\infty)$ setting.

\begin{corollary}[Wang inequality for~$\p_t{[}\cdot{]}$]
\label{res:wang_p_t}
Assume $(X,\d,\m)$ satisfies $\wBE(\c,\infty)$. 
For every $\eps,t>0$ and every $y\in X$, we have
\begin{equation*}
\int_X \p_t[y]\log(\p_t[y]+\eps)\di\m
\le
\log(\p_{2t}[y](x)+\eps)
+
\frac{\d^2(x,y)}{4\I_{-2}(t)}
\end{equation*}
for $\m$-a.e.\ $x\in X$.
In particular, if $\m\in\Prob(X)$, then for all $t>0$ and $y\in X$ we have
\begin{equation}\label{eq:wang_p_2t}
\p_{2t}[y](x)\ge\exp\left(-\frac{\d^2(x,y)}{4\I_{-2}(t)}\right)
\end{equation}
for $\m$-a.e.\ $x\in X$.
\end{corollary}

\begin{proof}
Thanks to \cref{res:density_p_t_props}\ref{item:density_p_t_semigroup}, the result immediately follows by applying \cref{res:wang_eps} to $f=\p_t[y]\in\Leb^1(X,\m)$ and passing to the limit as $\eps\to0^+$.
\end{proof}

We are now ready to prove  the main result of this section.

\begin{proof}[Proof of \cref{res:LlogL_reg}]
Let $x_0\in X$ and $r,t>0$ be fixed.
We divide the proof in two steps.

\smallskip

\textit{Step~1}.
Assume $\mu=f\m$ for some non-negative $f\in\Leb^1(X,\m)$.
Let us set $f_t=\P_tf$ and $\tilde f_t=\tilde\P_tf$.
Note that $\tilde f_t=f_t$ $\m$-a.e.\ in~$X$ since, as in the proof of \cref{res:wang_p_t}, we can compute
\begin{align*}
\int_X \tilde f_t\,\phi\di\m
=
\|\phi\|_{\Leb^1(X,\m)}
\int_Xf\di\H_t(\bar\phi\,\m)
=
\int_X f\,\P_t\phi\di\m
=
\int_X f_t\,\phi\di\m
\end{align*} 
for all non-negative Borel $\phi\in\Leb^1(X,\m)\cap\Leb^\infty(X,\m)$, where $\phi=\bar\phi\,\|\phi\|_{\Leb^1(X,\m)}$.
By~\eqref{eq:def_p_t_density} and
Jensen inequality, we can estimate
\begin{equation}\label{eq:panza}
\begin{split}
\tilde f_t(x)\log(\tilde f_t(x)+\eps)
&=
\left(\int_X f\di\H_t\delta_x\right)
\log\left(\eps+\int_X f\di\H_t\delta_x\right)\\
&=
\left(\int_X \p_t[x]\di\mu\right)
\log\left(\eps+\int_X \p_t[x]\di\mu\right)\\
&\le
\int_X \p_t[x](y)\log(\p_t[x](y)+\eps)\di\mu(y)
\end{split}
\end{equation}
for all $x\in X$ and $\eps>0$.
Integrating~\eqref{eq:panza}, by \cref{res:density_p_t_props}\ref{item:density_p_t_symmetry}, Tonelli Theorem and \cref{res:wang_p_t} we thus get
\begin{equation}\label{eq:ripanza}
\begin{split}
\Ent_\m(\H_t\mu)
=
\int_X f_t\log(f_t+\eps)\di\m
&\le
\int_X\left(\int_X \p_t[x](y)\log(\p_t[x](y)+\eps)\di\mu(y)\right)\di\m(x)\\
&=
\int_X\left(\int_X \p_t[y](x)\log(\p_t[y](x)+\eps)\di\m(x)\right)\di\mu(y)\\
&\le
\int_X
\left(\log(\p_{2t}[y](z)+\eps)
+
\frac{\d^2(z,y)}{4\I_{-2}(t)}\right)
\di\mu(y)
\end{split}
\end{equation}
for $\m$-a.e.\ $z\in X$.
Now let $q=\m(B_r(x_0))$ and define $\nu=q^{-1}\m\mres B_r(x_0)$.
Integrating~\eqref{eq:ripanza}, by Tonelli Theorem and Jensen inequality we get 
\begin{align*}
\Ent_\m(&\H_t\mu)
\le
\int_X\int_X
\left(\log(\p_{2t}[y](z)+\eps)
+
\frac{\d^2(z,y)}{4\I_{-2}(t)}\right)\di\mu(y)\di\nu(z)\\
&=
\int_X\int_X
\log(\p_{2t}[y](z)+\eps)\di\mu(y)\di\nu(z)
+
\frac{1}{2\I_{-2}(t)}
\left(r^2+\int_X\d^2(y,x_0)\di\mu(y)\right)\\
&\le
\log\left(
\eps
+
\int_X\int_X
\p_{2t}[y](z)\di\nu(z)\di\mu(y)
\right)
+
\frac{1}{2\I_{-2}(t)}
\left(r^2+\int_X\d^2(y,x_0)\di\mu(y)\right)\\
&=
\log\left(
\eps
+
\int_X \frac1q\int_{B_r(x_0)}
\p_{2t}[y](z)\di\m(z)\di\mu(y)
\right)
+
\frac{1}{2\I_{-2}(t)}
\left(r^2+\int_X\d^2(y,x_0)\di\mu(y)\right)\\
&\le
\log(q^{-1}+\eps)
+
\frac{1}{2\I_{-2}(t)}
\left(r^2+\int_X\d^2(y,x_0)\di\mu(y)\right)
\end{align*} 
Passing to the limit as $\eps\to0^+$, we prove~\eqref{eq:LlogL_reg} whenever $\mu=f\m$ for some non-negative $f\in\Leb^1(X,\m)$.

\smallskip
\textit{Step~2}.
Now let $\mu\in\Prob_2(X)$. We can find $(\mu_n)_{n\in\N}\subset\Prob_2^{\rm ac}(X)$ such that $\mu_n\overset{W_2}{\longto}\mu$ as $n\to+\infty$.
Thanks to the lower semicontinuity property of the entropy and the properties of Wasserstein distance, by Step~1 we get that
\begin{align*}
\Ent_\m(\H_t\mu)
&\le
\liminf_{n\to+\infty}
\Ent_\m(\H_t\mu_n)\\
&\le 
\frac{1}{2\I_{-2}(t)}
\left(r^2+\lim_{n\to+\infty}\int_X\d^2(x,x_0)\di\mu_n(x)\right)
-\log\m(B_r(x_0))\\
&=
\frac{1}{2\I_{-2}(t)}
\left(r^2+\int_X\d^2(x,x_0)\di\mu(x)\right)
-\log\m(B_r(x_0))
\end{align*} 
and the proof is complete.
\end{proof}

\section{Entropic inequalities in groups}
\label{sec:proof_equivalence}

We now proceed with the core argument of the proof of the entropic inequalities, adapting the action and the entropy estimates established in~\cite{AGS15}*{Section~4.3} (see also~\cite{EKS15}*{Section~4.2} for a closely related approach in the finite dimensional case). Since we frequently consider curves $s\mapsto\mu_s=f_s\m\in\AC^2([0,1];(\Prob_2(X),W_2))$ with $s\mapsto f_s\in\Cont^1([0,1];\Leb^p(X,\m))$ for some $p\in[1,+\infty)$, we shall denote by $s\mapsto\dot{f}_s\in\Cont([0,1];\Leb^p(X,\m))$ the functional derivative in~$\Leb^p(X,\m)$, keeping the notation $s\mapsto|\dot{\mu}_s|$ for the metric derivative in~$(\Prob_2(X),W_2)$.

\subsection{Strongly regular curves}
\label{subsec:strong_reg_curves}
 
Instead of considering \emph{regular curves} in $\Prob_2(X)$ as in~\cite{AGS15}*{Definition~4.10}, we deal with \emph{strongly regular curves} defined as follows.

\begin{definition}[Strongly regular curve in $\Prob_2(X)$]
\label{def:strong_reg_curve}
We say that a curve 
$s\mapsto\mu_s\in\AC^2([0,1];\Prob_2(X))$
is \emph{strongly regular} if $\mu_s=f_s\m$ for all $s\in[0,1]$ with
\begin{equation}\label{eq:def_strong_reg_curve_density_C1_L2}
s\mapsto f_s\in\Cont^1([0,1];\Leb^2(X,\m)).
\end{equation}
\end{definition}

Note that the $\Leb^2$-integrability property in~\eqref{eq:def_strong_reg_curve_density_C1_L2} immediately gives that
\begin{equation}\label{eq:strong_reg_curve_finite_Ent}
\sup_{s\in[0,1]}\Ent_\m(\mu_s)<+\infty
\end{equation}
whenever $s\mapsto\mu_s$ is a strongly regular curve.
Under the standard $\BE(K,N)$ condition, the uniform upper control of the entropy along \emph{regular curves} and, most importantly, the absolute continuity property of \emph{regular curves} at each time with respect to the reference measure~$\m$, are gained from the absolute continuity of heated measures and the $\Leb\log\Leb$-regularization property of the dual heat flow~$(\H_t)_{t\ge0}$. 
Having the more general $\BE_w(\c,\infty)$ condition at disposal, the absolute continuity of heated measures and the $\Leb\log\Leb$-regularization property are still available, recall \cref{res:dual_heat_ll_m} and \cref{res:LlogL_reg}, but the application of the dual heat flow to an arbitrary curve $s\mapsto\mu_s\in\AC^2([0,1];\Prob_2(X))$ drastically affects its Wasserstein velocity.
In more precise terms, by~\eqref{eq:dual_heat_weak_kuwada} we immediately get that the heated curve 
$s\mapsto\mu_s^t=\H_t\mu_s\in\AC^2([0,1];\Prob_2(X))$
satisfies
\begin{equation}\label{eq:heated_curve_velocity}
|\dot\mu_s^t|\le\c(t)\,|\dot\mu_s|
\quad
\text{for $\L^1$-a.e.}\ s\in[0,1]
\end{equation}
for all $t\ge0$, but, since the general strategy developed in~\cites{AGS15,EKS15} requires the approximation of $s\mapsto\mu_s$ with more regular $\AC^2$-curves in $\Prob_2(X)$ having Wasserstein velocity closer and closer to the velocity of the original curve (see equation~(4.20) in the statement of~\cite{AGS15}*{Proposition~4.11}), the velocity estimate in~\eqref{eq:heated_curve_velocity} becomes useful only if (at least) 
\begin{equation*}
\underline{\c}(0^+)
=
\liminf_{t\to+\infty}
\c(t)
=1,
\end{equation*}
a condition which is not available in sub-Riemannian manifolds.
For this reason, taking inspiration from the regularization procedure performed in~\cite{AS19}*{Theorem~4.8}, instead of relying on the contraction property of the dual heat flow, in our group-modeled framework (see \cref{subsec:admissible_groups} below) we will regularize arbitrary $\AC^2$-curves in~$\Prob_2(X)$ via the group (left-)convolution with suitable $\Leb^1$-kernels (see also \cref{rem:right_convolution} below for a strictly related discussion).

Since our strategy will be deeply based on the $\Leb^2$-regularity property~\eqref{eq:def_strong_reg_curve_density_C1_L2} of strongly regular curves, we will frequently take advantage of the following product rule for the functional derivative of Lipschitz curves in $\Leb^2(X,\m)$. We leave its elementary proof to the interested reader.

\begin{lemma}[Product rule for $\Lip$-curves in $\Leb^2$]
\label{res:Lip_functions_in_L2}
If $s\mapsto a_s,b_s\in\Lip([0,1];\Leb^2(X,\m))$, then 
$s\mapsto \int_X a_s\,b_s\di\m\in\Lip([0,1];\R)$
with
\begin{equation*}
\frac{\di}{\di s}
\int_X a_s\,b_s\di\m
=
\int_X \dot a_s\,b_s\di\m
+
\int_X a_s\,\dot b_s\di\m
\end{equation*}
for $\L^1$-a.e.\ $s\in[0,1]$.
\end{lemma}

\subsection{Action along strongly regular curves}

We begin by fixing two important functions. 
Here and in the rest of the section, we let
\begin{equation*}
\theta\in\Cont^1([0,1];[0,1])\
\text{with}\
\theta(i)=i,\ i=0,1,
\end{equation*} 
and
\begin{equation*}
\eta\in\Cont^2([0,1];[0,+\infty))\
\text{with}\
\dot\eta\ge0\
\text{and}\
\eta(s)>0\
\text{for all}\ 
s>0.
\end{equation*}

In the following result, analogous to~\cite{AGS15}*{Lemma~4.13} and~\cite{EKS15}*{Lemma~4.12}, we compute the derivative of the action along a strongly regular curve. 
Recall that $\Lip_\star(X)$ was defined in~\eqref{eq:Lip_star} as the set of non-negative Lipschitz functions with bounded support.

\begin{lemma}[Derivative of action along strongly regular curves]
\label{res:action_double_variation}
Assume $(X,\d,\m)$ satisfies $\wBE(\c,\infty)$. 
Let $s\mapsto\mu_s=f_s\m\in\AC^2([0,1];\Prob_2(X))$ be a strongly regular curve
and define
\begin{equation}\label{eq:double_var_reg_curve}
s\mapsto\tilde{\mu}_s
=
\H_{\eta(s)}\mu_{\theta(s)}
=
\tilde{f}_s\m
\in\Cont([0,1];\Prob_2(X)).
\end{equation}
If $\phi\in\Lip_\star(X)$ and $\phi_s=Q_s\phi$ for all $s\in[0,1]$, then
\begin{equation}
\label{eq:double_var_reg_curve_action_AC}
s\mapsto\int_X\phi_s\,\di\tilde{\mu}_s
\in
\Lip([0,1];\R)
\end{equation}
with
\begin{equation}\label{eq:action_double_variation}
\frac{\di}{\di s}
\int_X\phi_s\di\tilde{\mu}_s
=
-
\frac{1}{2}\int_X|\D\phi_s|^2\di\tilde{\mu}_s
-
\dot{\eta}(s)\int_X\Gamma(\tilde{f}_s,\phi_s)\di\m
+
\dot{\theta}(s)\int_X\dot{f}_{\theta(s)}\,\P_{\eta(s)}\phi_s\di\m
\end{equation}
for $\L^1$-a.e.\ $s\in(0,1)$.
\end{lemma}

Note that the function~$\eta$ equals $\eta(s)=st$ for all $s\in[0,1]$, $t\ge0$, in~\cite{AGS15}*{Lemma~4.13}, while in~\cite{EKS15}*{Lemma~4.12} $\eta$ is an increasing $\Cont^1$ time-change satisfying $\eta(0)=0$ depending on the dimension $N\in(0,+\infty)$.
For similar variations of curves via semigroup operators, we refer the reader to~\cites{DS08,OW05}. 

In the proof of \cref{res:action_double_variation}, having in mind the product rule provided by \cref{res:Lip_functions_in_L2}, we will use the following result about the $\Leb^2$-functional derivative of the Hopf--Lax semigroup.
For the reader's convenience, we give a sketch of its proof (see also the discussion in the first three paragraphs of the proof of~\cite{AGS14}*{Lemma~6.1}).

\begin{lemma}[Hopf--Lax semigroup is $\Lip$ in $\Leb^2$]
\label{res:hopf--lax_in_L2}
If $f\in\Lip_\star(X)$, then $s\mapsto Q_s f\in\Lip([0,+\infty);\Leb^2(X,\m))$ with
\begin{equation}\label{eq:hopf--lax_in_L2}
\frac{\di^{\,+}}{\di s}\,Q_s f
=
-\frac12\,|\D Q_s f|
\quad
\text{in}\ \Leb^2(X,\m)
\end{equation}
for $\L^1$-a.e.\ $s>0$.
\end{lemma}

\begin{proof}
If $f\in\Lip_\star(X)$, then also $Q_s f\in\Lip_\star(X)$ with $\supp(Q_s f)\subset\supp f$ for all~\mbox{$s\ge0$}.
Hence by~\eqref{eq:hopf-lax_lip} we can estimate
\begin{equation*}
\|Q_{s+h}f-Q_sf\|_{\Leb^2(X,\m)}
\le
\sqrt{\m(\supp f)}\cdot 2\Lip(f)^2\,h
\end{equation*}
for all $0<s\le s+h$, so that the conclusion follows by~\eqref{eq:hopf-lax_HJ} and the Dominated Convergence Theorem.
\end{proof}

\begin{proof}[Proof of \cref{res:action_double_variation}]
By \cref{res:kuwada_equivalence}\ref{item:kuwada_equiv_2}, we can estimate
\begin{align*}
W_2(\tilde{\mu}_{s_0},\tilde{\mu}_{s_1})
&=
W_2(\H_{\eta(s_0)}\mu_{\theta(s_0)},
\H_{\eta(s_1)}\mu_{\theta(s_1)})\\
&\le
W_2(\H_{\eta(s_0)}\mu_{\theta(s_0)},
\H_{\eta(s_0)}\mu_{\theta(s_1)})
+
W_2(\H_{\eta(s_0)}\mu_{\theta(s_1)},
\H_{\eta(s_1)}\mu_{\theta(s_1)})
\\
&\le
\c(\eta(s_0))\, 
W_2(\mu_{\theta(s_0)},\mu_{\theta(s_1)})
+
W_2(\H_{\eta(s_0)}\mu_{\theta(s_1)},
\H_{\eta(s_1)}\mu_{\theta(s_1)})
\end{align*}
for all $s_0,s_1\in[0,1]$, so that~\eqref{eq:double_var_reg_curve} readily follows from \cref{res:dual_heat_W_2_continuity}.
Since we have $s\mapsto f_s\in\Cont^1([0,1];\Leb^2(X,\m))$ by~\eqref{eq:def_strong_reg_curve_density_C1_L2} in \cref{def:strong_reg_curve}, 
we also have $s\mapsto\tilde f_s=\P_{\eta(s)}f_{\theta(s)}\in\Cont^1((0,1];\Leb^2(X,\m))$
with
\begin{equation}\label{eq:tilde_f_s_deriv_L2}
\frac{\di}{\di s}\,\tilde f_s
=
\dot\eta(s)\,
\Delta_{\d,\m}\P_{\eta(s)}f_{\theta(s)}
+
\dot\theta(s)\,
\P_{\eta(s)}\dot f_{\theta(s)}
\quad
\text{in}\ \Leb^2(X,\m)
\end{equation}
for all $s\in(0,1]$.
Hence, by \cref{res:Lip_functions_in_L2} and \cref{res:hopf--lax_in_L2}, we get~\eqref{eq:double_var_reg_curve_action_AC} and we can compute
\begin{align*}
\frac{\di}{\di s}
\int_X \phi_s \di\tilde\mu_s
&=
\frac{\di}{\di s}\int_X \phi_s\,\tilde f_s\di\m
=
\int_X \tilde f_s\,\frac{\di}{\di s}\,\phi_s\di\m
+
\int_X \phi_s\,\frac{\di}{\di s}\,\tilde f_s\di\m\\
&=
-\frac12\int_X|\D\phi_s|\di\tilde\mu_s
+
\int_X
\left(
\dot\eta(s)\,
\Delta_{\d,\m}\tilde f_s
+
\dot\theta(s)\,
\P_{\eta(s)}\dot f_{\theta(s)}
\right)\phi_s\di\m\\
&=
-\frac12\int_X|\D\phi_s|\di\tilde\mu_s
-
\dot\eta(s)
\int_X\Gamma(\tilde f_s,\phi_s)\di\m
+
\dot\theta(s)
\int_X\dot f_{\theta(s)}\,\P_{\eta(s)}\phi_s\di\m
\end{align*}
for $\L^1$-a.e.\ $s\in(0,1)$ by~\eqref{eq:hopf--lax_in_L2} and \eqref{eq:int_by_part_formula}, since $\phi_s\in\Sob^{1,2}(X,\d,\m)$ for all $s\in[0,1]$, proving~\eqref{eq:action_double_variation}.
This concludes the proof.
\end{proof}

\subsection{Entropy along strongly regular curves}

We now introduce some notation.
For every~$\eps>0$, we let
\begin{equation*}
\ell_\eps(r)=\log(\eps+r)
\quad
\text{for all}\ r\ge0,
\end{equation*}
and for all $\mu\in\Prob_2(X)$ we define
\begin{equation}\label{eq:def_truncated_entropy}
E_\eps(\mu)
=
\begin{cases}
\displaystyle\int_X \ell_\eps(f)\di\mu
&
\text{if}\ \mu=f\m,\\[3mm]
+\infty 
&
\text{otherwise}.
\end{cases}
\end{equation}
Note that 
\begin{equation}\label{eq:elio}
E_\eps(\mu)\ge\Ent_\m(\mu)
\quad
\text{for all }\
\mu\in\Prob_2(X)
\end{equation}
and, moreover, that if $\mu=f\m\in\Prob_2(X)$ with $f\in\Leb^2(X,\m)$ then
\begin{equation}\label{eq:E_eps_ok_on_L2}
E_\eps(\mu)
=
\log\eps
+
\int_X (\ell_\eps(f)-\log\eps)\di\mu
\le
\log\eps
+
\frac1\eps\int_X f^2\di\m
<
+\infty.
\end{equation}
For this reason, we set 
\begin{equation*}
\hat\ell_\eps(r)=\ell_\eps(r)-\log\eps
\quad
\text{for all}\ r\ge0.
\end{equation*}
We also let 
\begin{equation*}
p_\eps(r)
=
\hat\ell_\eps(r)+r\,\hat\ell_\eps'(r)
\quad
\text{for all}\ r\ge0.
\end{equation*}
Note that
\begin{equation}\label{eq:cantonese}
p_\eps'(r)
=
2\ell_\eps'(r)+r\,\ell_\eps''(r)
=
\frac{r+2\eps}{(r+\eps)^2}
\quad
\text{for all}\ r\ge0.
\end{equation} 
Finally, note that if $\eps\in(0,1]$ then $\log(\eps+r)\le\log(1+r)\le r$ for all $r\ge0$.
Thus, if $\mu=f\m\in\Prob_2(X)$ for some $f\in\Leb^2(X,\m)$, then 
$[\log(\eps+f)]^-\le[\log f]^-\in\Leb^1(X,\mu)$ 
by~\eqref{eq:exp_growth} and 
$[\log(\eps+f)]^+\le f\in\Leb^1(X,\mu)$, so that
\begin{equation*}
\Ent_\m(\mu)
\le
E_\eps(\mu)
\le
|E_\eps(\mu)|
\le
\int_X[\log f]^-\di\mu+\int_X f^2\di\m
\end{equation*}
and
\begin{equation}\label{eq:masterone}
\Ent_\m(\mu)
=
\lim_{\eps\to0^+}
E_\eps(\mu),
\end{equation}
by the Dominated Convergence Theorem. 

In the following result, analogous to~\cite{AGS15}*{Lemma~4.15} and~\cite{EKS15}*{Lemma~4.13}, we compute the derivative of the truncated entropy $E_\eps$ defined in~\eqref{eq:def_truncated_entropy} along a strongly regular curve.

\begin{lemma}[Derivative of $E_\eps$ along strongly regular curves]\label{res:entropy_double_variation}
Assume $(X,\d,\m)$ satisfies $\wBE(\c,\infty)$
and let $\eps>0$.
Under the same assumptions of \cref{res:action_double_variation} and the notation above, 
we have
\begin{equation}
\label{eq:entropy_double_variation_AC}
s\mapsto E_\eps(\tilde\mu_s)
\in
\Cont^1((0,1];\R)
\end{equation}
with
\begin{equation}\label{eq:entropy_double_variation}
\frac{\di}{\di s}
E_\eps(\tilde{\mu}_s)
\le
-
\dot{\eta}(s)
\int_X\Gamma(g_s^\eps)\di\tilde\mu_s
+
\dot{\theta}(s)
\int_X\dot{f}_{\theta(s)}\,\P_{\eta(s)}g_s^\eps\di\m
\end{equation}
for all $s\in(0,1]$, where 
$g_s^\eps=p_\eps(\tilde f_s)$
for all $s\in[0,1]$. 
\end{lemma}

\begin{proof}
Let $\eps>0$ be fixed.
Since $\hat\ell_\eps\in\Cont^1([0,+\infty);[0,+\infty))\cap\Lip([0,+\infty);[0,+\infty))$ with $\hat\ell_\eps(0)=0$, by~\eqref{eq:tilde_f_s_deriv_L2} and the Mean Value Theorem we get that $s\mapsto\hat\ell_\eps(\tilde f_s)\in\Cont^1((0,1];\Leb^2(X,\m))$ with
\begin{equation*}
\frac{\di}{\di s}\,
\hat\ell_\eps(\tilde f_s)
=
\hat\ell_\eps'(\tilde f_s)\,
\frac{\di}{\di s}\,\tilde f_s
\quad
\text{in}\ \Leb^2(X,\m)
\end{equation*}
for all $s\in(0,1]$.
Similarly, since $p_\eps\in\Cont^1([0,+\infty);\R)\cap\Lip([0,+\infty);\R)$ with \mbox{$p_\eps(0)=0$},  we also have that $s\mapsto g_s^\eps=p_\eps(\tilde f_s)\in\Cont^1((0,1];\Leb^2(X,\m))$. In addition, recalling that $\tilde f_s=\P_{\eta(s)}f_{\theta(s)}$ for all $s\in[0,1]$, by \cref{res:heat_flow_sob_cont} we also have that $s\mapsto\tilde f_s\in\Cont((0,1];\Sob^{1,2}(X,\d,\m))$ and thus also $s\mapsto g_s^\eps\in\Cont((0,1];\Sob^{1,2}(X,\d,\m))$ by the chain rule~\eqref{eq:chain_rule_min_weak_grad}.
Hence, again by~\eqref{eq:tilde_f_s_deriv_L2} and by \cref{res:Lip_functions_in_L2}, we get that
\begin{equation*}
s\mapsto
E_\eps(\tilde\mu_s)
=
\log\eps
+
\int_X \hat\ell_\eps(\tilde f_s)\,\tilde f_s\di\m
\in\Cont^1((0,1];\R),
\end{equation*}
proving~\eqref{eq:entropy_double_variation_AC},
and we can compute
\begin{align*}
\frac{\di}{\di s}\,
E_\eps(\tilde{\mu}_s)
&=
\frac{\di}{\di s}
\int_X\hat\ell_\eps(\tilde f_s)\,\tilde f_s
\di\m\\
&=
\int_X \tilde f_s\,\frac{\di}{\di s}\,\hat\ell_\eps(\tilde f_s)
+
\hat\ell_\eps(\tilde f_s)\,\frac{\di}{\di s}\,\tilde f_s\di\m
\\
&=
\int_X p_\eps(\tilde f_s)\,\frac{\di}{\di s}\,\tilde f_s\di\m
\\
&=
\int_X g^\eps_s
\left(
\dot{\eta}(s)\,
\Delta_{\d,\m}\tilde f_s
+
\dot{\theta}(s)\,\P_{\eta(s)}\dot{f}_{\theta(s)}
\right)\di\m
\end{align*}
for all $s\in(0,1]$.
By the integration-by-part formula~\eqref{eq:int_by_part_formula} and the chain rule~\eqref{eq:Gamma_chain_rule}, we can write
\begin{equation*}
\int_X g^\eps_s\,\Delta_{\d,\m}\tilde f_s\di\m
=
-\int_X\Gamma(g^\eps_s,\tilde f_s)\di\m
=
-\int_X\Gamma(p_\eps(\tilde f_s),\tilde f_s)\di\m
=
-\int_X p_\eps'(\tilde f_s)\,\Gamma(\tilde f_s)\di\m
\end{equation*}
Observing that 
$-p_\eps'(r)\le-r\,(p_\eps'(r))^2$ for all $r\ge0$, since $r\,p_\eps'(r)\le1$ for all $r\ge0$ by~\eqref{eq:cantonese}, again by the chain rule~\eqref{eq:chain_rule_min_weak_grad} 
we can estimate
\begin{align*}
\int_X g^\eps_s\,\Delta_{\d,\m}\tilde f_s\di\m
&=
-\int_X p_\eps'(\tilde f_s)\,\Gamma(\tilde f_s)\di\m\\
&\le 
-\int_X \tilde f_s\,(p_\eps'(\tilde f_s))^2\,\Gamma(\tilde f_s)\di\m\\
&=
-\int_X \Gamma(p_\eps(\tilde f_s))\di\tilde\mu_s\\
&=
-\int_X \Gamma(g_s^\eps)\di\tilde\mu_s
\end{align*}
for all $s\in(0,1]$, so that
\begin{align*}
\frac{\di}{\di s}\,
E_\eps(\tilde{\mu}_s)
&=
\dot{\eta}(s)
\int_X g^\eps_s\,
\Delta_{\d,\m}\tilde f_s\di\m
+
\dot{\theta}(s)
\int_X g^\eps_s\,
\P_{\eta(s)}\dot{f}_{\theta(s)}
\di\m\\
&\le
-\dot{\eta}(s)
\int_X \Gamma(g_s^\eps)\di\tilde\mu_s
+
\dot{\theta}(s)
\int_X 
\dot{f}_{\theta(s)}\,
\P_{\eta(s)}g^\eps_s
\di\m
\end{align*}
for all $s\in(0,1]$, concluding the proof of~\eqref{eq:entropy_double_variation}.
\end{proof}

\subsection{Action and entropy along regular curves}

We now come to the following crucial result connecting the action estimate obtained in \cref{res:action_double_variation} with the entropic inequality proved in \cref{res:entropy_double_variation}.
For the same result in the standard $\BE(K,N)$ framework, we refer the reader to~\cite{AGS15}*{Theorem~4.16} and~\cite{EKS15}*{Proposition~4.16}.

\begin{theorem}[Action and entropy along strongly regular curves]
\label{res:action_ent_estimate_reg_curves}
Assume $(X,\d,\m)$ sa\-ti\-sfies $\wBE(\c,\infty)$ and is heat-smoothing 
as in~\cref{def:heat-smoothing_space}.
Under the same assumptions of \cref{res:action_double_variation} and \cref{res:entropy_double_variation},
if $\eps>0$ then
\begin{equation*}
\frac{1}{2}\,W_2^2(\tilde\mu_1,\tilde\mu_0)
-
\int_0^1\ddot{\eta}(s)
\,E_\eps(\tilde\mu_s)
\di s
+\dot{\eta}(1)\,
E_\eps(\tilde\mu_1)
\le
\dot{\eta}(0)\,
E_\eps(\tilde\mu_0)
+
\frac{1}{2\I_{-2,\eta}(1)}
\int_0^1|\dot{\mu}_s|^2\di s,
\end{equation*}
where 
\begin{equation*}
\I_{p,\eta}(s)=\int_0^s\c^p(\eta(r))\di r
\quad
\text{for all}\ s\in[0,1]\
\text{and}\ p\in\R,
\end{equation*}
and 
\begin{equation}\label{eq:def_theta_magic}
\theta(s)=\dfrac{\I_{-2,\eta}(s)}{\I_{-2,\eta}(1)}
\quad
\text{for all}\ s\in[0,1].	
\end{equation}
\end{theorem}

\begin{proof}
Let $\eps>0$ be fixed.
On the one hand, recalling~\eqref{eq:wD_less_slope_for_Lip}, by \cref{res:action_double_variation} we can estimate
\begin{equation*}
\frac{\di}{\di s}
\int_X\phi_s\di\tilde{\mu}_s
\le
-
\frac{1}{2}\int_X\Gamma(\phi_s)\di\tilde{\mu}_s
-
\dot{\eta}(s)\int_X\Gamma(\tilde{f}_s,\phi_s)\di\m
+
\dot{\theta}(s)\int_X\dot{f}_{\theta(s)}\,\P_{\eta(s)}\phi_s\di\m
\end{equation*}
for $\L^1$-a.e.\ $s\in(0,1)$.
On the other hand, by \cref{res:entropy_double_variation}, we can also estimate
\begin{equation*}
\dot{\eta}(s)\,\frac{\di}{\di s}\,
E_\eps(\tilde{\mu}_s)
\le
-
\frac{\dot{\eta}^2(s)}{2}
\int_X\Gamma(g_s^\eps)\di\tilde\mu_s
+
\dot{\theta}(s)\,\dot{\eta}(s)
\int_X\dot{f}_{\theta(s)}\,\P_{\eta(s)}g_s^\eps\di\m
\end{equation*}
for all $s\in(0,1]$.
Hence, we can estimate
\begin{align*}
\frac{\di}{\di s}\int_X\phi_s\di\tilde{\mu}_s
+
\dot{\eta}(s)\,\frac{\di}{\di s}\,
E_\eps(\tilde{\mu}_s)
&\le
\dot{\theta}(s)
\int_X\dot{f}_{\theta(s)}\,
\P_{\eta(s)}(\phi_s+\dot{\eta}(s)\,g^\eps_s)
\di\m
\\
&\quad
-
\frac{1}{2}
\int_X
\big(
\Gamma(\phi_s)
+
\dot{\eta}^2(s)\,
\Gamma(g_s^\eps)
\big)
\di\tilde{\mu}_s
-
\dot{\eta}(s)
\int_X\Gamma(\tilde{f}_s,\phi_s)\di\m
\end{align*}
for $\L^1$-a.e.\ $s\in(0,1)$.
Recalling that $r\,p_\eps'(r)\le1$ for all $r\ge0$ by~\eqref{eq:cantonese}, by the chain rule~\eqref{eq:Gamma_chain_rule} we can also estimate
\begin{align*}
\int_X\Gamma(g^\eps_s,\phi_s)\di\tilde\mu_s
=
\int_X\tilde f_s\,p_\eps'(\tilde f_s)\Gamma(\tilde f_s,\phi_s)\di\m
\le
\int_X\Gamma(\tilde{f}_s,\phi_s)\di\m	
\end{align*}
for all $s\in(0,1]$, so that
\begin{align*}
\frac{1}{2}
\int_X\big(\Gamma(\phi_s)
&+\dot{\eta}^2(s)\,\Gamma(g_s^\eps)\big)
\di\tilde{\mu}_s
+
\dot{\eta}(s)
\int_X\Gamma(\tilde{f}_s,\phi_s)\di\m
\\
&\ge
\frac{1}{2}
\int_X\big(\Gamma(\phi_s)+\dot{\eta}^2(s)\,\Gamma(g_s^\eps)\big)
\di\tilde{\mu}_s
+
\dot{\eta}(s)
\int_X\Gamma(g^\eps_s,\phi_s)\di\tilde\mu_s
\\
&=
\frac12
\int_X\Gamma(\phi_s+\dot{\eta}(s)\,g_s^\eps)
\di\tilde{\mu}_s
\end{align*}
for all $s\in(0,1]$. 
Thus, by~\eqref{eq:def_wBE}, we can estimate
\begin{equation}\label{eq:dolores}
\begin{split}
\frac{\di}{\di s}\int_X\phi_s\di\tilde{\mu}_s
+
\dot{\eta}(s)\,
\frac{\di}{\di s}\,
E_\eps(\tilde{\mu}_s)
&\le
\dot{\theta}(s)
\int_X\dot{f}_{\theta(s)}\,
\P_{\eta(s)}(\phi_s+\dot{\eta}(s)\,g^\eps_s)
\di\m\\
&\quad
-\frac12
\int_X\Gamma(\phi_s+\dot{\eta}(s)\,g_s^\eps)
\di\tilde{\mu}_s
\\
&=
\dot{\theta}(s)
\int_X\dot{f}_{\theta(s)}\,
\P_{\eta(s)}(\phi_s+\dot{\eta}(s)\,g^\eps_s)
\di\m\\
&\quad
-\frac12
\int_X\P_{\eta(s)}\big(\Gamma(\phi_s+\dot{\eta}(s)\,g_s^\eps)\big)
\di\mu_{\theta(s)}
\\
&\le
\dot{\theta}(s)
\int_X\dot{f}_{\theta(s)}\,
\P_{\eta(s)}(\phi_s+\dot{\eta}(s)\,g^\eps_s)
\di\m\\
&\quad
-
\frac{\c^{-2}(\eta(s))}2
\int_X\Gamma\big(\P_{\eta(s)}(\phi_s+\dot{\eta}(s)\,g_s^\eps)\big)
\di\mu_{\theta(s)}
\end{split}
\end{equation}
for $\L^1$-a.e.\ $s\in(0,1)$.
Since $\eta(s)>0$ for all $s\in(0,1]$, we have
\begin{equation*}
\phi_s+\dot{\eta}(s)\,g^\eps_s\in\Leb^\infty(X,\m)\cap\Sob^{1,2}(X,\d,\m)
\end{equation*}
for all $s\in(0,1]$ so that, thanks to heat-smoothing assumption~\eqref{assumption:heat-smoothing},
\begin{equation*}
\Gamma(\P_{\eta(s)}(\phi_s+\dot{\eta}(s)g^\eps_s))
=
|\D\P_{\eta(s)}(\phi_s+\dot{\eta}(s)g^\eps_s)|^2
\quad
\text{$\m$-a.e.\ in~$X$}
\end{equation*}
for all $s\in(0,1]$.
Thus, by~\eqref{eq:def_strong_reg_curve_density_C1_L2} in \cref{def:strong_reg_curve}, by \cref{res:Lisini_Lip} and by Young inequality, we can estimate
\begin{equation}\label{eq:master}
\begin{split}
\dot{\theta}(s)
\int_X\dot{f}_{\theta(s)}\,
\P_{\eta(s)}(\phi_s&+\dot{\eta}(s)\,g^\eps_s)\di\m
=
\frac{\di}{\di r}
\int_X
\P_{\eta(s)}(\phi_s+\dot{\eta}(s)\,g^\eps_s)
\di\mu_{\theta(r)}
\,\bigg|_{r=s}
\\
&\le
|\dot{\theta}(s)|\,|\dot{\mu}_{\theta(s)}|
\left(
\int_X|\D\P_{\eta(s)}(\phi_s+\dot{\eta}(s)g^\eps_s)|^2\di\mu_{\theta(s)}
\right)^{\frac12}
\\
&= 
|\dot{\theta}(s)|\,|\dot{\mu}_{\theta(s)}|
\left(
\int_X\Gamma(\P_{\eta(s)}(\phi_s+\dot{\eta}(s)g^\eps_s))\di\mu_{\theta(s)}
\right)^{\frac12}
\\
&\le
\frac{\c^2(\eta(s))}{2}\,\dot{\theta}^2(s)
\,
|\dot{\mu}_{\theta(s)}|^2
+
\frac{\c^{-2}(\eta(s))}{2}
\int_X\Gamma(\P_{\eta(s)}(\phi_s+\dot{\eta}(s)g^\eps_s))
\di\mu_{\theta(s)}
\end{split}
\end{equation}
for all $s\in(0,1]$.
In conclusion, by combining~\eqref{eq:dolores} with~\eqref{eq:master}, we get
\begin{equation}\label{eq:gatti}
\frac{\di}{\di s}\int_X\phi_s\di\tilde{\mu}_s
+
\dot{\eta}(s)\,\frac{\di}{\di s}\,
E_\eps(\tilde{\mu}_s)
\le
\frac{\c^2(\eta(s))}{2}\,
\dot{\theta}^2(s)\,|\dot{\mu}_{\theta(s)}|^2
\end{equation}
for $\L^1$-a.e.\ $s\in(0,1)$.
We now integrate~\eqref{eq:gatti} in $s\in[0,1]$.
For the left-hand side of~\eqref{eq:gatti}, we have
\begin{equation}\label{eq:cane}
\begin{split}
\int_0^1
&
\frac{\di}{\di s}
\int_X\phi_s\di\tilde{\mu}_s\di s
+
\int_0^1\dot{\eta}(s)\,
\frac{\di}{\di s}\,
E_\eps(\tilde{\mu}_s)\di s
\\
&=
\int_X\phi_1\di\tilde{\mu}_1
-
\int_X\phi_0\di\tilde{\mu}_0
-
\int_0^1\ddot{\eta}(s)\,E_\eps(\tilde{\mu}_s)\di s
+\big(\dot{\eta}(1)E_\eps(\tilde{\mu}_1)-\dot{\eta}(0)E_\eps(\tilde{\mu}_0)\big).
\end{split}
\end{equation}
For the right-hand side of~\eqref{eq:gatti}, instead, we simply choose
$\theta\colon[0,1]\to[0,1]$ 
as in~\eqref{eq:def_theta_magic}, 
so that 
\begin{equation}
\frac{1}{2}\label{eq:gatto}
\int_0^1\dot{\theta}^2(s)\,\c^2(\eta(s))\,|\dot{\mu}_{\theta(s)}|^2\di s
=
\frac{1}{2\I_{-2,\eta}(1)}
\int_0^1|\dot{\mu}_{\theta(s)}|^2\,\dot{\theta}(s)\di s
=
\frac{1}{2\I_{-2,\eta}(1)}
\int_0^1|\dot{\mu}_s|^2\di s.
\end{equation}
Combining~\eqref{eq:cane} with~\eqref{eq:gatto}, we get 
\begin{equation*}
\begin{split}
\int_X\phi_1\di\tilde{\mu}_1-\int_X\phi_0\di\tilde{\mu}_0
-
\int_0^1\ddot{\eta}(s)\,&E_\eps(\tilde{\mu}_s)\di s
+\dot{\eta}(1)\,E_\eps(\tilde{\mu}_1)\\
&\le
\dot{\eta}(0)\,E_\eps(\tilde{\mu}_0)
+\frac{1}{2\I_{-2,\eta}(1)}\int_0^1|\dot{\mu}_s|^2\di s
\end{split}
\end{equation*}
whenever $\eps>0$,
and the conclusion follows by taking the supremum on all $\phi\in\Lip_\star(X)$ thanks to~\eqref{eq:kantorovich_duality_p=2}.
\end{proof}

\begin{remark}[Errata to the proof of~\cite{AGS15}*{Theorem~4.16}]
We warn the reader that there is a typo in the last inequality of the long chain of inequalities in the proof of~\cite{AGS15}*{Theorem~4.16}: in place of $\frac{1}{2}(\dot{\theta}_s)e^{-2Kst}|\dot{\rho}_s|^2$, it should be written $\frac{1}{2}(\dot{\theta}_s)e^{-2Kst}|\dot{\rho}_{\theta(s)}|^2$. 
Unfortunately, this typo induces the authors to make the wrong choice of $\theta(s)$ at the beginning of~\cite{AGS15}*{p.~393}, making the proofs of~\cite{AGS15}*{Theorem~4.16 and Theorem~4.17} not completely corrected. 
The reader can easily fix all the computations needed in~\cite{AGS15} by adapting the ones performed above in the proof of \cref{res:action_ent_estimate_reg_curves}.
\end{remark}

\subsection{Admissible groups}
\label{subsec:admissible_groups}

We now focus our attention on some particular admissible metric-measure spaces that we call \emph{admissible groups}. 

\begin{definition}[Admissible group]
\label{def:admissible_group}
We say that an admissible metric-measure space $(X,\d,\m)$ is an \emph{admissible group} if:
\begin{enumerate}[label=(\roman*)]

\item
\label{item:def_adm_group_loc_compact}
the metric space $(X,\d)$ is locally compact;

\item 
the set $X$ is a \emph{topological group}, i.e.\ the group operations of \emph{multiplication} $(x,y)\mapsto xy$ 
and \emph{inversion} $x\mapsto x^{-1}$ are continuous;

\item
\label{item:def_adm_group_distance_left-inv}
$\d$ is \emph{left-invariant}, i.e.\ $\d(zx,zy)=\d(x,y)$ for all $x,y,z\in X$;

\item
\label{item:def_adm_group_measure_left-inv}
$\m$ is a \emph{left-invariant Haar measure}, i.e.\ $\m$ is a Radon measure such that $\m(xE)=\m(E)$ for all $x\in X$ and all Borel set $E\subset X$;

\item
\label{item:def_adm_group_unimodular}
$X$ is \emph{unimodular}, i.e.\ $\m$ is also right-invariant. 

\end{enumerate}	
\end{definition}

For an agile introduction on topological groups and Haar measures, we refer the reader to~\cite{F99}*{Section~11.1} and to~\cite{F16}*{Chapter~2}. 
For a more general approach to the subject, see~\cite{F69}*{Section~2.7}.

Note that, since $\d$ is left-invariant, we can write $B_r(x)=xB_r(\neu)$ for all $x\in X$ and $r>0$, where $\neu\in X$ is the identity element.
Since $\m$ is right invariant, we thus get that $\m(B_r(x))=\m(B_r(\neu))$ for all $x\in X$ and $r>0$, so that if~\eqref{eq:exp_growth} is satisfied for one $x_0\in X$, then it is satisfied (with the same constants $A,B>0$) for all $x_0\in X$. 
Hence, from now on, we assume $x_0=\neu$ in~\eqref{eq:exp_growth} for simplicity.

Here and in the rest of the paper, we let $L_x(y)=x^{-1}y$, $x,y\in X$, be the left-translation map. 
The following result is a simple consequence of \cref{def:admissible_group} and the definitions given in \cref{sec:preliminaries}.
We leave its proof to the interested reader.

\begin{proposition}[Invariance properties of metric-measure objects]
\label{res:left-invariance}
Let $(X,\d,\m)$ be an admissible group.
Let $p\in[1,+\infty]$ and~$x\in X$ be fixed.
The following hold.

\begin{enumerate}[label=(\roman*)]

\item\label{item:left-inv_W_p}
If $p<+\infty$, then $W_p((L_x)_\sharp\mu,(L_x)_\sharp\nu)=W_p(\mu,\nu)$ for all $\mu,\nu\in\Prob_p(X)$.

\item
If $f\in\Lip(X)$, then also $f\circ L_x\in\Lip(X)$,
with
$|\D(f\circ L_x)|=|\D f|\circ L_x$
and
$|\D^*(f\circ L_x)|=|\D^* f|\circ L_x$.

\item
If $f\in\Sob^{1,2}(X,\d,\m)$, then also
$f\circ L_x\in\Sob^{1,2}(X,\d,\m)$, 
with
$\Ch(f\circ L_x)=\Ch(f)$
and
$|\D(f\circ L_x)|_w=|\D f|_w\circ L_x$.

\item
If $f\in\Dom(\Delta_{\d,\m})$, then also $f\circ L_x\in\Dom(\Delta_{\d,\m})$ with $\Delta_{\d,\m}(f\circ L_x)=(\Delta_{\d,\m}f)\circ L_x$.

\item\label{item:left-inv_heat}
If $f\in\Leb^p(X,\m)$, then $\P_t(f\circ L_x)=(\P_tf)\circ L_x$ for all $t>0$.

\end{enumerate}	
\end{proposition}

\subsection{Convolution}

We let
\begin{equation*}
(f\star g)(x)
=
\int_X f(xy^{-1})\,g(y)\di\m(y),
\quad
x\in X,
\end{equation*}
be the \emph{group convolution} of $f,g\in\Leb^1(X,\m)$, and we use the notation $f*g$ to denote the usual convolution in~$\R^n$ of $f,g\in\Leb^1(\R^n,\L^n)$.
Since $X$ is unimodular, by a simple change of variables we can also write
\begin{equation*}
(f\star g)(x)
=
\int_X f(y)\,g(y^{-1}x)\di\m(y),
\quad
x\in X.
\end{equation*}
Thus, accordingly, for $f\in\Leb^1(X,\m)$ and $\mu\in\Prob(X)$ we write
\begin{equation*}
(f\star\mu)(x)
=
\int_Xf(xy^{-1})\di\mu(y),
\quad
x\in X.
\end{equation*}
For an account on the elementary properties of convolution in locally compact groups, we refer the reader to~\cite{F16}*{Section~2.5} (in particular, recall Young inequality in~\cite{F16}*{Proposition~2.40}).

The following result completes the information provided by \cref{res:density_p_t_props}.

\begin{lemma}[$(\tilde\P_t)_{t>0}$ as a right-convolution]
\label{res:p_t_neu}
Assume $(X,\d,\m)$ is an admissible group satisfying $\BE_w(\c,\infty)$.
If $t>0$, then 
\begin{equation*}
\p_t[x](y)=\p_t[\neu](y^{-1}x)
\quad
\text{for $\m$-a.e.}\ x,y\in X.
\end{equation*}
Consequently, if $t>0$, then we can write
\begin{equation*}
\tilde\P_t f(x)
=
\int_X f(y)\,\p_t[\neu](y^{-1}x)\di\m(y)
=(f\star\p_t[\neu])(x)
\end{equation*}
for $\m$-a.e.\ $x\in X$ and for all one-side bounded measurable functions $f\colon X\to\overline\R$.
\end{lemma}

\begin{proof}
Let $t>0$ be fixed.
We start by claiming that
\begin{equation}\label{eq:left-inv_dual_heat}
\H_t((L_x)_\sharp\mu)=(L_x)_\sharp(\H_t\mu)
\end{equation}
for all $\mu\in\Prob_2(X)$ and~$x\in X$.
Indeed, if $\mu\ll\m$, then by~\eqref{eq:dual_heat_def} claim~\eqref{eq:left-inv_dual_heat} is nothing but \cref{res:left-invariance}\ref{item:left-inv_heat}.
Since $\Prob_2^{\rm ac}(X)$ is $W_2$-dense in $\Prob_2(X)$, claim~\eqref{eq:left-inv_dual_heat} follows from \cref{res:left-invariance}\ref{item:left-inv_W_p} and \cref{res:kuwada_equivalence}\ref{item:kuwada_equiv_2} by a simple approximation argument.
Thanks to claim~\eqref{eq:left-inv_dual_heat}, we can compute
\begin{align*}
\int_X f(y)\,\p_t[x](y)\di\m(y)
&=
\int_X f(y)\di\H_t\delta_x(y)\\
&=
\int_X f(y)\di\H_t((L_{x^{-1}})_\sharp\delta_\neu)(y)\\
&=
\int_X f(y)\di\,(L_{x^{-1}})_\sharp(\H_t\delta_\neu)(y)\\
&=
\int_X f(xy)\,\p_t[\neu](y)\di\m(y)\\
&=
\int_X f(y)\,\p_t[\neu](x^{-1}y)\di\m(y)\\
\end{align*}
for all $f\in\Leb^\infty(X,\m)$ and~$x\in X$. 
Thus $\p_t[x](y)=\p_t[\neu](x^{-1}y)$ for all~$x\in X$ and $\m$-a.e.\ $y\in X$, and the conclusion follows by \cref{res:density_p_t_props}\ref{item:density_p_t_symmetry}.
\end{proof}

According to \cref{res:p_t_neu}, we thus simply write $\p_t[\neu]=\p_t$ for all~$t>0$ and we call $(\p_t)_{t>0}$ the (\emph{metric-measure}) \emph{heat kernel} of the (pointwise version of the) heat flow.

\begin{remark}[Application of~\eqref{eq:wang_p_2t}]
Assume $(X,\d,\m)$ is an admissible group satisfying $\BE_w(\c,\infty)$ with $\m\in\Prob(X)$. 
From inequality~\eqref{eq:wang_p_2t} in \cref{res:wang_p_t} we immediately have
\begin{equation}\label{eq:barabba}
\p_{2t}(x)
\ge
\exp\left(-\frac{\d^2(x,\neu)}{4\I_{-2}(t)}\right)
\end{equation}
for all $t>0$ and $\m$-a.e.\ $x\in X$.
Inequality~\eqref{eq:barabba} applies in particular to the (sub-Riemmanian) $\SU(2)$ group, see \cref{subsec:SU(2)_group} for the precise definition. 
Up to our knowledge, inequality~\eqref{eq:barabba} provides a new lower bound on the heat kernel in~$\SU(2)$.
\end{remark}

\subsection{Approximation by regular curves in admissible groups}

Let $(X,\d,\m)$ be an admissible group. 
We say that  $\rho\in\Leb^1(X,\m)$ is a \emph{convolution kernel} if it is non-negative, renormalized, symmetric and has bounded support, i.e.\
\begin{equation}\label{eq:def_mollifier_group}
\rho\ge0,
\quad
\int_X\rho\di\m=1,
\quad
\rho(x^{-1})=\rho(x)\ \text{for all}\ x\in X,
\quad
\supp\rho\
\text{is bounded}.
\end{equation}
Since $\d$ is left-invariant, $\d$-balls centered at $\neu\in X$ are symmetric, in the sense that 
\begin{equation*}
x\in B_r(\neu)	
\iff
x^{-1}\in B_r(\neu)
\end{equation*}
whenever $x\in X$ and $r>0$.
Thus, for all $r>0$, the function $\rho_r\in\Leb^1(X,\m)\cap\Leb^\infty(X,\m)$ defined by
\begin{equation}
\label{eq:def_convolution_kernel_group}
\rho_r(x)
=
\frac{\chi_{B_r(\neu)}(x)}{\m(B_r(\neu))},
\quad
x\in X,
\end{equation}
is a convolution kernel (and also an \emph{approximate identity} as $r\to0^+$, see~\cite{F16}*{Proposition~2.44}).

The following result provides a simple but useful relation between test plans and convolution.

\begin{lemma}[Convolution and plans]
\label{res:convolution_plan}
Let $(X,\d,\m)$ be an admissible group and let $\rho\in\Leb^1(X,\m)$ be as in~\eqref{eq:def_mollifier_group}.
Let 
$\mu_1,\mu_2\in\Prob(X)$
and define 
$\tilde\mu_1,\tilde\mu_2\in\Prob(X)$ 
as
$\tilde\mu_1=(\rho\star\mu_1)\m$
and 
$\tilde\mu_2=(\rho\star\mu_2)\m$.
If 
$\pi\in\mathsf{Plan}(\mu_1,\mu_2)$,
then 
the measure $\tilde\pi$ given by
\begin{equation}\label{eq:def_plan_mollifier}
\int_{X\times X}\phi(x_1,x_2)\di\tilde\pi(x_1,x_2)
=
\int_{X\times X}\int_X\rho(y)\,\phi(yx_1,yx_2)\di\m(y)\di\pi(x_1,x_2)
\end{equation}
for all $\phi\in\Cont_b(X\times X)$
is such that 
$\tilde\pi\in\mathsf{Plan}(\tilde\mu_1,\tilde\mu_2)$.
\end{lemma}

\begin{proof}
Note that $\tilde\pi\in\Prob(X\times X)$, since
\begin{align*}
\int_{X\times X}\di\tilde\pi(x_1,x_2)
=
\int_{X\times X}\int_X\rho(y)\di\m(y)\di\pi(x_1,x_2)
=
\int_{X\times X}\di\pi(x_1,x_2)
\end{align*}
by~\eqref{eq:def_mollifier_group} and~\eqref{eq:def_plan_mollifier}.
Let us now prove that 
$(p_i)_\sharp\tilde\pi=\mu_i$, $i=1,2$.
Let $\phi\in\Cont_b(X)$ and set $\psi_i=\phi\circ p_i\in\Cont_b(X\times X)$,
$i=1,2$.
By~\eqref{eq:def_plan_mollifier} and recalling that~$X$ is unimodular, we can write 
\begin{align*}
\int_X\phi\di\,(p_i)_\sharp\tilde\pi
&=
\int_{X\times X}\psi_i\di\tilde\pi
\\
&=
\int_{X\times X}\int_X\rho(y)\,\psi_i(yx_1,yx_2)\di\m(y)\di\pi(x_1,x_2)
\\
&=
\int_X
\int_X\rho(y)\,\phi(yx)\di\m(y)
\di\mu_i(x)\\
&=
\int_X\int_X\rho(zx^{-1})\,\phi(z)\di\m(z)\di\mu_i(x)\\
&=
\int_X\phi(z)
\int_X\rho(zx^{-1})
\di\mu_i(x)
\di\m(z)
\\
&=
\int_X(\rho\star\mu_i)(z)\,\phi(z)\di\m(z)
\\
&=
\int_X\phi\di\tilde\mu_i
\end{align*}
thanks to Fubini Theorem, concluding the proof.
\end{proof}

A fundamental consequence of \cref{res:convolution_plan} is the following estimate on the Wasserstein velocity of left-convoluted curves of measures. 

\begin{lemma}[Convolution and $W_q$-velocity]
\label{res:convolution_velocity}
Let $(X,\d,\m)$ be an admissible group and let $\rho\in\Leb^1(X,\m)$ be as in~\eqref{eq:def_mollifier_group}.
Let $p,q\in[1,+\infty)$ and let $I\subset\R$ be an interval.
If 
$s\mapsto\mu_s\in\AC^q(I;\Prob_p(X))$, 
then also 
$s\mapsto\tilde\mu_s\in\AC^q(I;\Prob_p(X))$,
where 
$\tilde\mu_s=(\rho\star\mu_s)\m$ for all $s\in I$,
with
$
|\dot{\tilde\mu}_s|
\le
|\dot\mu_s|
$
for $\L^1$-a.e.\ $s\in I$.
\end{lemma}

\begin{proof}
Since $X$ is unimodular and $\d$ is left-invariant, by Tonelli Theorem we can estimate
\begin{align*}
\int_X\d^p(x,\neu)\di\tilde\mu_s(x)
&=
\int_X\int_X
\rho(xy^{-1})\,\d^p(x,\neu)\di\mu_s(y)
\di\m(x)
\\
&=
\int_X\int_X
\rho(xy^{-1})\,\d^p(x,\neu)\di\m(x)
\di\mu_s(y)
\\
&=
\int_X
\int_X\rho(z)\,\d^p(zy,\neu)\di\m(z)\di\mu_s(y)
\\
&\le
2^{p-1}
\int_X
\int_X\rho(z)\,\big(\d^p(zy,z)+\d^p(z,\neu)\big)\di\m(z)\di\mu_s(y)
\\
&=
2^{p-1}\left(
\int_X\d^p(y,\neu)\di\mu_s(y)
+
\int_X\d^p(z,\neu)\,\rho(z)\di\m(z)
\right),
\end{align*}
proving that $\tilde\mu_s\in\Prob_p(X)$ for all $s\in I$.
Now for all $k\in\N$ let $\phi_k\in\Cont_b(X\times X)$ be defined as $\phi_k(x,y)=\d^p(x,y)\wedge k$ for all $x,y\in X$.
Let $s_0,s_1\in I$ and let $\pi_{s_0,s_1}\in\mathrm{OptPlan}(\mu_{s_0},\mu_{s_1})$ be an optimal plan between $\mu_{s_0}$ and $\mu_{s_1}$.
Let $\tilde\pi_{s_0,s_1}\in\mathsf{Plan}(\tilde\mu_{s_0},\tilde\mu_{s_1})$ be given by \cref{res:convolution_plan} accordingly.
Since $\phi_k(zx,zy)=\phi_k(x,y)$ for all $x,y,z\in X$ and $k\in\N$, we can estimate
\begin{align*}
\int_{X\times X}\phi_k(x,y)\di\tilde\pi_{s_0,s_1}(x,y)
&=
\int_{X\times X}\int_X\rho(z)\,\phi_k(zx,zy)\di\m(z)\di\pi_{s_0,s_1}(x,y)
\\
&=
\int_{X\times X}\int_X\rho(z)\,\phi_k(x,y)\di\m(z)\di\pi_{s_0,s_1}(x,y)
\\
&\le
\int_{X\times X}\d^p(x,y)\di\pi_{s_0,s_1}(x,y)
\\
&=
W_p^p(\mu_{s_0},\mu_{s_1}).
\end{align*}
By the Monotone Convergence Theorem, we can pass to the limit as $k\to+\infty$ and get
\begin{equation*}
W_p(\tilde\mu_{s_0},\tilde\mu_{s_1})
\le
W_p(\mu_{s_0},\mu_{s_1})
\end{equation*}
whenever $s_0,s_1\in I$, concluding the proof.
\end{proof}

\begin{remark}[Right-convoluted measures and velocity]
\label{rem:right_convolution}
It is not difficult to see that a statement similar to that of \cref{res:convolution_plan} holds for the right-convoluted measures 
$\hat\mu_1=(\mu_1\star\rho)\m$
and
$\hat\mu_2=(\mu_2\star\rho)\m$.
However, since $\d$ is not necessarily right-invariant, it is not clear how to prove a statement similar to that of \cref{res:convolution_velocity} for the right-convoluted curve $s\mapsto\hat\mu_s=(\mu_s\star\rho)\m$.
Since in admissible groups the heat semigroup acts on measures as the right-convolution with the heat kernel as seen in \cref{res:p_t_neu}, the lack of an estimate on the Wasserstein velocity of right-convoluted curves is a central obstacle for the use of the heat-regularization techniques (which, in the standard $\BE(K,N)$ framework, inevitably rely on the crucial fact that $\underline\c(0^+)=1$, recall the discussion made in \cref{subsec:strong_reg_curves}).
This also explains why, in \cref{res:approx_strong_reg_curves} below, we need to assume that the ambient space is an admissible group and rely on left-convolution of measures.
\end{remark}

We now prove the following crucial approximation result. 
The line of the proof is close to that of~\cite{AS19}*{Theorem~4.8}.
For the approximation of curves under the standard $\BE(K,N)$ condition, we refer the reader to~\cite{AGS15}*{Proposition~4.11} and~\cite{EKS15}*{Lemma~4.11}.

\begin{theorem}
[Approximation by strongly regular curves in $\Prob_2(X)$]
\label{res:approx_strong_reg_curves}
Let $(X,\d,\m)$ be an admissible group. 
If 
$s\mapsto\mu_s\in\AC^2([0,1];\Prob_2(X))$, 
then there exist strongly regular curves 
$s\mapsto\mu_s^n\in\AC^2([0,1];\Prob_2(X))$,
$n\in\N$, in the sense of \cref{def:strong_reg_curve},
such that:
\begin{enumerate}[label=(\roman*)]

\item\label{item:approx_strong_reg_curves_W_2} 
$\mu^n_s\overset{W_2}{\longto}\mu_s$ as $n\to+\infty$ for all $s\in[0,1]$;

\item\label{item:approx_reg_curves_limsup} 
$\limsup\limits_{n\to+\infty}
\displaystyle\int_0^1|\dot\mu^n_s|^2\di s
\le
\int_0^1|\dot\mu_s|^2\di s$;

\item\label{item:approx_reg_curves_Ent} 
$\lim\limits_{n\to+\infty}\Ent_\m(\H_t\mu^n_s)=\Ent_\m(\H_t\mu_s)$ for all $s\in[0,1]$ and $t\ge0$.
\end{enumerate}
\end{theorem}

\begin{proof}
We divide the proof in four steps.

\smallskip

\textit{Step~1: time-extension to~$\R$}.
We define $\R\ni s\mapsto\nu_s\in\Prob_2(X)$ by extending $[0,1]\ni s\mapsto\mu_s\in\Prob_2(X)$ by continuity with constant values in $(-\infty,0)\cup(1,+\infty)$.
Clearly, we have that $s\mapsto\nu_s\in\AC^2(\R;\Prob_2(X))$.

\smallskip

\textit{Step~2: smoothing in the space variable}.
For all $r>0$, let $\rho_r\in\Leb^1(X,\m)\cap\Leb^\infty(X,\m)$ be defined as in~\eqref{eq:def_convolution_kernel_group}.
We thus define $\nu_s^r=f^r_s\m$, where
\begin{equation*}
f^r_s(x)
=
(\rho_r\star\nu_s)(x)
=
\int_X\rho_r(xy^{-1})\di\nu_s(y),
\quad
x\in X,
\end{equation*}
for all $s\in\R$ and $r>0$.
By \cref{res:convolution_velocity}, we have $s\mapsto\nu_s^r\in\AC^2(\R;\Prob_2(X))$, with $|\dot\nu^r_s|\le|\dot\nu_s|$ for $\L^1$-a.e.\ $s\in\R$, for all~$r>0$.
Since the family $(\rho_r)_{r>0}$ is a symmetric approximation of the identity, we have 
\begin{equation}\label{eq:du}
\lim_{r\to0^+}
\int_X\phi\di\nu_s^r
=
\lim_{r\to0^+}
\int_X(\rho_r\star\phi)\di\nu_s
=
\int_X\phi\di\nu_s
\end{equation}
for all $\phi\in\Cont_b(X)$ by the Dominated Convergence Theorem for all~$s\in\R$, so that $\nu_s^r\weakto\nu_s$ as~$r\to0^+$ for all~$s\in\R$.
In addition, we can write
\begin{align*}
\int_X\d^2(x,\neu)\di\nu^r_s(x)
&=
\int_X(\rho_r\star\d^2(\cdot,\neu))(x)\di\nu_s(x)
\\
&=
\int_X
\int_X\rho_r(y)\,\d^2(x,y)\di\m(y)
\di\nu_s(x)
\\
&=
\int_X
\aint_{B_r(\neu)}\d^2(x,y)\di\m(y)
\di\nu_s(x)
\end{align*} 
for all $s\in\R$ and $r>0$, so that
\begin{align*}
\left|
\int_X\d^2(x,\neu)\di\nu^r_s(x)
-
\int_X\d^2(x,\neu)\di\nu_s(x)
\right|
&\le
\int_X
\aint_{B_r(\neu)}\left|\d^2(x,y)-\d^2(x,\neu)\right|\di\m(y)
\di\nu_s(x)
\\
&\le
\int_X
r(r+2\d(x,\neu))
\di\nu_s(x)
\end{align*}
for all $s\in\R$ and $r>0$. 
Hence
\begin{equation}\label{eq:da}
\lim_{r\to0^+}
\int_X\d^2(x,\neu)\di\nu^r_s(x)
=
\int_X\d^2(x,\neu)\di\nu_s(x)
\end{equation}
for all $s\in\R$.
Consequently, from~\eqref{eq:du} and~\eqref{eq:da} we infer that $\nu_s^r\overset{W_2}{\longto}\nu_s$ for all $s\in\R$ and, in particular,
\begin{equation*}
\liminf_{r\to0^+}
\Ent_\m(\nu_s^r)
\ge
\Ent_\m(\nu_s)
\end{equation*}
for all $s\in\R$.
We can also write
\begin{equation*}
\nu^r_s
=
\int_X (L_y)_\sharp\nu_s\,\rho_r(y)\di\m(y)
\end{equation*}
for all $s\in\R$, where $L_y(x)=y^{-1}x$, $x,y\in X$, denotes the left-translation map.
If $\mathfrak n=e^{-c\,\d^2(\cdot,\neu)}\m$
is as in~\eqref{eq:ent_useful_formula}, then by Jensen inequality we can estimate
\begin{equation*}
\Ent_\mathfrak{n}(\nu^r_s)
\le
\int_X
\Ent_\mathfrak{n}((L_y)_\sharp\nu_s)
\,\rho_r(y)\di\m(y)
\end{equation*} 
for all $s\in\R$, so that by~\eqref{eq:ent_sharp_formula} and~\eqref{eq:ent_useful_formula} we can write
\begin{align*}
\Ent_\mathfrak{n}((L_y)_\sharp\nu_s)
=
\Ent_{\mathfrak{n}_{y^{-1}}}(\nu_s)
=
\Ent_{\m}(\nu_s)
+
c\int_X\d^2(yx,\neu)\di\nu_s(x)
\end{align*}
for all $y\in X$ and $s\in\R$, where 
$\mathfrak n_y
=
(L_y)_\sharp\mathfrak n
=
e^{-c\,\d^2(\cdot,y)}\m$.
Since
\begin{align*}
\int_X\int_X\d^2(yx,\neu)\di\nu_s(x)\,\rho_r(y)\di\m(y)
&=
\int_X\int_X\d^2(x,y^{-1})\,\rho_r(y)\di\m(y)\di\nu_s(x)
\\
&=
\int_X\int_X\d^2(x,y)\,\rho_r(y)\di\m(y)\di\nu_s(x)
\\
&=
\int_X\d^2(x,\neu)\di\nu_s^r(x)
\end{align*}
by the symmetry of $\rho_r$, again by~\eqref{eq:ent_useful_formula} we can estimate 
\begin{align*}
\Ent_\m(\nu_s^r)
&=
\Ent_\nu(\nu_s^r)
-
c\int_X\d^2(x,\neu)\di\nu_s^r(x)
\\
&\le
\int_X
\left(
\Ent_{\m}(\nu_s)
+
c\int_X\d^2(yx,\neu)\di\nu_s(x)
\right)
\rho_r(y)\di\m(y)
-
c\int_X\d^2(x,\neu)\di\nu_s^r(x)
\\
&=
\Ent_\m(\nu_s)
+c
\left(
\int_X\int_X\d^2(x,y)\,\rho_r(y)\di\m(y)\di\nu_s(x)
-
\int_X\d^2(x,\neu)\di\nu_s^r(x)
\right)
\\
&=
\Ent_\m(\nu_s)
\end{align*}
for all $s\in\R$ and $r>0$, so that 
\begin{equation*}
\lim_{r\to0^+}
\Ent_\m(\nu_s^r)
=
\Ent_\m(\nu_s).
\end{equation*} 

\smallskip

\textit{Step~3: smoothing in the time variable}.
Now let $r>0$ be fixed.
Let $\zeta\colon\R\to\R$ be a symmetric smooth mollifier in~$\R$, i.e.\
\begin{equation*}
\zeta\in\Cont^\infty_c(\R),
\quad
\supp\zeta\subset[-1,1],
\quad 
0\le\zeta\le1,
\quad
\int_{\R}\zeta\di\tau=1.
\end{equation*}
We define $\zeta_j(\tau)=j\,\zeta(j\tau)$ for all $\tau\in\R$ and $j\in\N$ and consider
$\nu^{j,r}_s=f^{j,r}_s\m$ with
\begin{equation*}
f_s^{j,r}
=
(\zeta_j*f_\cdot^r)(s)
=
\int_\R f_\tau^r\,\zeta_j(\tau-s)\di\tau
\end{equation*}
for all $s\in\R$ and all $j\in\N$.
If $s,s'\in\R$ and 
$\pi_{s,s'}^r\in\mathsf{Plan}(\nu_s^r,\nu_{s'}^r)$, 
then the measure 
$\pi^{j,r}_s\in\Prob(X\times X)$ given by
\begin{equation*}
\int_{X\times X}\phi(x,y)\di\pi^{j,r}_s(x,y)
=
\int_\R
\zeta_j(s-\tau)
\int_{X\times X}
\phi(x,y)\di\pi_{s,\tau}^r(x,y)
\di\tau,
\end{equation*}
for any $\phi\in\Cont_b(X\times X)$, satisfies
$\pi_s^{j,r}\in\mathsf{Plan}(\nu_s^r,\nu_s^{j,r})$. 
Thus, by the convexity properties of the squared $2$-Wasserstein distance and Jensen inequality, we get
\begin{equation*}
W_2^2(\nu^{j,r}_s,\nu^r_s)
\le
\int_\R\zeta_j(s-\tau)
\,
W_2^2(\nu^r_s,\nu^r_\tau)\di\tau
\end{equation*}
for all $s\in\R$ and $j\in\N$, so that
$\nu^{j,r}_s\overset{W_2}{\longto}\nu^r_s$ as $j\to+\infty$ and, in particular,
\begin{equation*}
\liminf_{j\to+\infty}
\Ent_\m(\nu^{j,r}_s)
\ge
\Ent_\m(\nu^{r}_s)
\end{equation*}
for all $s\in\R$.
In a similar fashion, we can estimate
\begin{align*}
W_2^2(\nu_s^{j,r},\nu_{s'}^{j,r})
=
W_2^2
\left(
\int_\R\nu_{s-\tau}^r\,\zeta_j(\tau)\di\tau
,
\int_\R\nu_{s'-\tau}^r\,\zeta_j(\tau)\di\tau
\right)
\le
\int_\R
W_2^2
\left(
\nu_{s-\tau}^r
,
\nu_{s'-\tau}^r
\right)
\zeta_j(\tau)\di\tau
\end{align*}
for all $s,s'\in\R$ and $j\in\N$, so that $s\mapsto\nu_s^{j,r}\in\AC^2(\R;\Prob_2(X))$ with $|\dot\nu_s^{j,r}|\le(\zeta_j*|\dot\nu_\cdot^r|)(s)$ for $\L^1$-a.e.\ $s\in\R$ and all $j\in\N$.
As in Step~2, let $\mathfrak n=e^{-c\,\d^2(\cdot,\neu)}\m$ be as in~\eqref{eq:ent_useful_formula}.
Since the function $H(u)=u\log u+(1-u)$, defined for all $u\ge0$, is convex and non-negative, by Jensen inequality we can estimate
\begin{align*}
\Ent_{\mathfrak n}(\nu_s^{j,r})
&=
\int_X H\left(f_s^{j,r}\,\tfrac{\di\m}{\di\mathfrak n}\right)\di\mathfrak n
\\
&=
\int_X H\left(
(\zeta_j*f_\cdot^r)(s)\,
\tfrac{\di\m}{\di\mathfrak n}
\right)\di\mathfrak n
\\
&=
\int_X H\left((\zeta_j*(f_\cdot^r\,
\tfrac{\di\m}{\di\mathfrak n}
))(s)\right)\di\mathfrak n
\\
&\le
\int_X \left(\zeta_j*H\left(f_\cdot^r\,
\tfrac{\di\m}{\di\mathfrak n}
\right)\right)(s)\di\mathfrak n
\\
&=
(\zeta_j*\Ent_{\mathfrak n}(\nu^r_\cdot))(s)
\end{align*}
for all $s\in\R$ and $j\in\N$.
Arguing as in Step~2, we immediately deduce that 
$
\Ent_\m(\nu_s^{j,r})
\le
(\zeta_j*\Ent_\m(\nu_\cdot^r))(s)
$
for all $s\in\R$ and $j\in\N$, so that
\begin{equation*}
\lim_{j\to+\infty}
\Ent_\m(\nu_s^{j,r})
=
\Ent_\m(\nu_s^r)
\end{equation*}
for all $s\in\R$.

\smallskip

\textit{Step~4: time-restriction to $[0,1]$ and conclusion}.
Define the curve $s\mapsto\mu_s^{j,k}$ as the restriction of the curve $s\mapsto\nu_s^{j,k}$ to the interval $[0,1]$.
Note that the regularization map $\mathcal{R}_{j,r}$ sending the original curve $s\mapsto\mu_s$ to the regularized curve $s\mapsto\mu_s^{j,k}$ is linear (with respect to convex combinations) and thus commutes with the dual heat flow map, i.e.\ $\mathcal{R}_{j,r}\circ\H_t=\H_t\circ\mathcal{R}_{j,r}$ for all $t\ge0$.
Since by construction
\begin{align*}
\lim_{j,r}
\Ent_\m(\mu^{j,r}_s)
=
\Ent_\m(\mu_s)
\end{align*}
for all $s\in[0,1]$, we thus have
\begin{align*}
\limsup_{j,r}
\Ent_\m(\H_t\mu^{j,r}_s)
=
\limsup_{j,r}
\Ent_\m(\mathcal R^{j,r}(\H_t\mu_s))
=
\Ent_\m(\H_t\mu_s)
\le
\Ent_\m(\mu_s)
\end{align*}
for all $t\ge0$ and $s\in[0,1]$ by \cref{res:Fisher_Entropy_along_H_t}\ref{item:entropy_decreasing_dual_heat}.
The conclusion thus follows by a standard diagonalization argument.
\end{proof}

\subsection{Entropic inequalities in admissible groups}

We can now state and prove the main result of this paper.
We refer the reader to~\cite{AGS15}*{Theorem~4.17} and to~\cite{EKS15}*{Theorem~4.19} for the analogous results in the standard $\BE(K,N)$ framework. 

\begin{theorem}[Entropic inequalities]\label{res:th_equivalence}
Let $(X,\d,\m)$ be an admissible group satisfying the heat-smoothing property as in~\cref{def:heat-smoothing_space}.
The following are equivalent.

\begin{enumerate}[label=(\roman*)]

\item\label{item:th_equiv_wBE} 
$(X,\d,\m)$ satisfies $\wBE(\c,\infty)$.

\item\label{item:th_equiv_wEVI} 
The dual heat semigroup $(\H_t)_{t\ge0}$ in~\eqref{eq:dual_heat_def} satisfies
\begin{equation}\label{eq:th_equiv_wEVI}
\frac{1}{2}\,W_2^2(\H_{t_1}\mu_1,\H_{t_0}\mu_0)
-\frac{1}{2\RI(t_0,t_1)}\,W_2^2(\mu_1,\mu_0)
\le
(t_1-t_0)\Big(\Ent_\m(\H_{t_0}\mu_0)-\Ent_\m(\H_{t_1}\mu_1)\Big)
\end{equation}
for all $\mu_0\in\Dom(\Ent)$, $\mu_1\in\Prob_2(X)$ and $0\le t_0\le t_1$, with also $\mu_1\in\Dom(\Ent_\m)$ in the particular case $t_1=t_0=0$,
where 
\begin{equation}\label{eq:RI_def}
\RI(t_0,t_1)
=
\int_0^1\c^{-2}((1-s)t_0+st_1)\di s.
\end{equation}

\item\label{item:th_equiv_wConvex} 
The dual heat semigroup $(\H_t)_{t\ge0}$ in~\eqref{eq:dual_heat_def} uniquely extends to a map on $\Prob_2(X)$ satisfying~\eqref{eq:dual_heat_weak_kuwada}
and such that 
\begin{equation}\label{eq:th_equiv_wConvex}
\begin{split}
\Ent_\m(\H_{t+h}\mu_s)
&\le 
(1-s)\,\Ent_m(\H_t\mu_0)+s\,\Ent_\m(\H_t\mu_1)\\
&\qquad+\frac{s(1-s)}{2h}\left(\frac{1}{\RI(t,t+h)}\,W_2^2(\mu_0,\mu_1)-W_2^2(\H_t\mu_0,\H_t\mu_1)\right)
\end{split}
\end{equation} 
for all $t\ge0$ and $h>0$,
whenever
$s\mapsto\mu_s\in\mathrm{Geo}([0,1];\Prob_2(X))$ 
is a $W_2$-geodesic joining 
$\mu_0,\mu_1\in\Dom(\Ent_\m)$
and $\RI$ is as in~\eqref{eq:RI_def}. 
\end{enumerate}
\end{theorem}

Mimicking the standard framework, thanks to \cref{res:th_equivalence} we can introducing the following notation.

\begin{definition}[$\EVI_w(\c)$ and $\RCD_w(\c,\infty)$ conditions]
An admissible metric-measure space $(X,\d,\m)$ is said to satisfy the \emph{weak Evolution Variation Inequality} with respect to the function $\c\colon[0,+\infty)\to(0,+\infty)$ in~\eqref{eq:def_c}, $\EVI_w(\c)$ for short, if inequality~\eqref{eq:th_equiv_wEVI} in \cref{res:th_equivalence}\ref{item:th_equiv_wEVI} holds.
Analogously, $(X,\d,\m)$ is said to satisfy the (\emph{dimension-free}) \emph{weak Riemannian Curvature-Dimension Condition} with respect to the function~$\c$, $\RCD_w(\c,\infty)$ for short, if inequality~\eqref{eq:th_equiv_wConvex} in \cref{res:th_equivalence}\ref{item:th_equiv_wConvex} holds.
\end{definition}

With this terminology, one can rephrase \cref{res:th_equivalence} simply writing that, for an admissible heat-smoothing group $(X,\d,\m)$, it holds
\begin{equation*}
\BE_w(\c,\infty)
\iff
\EVI_w(\c)
\iff
\RCD_w(\c,\infty).
\end{equation*}

\begin{proof}[Proof of \cref{res:th_equivalence}]
We prove each implication separately.

\medskip

\textit{Proof of \ref{item:th_equiv_wBE}$\,\Rightarrow\,$\ref{item:th_equiv_wEVI}}. 
Let $0\le t_0\le t_1$ be fixed. 
Let $s\mapsto\mu_s\in\AC^2([0,1];\Prob_2(X))$ be a curve joining $\mu_0\in\Dom(\Ent_\m)$ and $\mu_1\in\Prob_2(X)$ (with also $\mu_1\in\Dom(\Ent_\m)$ in the particular cas $t_1=t_0=0$).
We can find strongly regular curves $s\mapsto\mu^n_s\in\AC^2([0,1];\Prob_2(X))$, $n\in\N$, as in \cref{def:strong_reg_curve}, approximating the curve $s\mapsto\mu_s$ as stated in~\cref{res:approx_strong_reg_curves}.
By \cref{res:action_ent_estimate_reg_curves} applied to each $s\mapsto\mu_s^n$ with $\eta(s)=(1-s)t_0+st_1$ for all $s\in[0,1]$, we get
\begin{equation}\label{eq:lacrime}
\begin{split}
\frac{1}{2}\,
W_2^2(\H_{t_1}\mu_1^n
&
,\H_{t_0}\mu_0^n)
+
(t_1-t_0)\,
E_\eps(\H_{t_1}\mu_1^n)
\\
&\le
(t_1-t_0)\,
E_\eps(\H_{t_0}\mu_0^n)
+
\frac{1}{2\I_{-2,\eta}(1)}
\int_0^1|\dot{\mu}_s^n|^2\di s
\end{split}
\end{equation}
for all $n\in\N$ and $\eps>0$.
On the one hand, recalling~\eqref{eq:elio}, we have 
$$
E_\eps(\H_{t_1}\mu_1^n)\ge
\Ent_\m(\H_{t_1}\mu_1^n)
$$ 
for all $\eps>0$ and $n\in\N$.
On the other hand, by~\eqref{eq:masterone} we have that 
\begin{equation*}
\lim_{\eps\to0^+}
E_\eps(\H_{t_0}\mu_0^n)
=
\Ent_\m(\H_{t_0}\mu_0^n)
\end{equation*}
for all $n\in\N$, since $\mu_0^n=f_0^n\m$ with $f_0^n\in\Leb^2(X,\m)$.
Thus we can pass to the limit as~$\eps\to0^+$ in~\eqref{eq:lacrime} and get
\begin{equation}\label{eq:sangue}
\begin{split}
\frac{1}{2}\,
W_2^2(\H_{t_1}\mu_1^n
&
,\H_{t_0}\mu_0^n)
+
(t_1-t_0)\,
\Ent_\m(\H_{t_1}\mu_1^n)
\\
&\le
(t_1-t_0)\,
\Ent_\m(\H_{t_0}\mu_0^n)
+
\frac{1}{2\I_{-2,\eta}(1)}
\int_0^1|\dot{\mu}_s^n|^2\di s
\end{split}
\end{equation}
for all $n\in\N$.
By~\eqref{eq:dual_heat_weak_kuwada} in \cref{res:kuwada_equivalence} and \cref{res:approx_strong_reg_curves}\ref{item:approx_strong_reg_curves_W_2}, we have
$\H_{t_i}\mu_i^n\overset{W_2}{\longto}\H_{t_i}\mu_i$ as $n\to+\infty$ for $i=0,1$, so that
\begin{equation}\label{eq:z15}
\lim_{n\to+\infty}
W_2(\H_{t_1}\mu_1^n,\H_{t_0}\mu_0^n)
= 
W_2(\H_{t_1}\mu_1,\H_{t_0}\mu_0).	
\end{equation}
Also, by the lower semicontinuity of the entropy, we have 
\begin{equation}\label{eq:z16}
\Ent_\m(\H_{t_1}\mu_1)
\le
\liminf_{n\to+\infty}
\Ent_\m(\H_{t_1}\mu_1^n).
\end{equation}
Finally, by \cref{res:approx_strong_reg_curves}\ref{item:approx_reg_curves_Ent}, we can estimate
\begin{equation}\label{eq:z17}
\limsup_{n\to+\infty}
\Ent_\m(\H_{t_0}\mu_0^n)
\le
\Ent_\m(\H_{t_0}\mu_0).
\end{equation}
By~\eqref{eq:z15}, \eqref{eq:z16} and~\eqref{eq:z17}, we can thus pass to the limit as~$n\to+\infty$ in~\eqref{eq:sangue} getting
\begin{equation}\label{eq:ansia}
\begin{split}
\frac{1}{2}\,W_2^2(\H_{t_1}\mu_1
&
,\H_{t_0}\mu_0)
+(t_1-t_0)\,\Ent_\m(\H_{t_1}\mu_1)\\
&\le
(t_1-t_0)\,\Ent_\m(\H_{t_0}\mu_0)
+
\frac{1}{2\RI(t_0,t_1)}\int_0^1|\dot{\mu}_s|^2\di s,
\end{split}
\end{equation}
so that~\ref{item:th_equiv_wEVI} follows by minimizing~\eqref{eq:ansia} with respect to the curves $\mu\in\AC^2([0,1];\Prob_2(X))$ joining $\mu_0$ and $\mu_1$.

\medskip

\textit{Proof of \ref{item:th_equiv_wEVI}$\,\Rightarrow\,$\ref{item:th_equiv_wBE}}. 
Choosing $t_0=t_1=t>0$ in~\eqref{eq:th_equiv_wEVI}, we get
\begin{equation*}
W_2(\H_t\mu_1,\H_t\mu_0)
\le 
\c(t)\,W_2(\mu_1,\mu_0)
\end{equation*} 
for all $\mu_0,\mu_1\in\Dom(\Ent_\m)$.
This proves the validity of~\eqref{eq:kuwada_equiv_2} in \cref{res:kuwada_equivalence}\ref{item:kuwada_equiv_2} for 
$\mathscr D=\Dom(\Ent_\m)$. 
Since $\Dom(\Ent_\m)$ is a $W_2$-dense subset of $\Prob_2^{\rm ac}(X)$, \ref{item:th_equiv_wBE} immediately follows by \cref{res:kuwada_equivalence}. 

\medskip

\textit{Proof of \ref{item:th_equiv_wBE}$\,\Rightarrow\,$\ref{item:th_equiv_wConvex}}.
Since we already know that \ref{item:th_equiv_wBE}$\,\Leftrightarrow\,$\ref{item:th_equiv_wEVI}, by \cref{res:kuwada_equivalence} we have that the dual heat semigroup uniquely extends to a map on $\Prob_2(X)$ satisfying~\eqref{eq:dual_heat_weak_kuwada} and  we can thus argue as in the proof of~\cite{DS08}*{Theorem~3.2}. 
So let $t\ge0$ and $h>0$ be fixed and let $s\mapsto\mu_s\in\mathrm{Geo}([0,1];\Prob_2(X))$ be a $W_2$-geodesic joining $\mu_0,\mu_1\in\Dom(\Ent_\m)$. 
By~\eqref{eq:th_equiv_wEVI} applied respectively to the couple $\mu_0\in\Dom(\Ent_\m)$, $\mu_s\in\Prob_2(X)$, and to the couple $\mu_1\in\Dom(\Ent_\m)$, $\mu_s\in\Prob_2(X)$, both with the choice $t_0=t$ and $t_1=t+h$ for all $s\in[0,1]$ (recall that $\H_t\mu\in\Dom(\Ent_\m)$ for all $\mu\in\Prob_2(X)$ and $t>0$ by \cref{res:LlogL_reg}), we get
\begin{equation}\label{eq:wConvex_est0}
\begin{split}
\frac{1-s}{2}\,W_2^2&(\H_{t+h}\mu_s,\H_t\mu_0)+\frac{s}{2}\,W_2^2(\H_{t+h}\mu_s,\H_t\mu_1)\\
&\qquad
-
\frac{1}{2\RI(t,t+h)}
\left((1-s)\,
W_2^2(\mu_s,\mu_0)
+
s\,W_2^2(\mu_s,\mu_1)\right)\\
&\quad\le
h\Big((1-s)\,\Ent_\m(\H_t\mu_0)+s\,\Ent_\m(\H_t\mu_1)-\Ent_\m(\H_{t+h}\mu_s)\Big)
\end{split}
\end{equation} 
for all $s\in[0,1]$.
Since $s\mapsto\mu_s$ is a $W_2$-geodesic, we can estimate
\begin{equation}\label{eq:wConvex_est1}
(1-s)\,W_2^2(\mu_s,\mu_0)
+
s\,W_2^2(\mu_s,\mu_1)
=
s(1-s)\,W_2^2(\mu_1,\mu_0)
\end{equation}
for all $s\in[0,1]$.
Thanks to the elementary inequality
\begin{equation*}
(1-s)a^2+sb^2\ge s(1-s)(a+b)^2
\quad
\text{for all}\ a,b\in\R,\ s\in[0,1],
\end{equation*}
by the triangular inequality we can also estimate
\begin{equation}\label{eq:wConvex_est2}
\begin{split}
(1-s)\,W_2^2(\H_{t+h}\mu_s,&\H_t\mu_0)+s\,W_2^2(\H_{t+h}\mu_s,\H_t\mu_1)\\
&\ge s(1-s)\Big(W_2(\H_{t+h}\mu_s,\H_t\mu_0)+W_2(\H_{t+h}\mu_s,\H_t\mu_1)\Big)^2\\
&\ge s(1-s)\,W_2^2(\H_t\mu_1,\H_t\mu_0).
\end{split}
\end{equation}
By combining~\eqref{eq:wConvex_est1} and~\eqref{eq:wConvex_est2} with~\eqref{eq:wConvex_est0},
we immediately deduce~\ref{item:th_equiv_wConvex}. 

\medskip

\textit{Proof of \ref{item:th_equiv_wConvex}$\,\Rightarrow\,$\ref{item:th_equiv_wBE}}.
Since $(\H_t)_{t\ge0}$ satisfies~\eqref{eq:dual_heat_weak_kuwada}, \ref{item:th_equiv_wBE} trivially follows by \cref{res:kuwada_equivalence}.  
\end{proof}

\subsection{Application to Carnot groups and the \texorpdfstring{$\SU(2)$}{SU(2)} group}
\label{subsec:applications}

We conclude this section with the application of \cref{res:th_equivalence} to Carnot groups and to the $\SU(2)$ group.

\subsubsection{Carnot groups}
We recall that a Carnot group $\G$  is a connected, simply connected and nilpotent Lie group whose Lie algebra $\mathfrak{g}$ of left-invariant vector fields has dimension~$n\in\N$ and admits a stratification of step~$\kappa\in\N$, 
\begin{equation*}
\mathfrak{g}=V_1\oplus V_2\oplus\cdots\oplus V_\kappa
\end{equation*}
with
\begin{equation*}
V_i=[V_1,V_{i-1}]\quad \text{for } i=1,\dots,\kappa, \qquad [V_1,V_\kappa]=\set{0}. 
\end{equation*}
We set $m_i=\dim(V_i)$ and $h_i=m_1+\dots+m_i$ for $i=1,\dots,\kappa$, with $h_0=0$ and $h_\kappa=n$. 
We fix an adapted basis of $\mathfrak{g}$, i.e.\ a basis $X_1,\dots,X_n$ such that
\begin{equation*}
X_{h_{i-1}+1},\dots,X_{h_i}\ \text{is a basis of}\ V_i,\qquad i=1,\dots,\kappa.
\end{equation*}
Using the exponential coordinates, it is possible to identify~$\G$ with $\R^n$ endowed with the group law determined by the Campbell--Hausdorff formula (in particular, the identity $\neu\in\G$ corresponds to $0\in\R^n$ and $x^{-1}=-x$ for $x\in\G$), and it is not restrictive to assume that $X_i(0)=\mathrm{e}_i$ for any $i=1,\dots,n$.
In particular, by left-invariance, for any $x\in\G$ we get
\begin{equation*}
X_i(x)=dL_x\mathrm{e}_i, \qquad i=1,\dots,n,
\end{equation*}
where $L_x\colon\G\to\G$ is the left-translation by $x\in\G$.
We endow $\mathfrak{g}$ with the left-invariant Riemannian metric $\scalar*{\cdot,\cdot}_\G$ that makes the basis $X_1,\dots,X_n$ orthonormal.
We let $H\G\subset T\G$ be the \emph{horizontal tangent bundle} of the group~$\G$, i.e.\ the left-invariant sub-bundle of the tangent bundle~$T\G$ such that $H_0\G=\set*{X(0) : X\in V_1}$, and we let
\begin{equation*}
\nabla_\G f
=
\sum_{j=1}^{m_1} (X_j f) \, X_j\in V_1
\end{equation*}
be the \emph{horizontal gradient} of~$f$.

For any $i=1,\dots,n$, we define the \emph{degree} $d(i)\in\set*{1,\dots,\kappa}$ of the basis vector field $X_i$ as $d(i)=j$ if and only if $X_i\in V_j$. 
Using this notation, the one-parameter family of\emph{ group dilations} $(\delta_\lambda)_{\lambda\ge0}\colon\G\to\G$ is given by
\begin{equation}\label{eq:def_dilation}
\delta_\lambda(x)
=
\delta_\lambda(x_1,\dots,x_n)
=
(\lambda x_1,\dots,\lambda^{d(i)} x_i,\dots,\lambda^\kappa x_n), 
\quad
\text{for all}\ x\in\G.
\end{equation}

The bi-invariant Haar measure of the group~$\G$ coincides (up to a multiplicative constant) with the $n$-dimensional Lebesgue measure~$\leb^n$ and has the homogeneity property $\leb^n(\delta_\lambda(E))=\lambda^Q\leb^n(E)$, where the integer $Q=\sum_{i=1}^\kappa i\dim(V_i)$ is called the \emph{homogeneous dimension} of the group. 

We endow the group~$\G$ with the canonical \emph{Carnot--Carathéodory metric structure} induced by~$H\G$. 
More precisely, the \emph{Carnot--Carathéodory distance} between $x,y\in\G$ is then defined as
\begin{equation}\label{eq:def_cc_distance_Carnot}
\dcc(x,y)=\inf\set*{\int_0^1\|\dot{\gamma}(t)\|_\G\ dt : \gamma\ \text{is horizontal},\ \gamma(0)=x,\ \gamma(1)=y}.
\end{equation}
Here and in the following, we say that a Lipschitz curve $\gamma\colon[0,1]\to\G$ is a \emph{horizontal curve} if $\dot{\gamma}(t)\in H_{\gamma(t)}\G$ for a.e.\ $t\in[0,1]$.
By Chow--Rashevskii's Theorem, the function $\dcc$ is in fact a distance, which is also left-invariant and homogeneous with respect to the dilations defined in~\eqref{eq:def_dilation}, in the sense that
\begin{equation*}
\dcc(zx,zy)
=
\dcc(x,y),
\quad 
\dcc(\delta_\lambda(x),\delta_\lambda(y))
=
\lambda\,\dcc(x,y),
\end{equation*}
for all $x,y,z\in\G$ and $\lambda\ge0$. 
The resulting metric space $(\G,\dcc)$ is a complete, separable, locally compact and geodesic space. 
Note that 
\begin{equation*}
\leb^n(B_r(x))=c_n r^Q
\quad
\text{for all $x\in\G$ and $r\ge0$},
\end{equation*}
where $c_n=\leb^n(B_1(0))$. 

The standard \emph{sub-Laplacian operator} is  $\Delta_\G=\sum_{i=1}^{m_1} X_i^2$.
Since the horizontal vector fields $X_1,\dots,X_{h_1}$ satisfy the \emph{H\"ormander condition}, by H\"ormander Theorem the \emph{sub-elliptic heat operator} $\de_t-\Delta_\G$ is \emph{hypoelliptic}, meaning that its fundamental solution, the heat kernel $\p\colon(0,+\infty)\times\G\to(0,+\infty)$, is a  smooth function. 
For the main properties of the heat kernel, we refer to~\cite{AS19}*{Theorem~2.3} and to the references therein. 
Here we only recall that, given a function $f\in \Leb^1(\G,\L^n)$, the function 
\begin{equation*}
\P_t f(x)
=
f_t(x)
=
(f\star\p_t)(x)
=
\int_\G f(y)\,\p_t(y^{-1} x)\di y,
\quad
(t,x)\in(0,+\infty)\times\G,
\end{equation*}
is smooth and is a solution of the heat diffusion problem
\begin{equation*}
\begin{cases}
\de_t f_t=\Delta_\G f_t &\text{in}\ (0,+\infty)\times\G,\\[3mm]
f_0=f, &\text{on}\ \{0\}\times\G,
\end{cases}
\end{equation*}
where the initial datum is assumed in the $\Leb^1$-sense, i.e.\ $\lim\limits_{t\to0}f_t=f$ in $\Leb^1(\G,\L^n)$.
Accordingly, we can define
\begin{equation*}
\P_t\mu(x)
=
(\mu\star\p_t)(x)
=
\int_\G\p_t(y^{-1}x)\di\mu(y),
\quad
(t,x)\in(0,+\infty)\times\G,
\end{equation*}
whenever $\mu\in\Prob(X)$, so that we can identify $\H_t=\P_t$ for all $t\ge0$.

It is not difficult to recognize that the space $\Sob^{1,2}(\G,\dcc,\L^n)$ induced by the metric-measure structure $(\G,\dcc,\L^n)$ actually coincides with the well-known \emph{horizontal Sobolev space}
\begin{equation*}
\Sob^{1,2}_\G(\G)
=
\set*{f\in\Leb^2(\G,\L^n) :  X_i f\in\Leb^2(\G,\L^n),\ i=1,\dots,m_1},
\end{equation*}
where $X_i f$ stands for the derivative of the function~$f$ in the direction $X_i$ defined in the usual weak sense via integration by parts against test functions.
In particular, the Sobolev space $\Sob^{1,2}(\G,\dcc,\L^n)$ is Hilbertian and the Cheeger energy coincides with the \emph{horizontal Dirichlet energy}
\begin{equation*}
\Ch(f)
=
\int_\G \|\nabla_\G f(x)\|_\G^2\di x
\quad
\text{for all}\
f\in\Sob^{1,2}_\G(\G),
\end{equation*} 
so that $|\D f|_w=\|\nabla_\G f\|_\G$.
We refer the reader to~\cite{HS20}*{Theorem~1.3}, \cite{LeDLP19}*{Theorem~1.2} and~\cite{LP20}*{Theorem~6.3} for more general results in this direction (note that strictly related observations are made in~\cite{AGM15}*{Section~3.2} for the $BV$ space in the sub-Riemannian context).

By a standard regularization argument via group convolution, we thus immediately deduce that if $f\in\Sob^{1,2}_\G(\G)$ with $\|\nabla_\G f\|_\G\le L$, then $f$ agrees $\L^n$-a.e.\ with a $\dcc$-Lipschitz function with Lipschitz constant not larger than~$L$.

In conclusion, $(\G,\dcc,\L^n)$ is an admissible metric-measure space in the sense of \cref{subsec:main_ass} which is also a heat-smoothing admissible group as in Definitions~\ref{def:heat-smoothing_space} and~\ref{def:admissible_group}. 
Therefore, combining~\cite{M08}*{Theorem~1.8} with \cref{res:kuwada_equivalence} and \cref{res:th_equivalence}, we get the following result. 

\begin{theorem}[Equivalence in Carnot groups]
\label{res:th_equiv_Carnot}
Let $(\G,\dcc,\L^n)$ be a Carnot group.
There exists an optimal constant $C_\G\ge1$ (depending only on the group structure and such that $C_\G=1$ if and only if~$\G$ is commutative) satisfying the following four equivalent properties.

\begin{enumerate}[label=(\roman*)]

\item {\normalfont [$\BE_w(C_\G,\infty)$]}
If 
$f\in\Cont^\infty(\G)$, 
then
$\Gamma^\G(\P_t f)\le C_\G^2\,\P_t\Gamma^\G(f)$
for all $t\ge0$.

\item {\normalfont [Kuwada]}
If 
$\mu,\nu\in\Prob_2(\G)$,
then
$W_2(\P_t\mu,\P_t\nu)\le C_\G\,W_2(\mu,\nu)$
for all $t\ge0$.

\item {\normalfont [$\EVI_w(C_\G)$]}
If 
$\mu_0,\mu_1\in\Dom(\Ent_{\L^n})$, 
then
\begin{equation*}
W_2^2(\P_{t_1}\mu_1,\P_{t_0}\mu_0)
-C_\G^2\,W_2^2(\mu_1,\mu_0)
\le
2\,(t_1-t_0)\,
\big(\Ent_{\L^n}(\P_{t_0}\mu_0)-\Ent_{\L^n}(\P_{t_1}\mu_1)\big)	
\end{equation*}
for all $0\le t_0\le t_1$.

\item\label{item:th_equiv_Carnot_RCDw}
{\normalfont [$\RCD_w(C_\G,\infty)$]}
If 
$s\mapsto\mu_s\in\mathrm{Geo}([0,1];\Prob_2(\G))$ connects 
$\mu_0,\mu_1\in\Dom(\Ent_{\L^n})$, 
then
\begin{align*}
\Ent_{\leb^n}(\P_{t+h}\mu_s)
&\le 
(1-s)\,\Ent_{\leb^n}(\P_t\mu_0)
+s\,\Ent_{\leb^n}(\P_t\mu_1)\\
&\quad+\frac{s(1-s)}{2h}
\left(C_\G^2\,
W_2^2(\mu_0,\mu_1)-W_2^2(\P_t\mu_0,\P_t\mu_1)\right)
\end{align*} 
for all $s\in[0,1]$, $t\ge0$ and $h>0$. 
\end{enumerate}	
\end{theorem}

Thanks to the entropy dissipation along the heat flow proved in~\cite{AS19}*{Proposition~4.2} and to the integrability property of the Fisher information along the heat flow given by \cref{res:Fisher_Entropy_along_H_t}\ref{item:Fisher_moment_est_dual_heat}, from \cref{res:th_equiv_Carnot}\ref{item:th_equiv_Carnot_RCDw} we deduce the following weak convexity property of the entropy along $W_2$-geodesics in Carnot groups.
Here and in the following, we let 
\begin{equation*}
\F_\G(f)=\int_{\G\cap\set*{f>0}}\frac{\|\nabla_\G f\|_{\G}^2}{f}\di x
\end{equation*}
be the Fisher information in the Carnot group~$\G$.

\begin{corollary}[Weak convexity of $\Ent_{\L^n}$ in Carnot groups]
\label{res:civetta}
Let $(\G,\dcc,\L^n)$ be a Carnot group and let $C_\G\ge1$ be as in \cref{res:th_equiv_Carnot}.
Let 
$s\mapsto\mu_s\in\mathrm{Geo}([0,1];\Prob_2(\G))$ be a $W_2$-geodesic connecting
$\mu_0,\mu_1\in\Dom(\Ent_{\L^n})$.
If $\mu_s\in\Dom(\Ent_{\L^n})$ for some $s\in(0,1)$, then
\begin{equation}\label{eq:civetta}
\begin{split}
\Ent_{\L^n}(\mu_s)
&\le
(1-s)\,\Ent_{\leb^n}(\P_t\mu_0)
+s\,\Ent_{\leb^n}(\P_t\mu_1)\\
&\quad
+
\frac{s(1-s)}{2h}
\left(
C_\G^2\,
W_2^2(\mu_0,\mu_1)
-
W_2^2(\P_t\mu_0,\P_t\mu_1)
\right)
+
\int_0^{t+h}\F_\G(\P_r\mu_s)\di r
\end{split}	
\end{equation}
for all $t\ge0$ and $h>0$.
\end{corollary}

It is interesting to compare inequality~\eqref{eq:civetta} when $t=0$ with the entropic inequality obtained in~\cite{BKS18}*{Corollary~3.4} when $\G=\mathbb H^n$, the $n$-dimensional Heisenberg group.
Note that similar comparisons can be made for several others Carnot groups thanks to the results obtained in~\cites{BKS19,BR19}.

The Heisenberg group $\mathbb H^n$, $n\in\N$, is the non-commutative Carnot group of step~$2$ whose Lie algebra satisfies $\mathfrak g=V_1\oplus V_2$ with $m_1=2n$, $m_2=1$ and
\begin{align*}
X_i=\de_{x_i}-\tfrac{x_{n+i}}{2}\,\de_{x_{2n+1}},
\quad
X_{n+i}=\de_{x_{n+i}}+\tfrac{x_i}{2}\,\de_{x_{2n+1}},
\quad
X_{2n+1}=\de_{x_{2n+1}},
\end{align*}
for all $i=1,\dots,n$.
Thanks to~\cite{BKS18}*{Corollary~3.4}, the unique $W_2$-geodesic 
$s\mapsto\mu_s\in\mathrm{Geo}([0,1];\Prob_2(X))$
joining two compactly supported measures $\mu_0,\mu_1\in\Dom(\Ent_{\L^{2n+1}})$ satisfies 
\begin{equation}\label{eq:Ent_convex_Heis_Balogh}
\Ent_{\L^{2n+1}}(\mu_s)
\le
(1-s)\,\Ent_{\L^{2n+1}}(\mu_0)
+
s\,\Ent_{\L^{2n+1}}(\mu_1)
+
w(s)
\end{equation} 
for all $s\in(0,1)$,
where
\begin{equation}\label{eq:defect_Ent_Balogh}
w(s)
=
-2\log\left((1-s)^{(1-s)}s^s\right)
\end{equation} 
for all $s\in(0,1)$.
Note that the function in~\eqref{eq:defect_Ent_Balogh} is concave and such that 
\begin{equation*}
\lim\limits_{s\to0^+}w(s)=\lim\limits_{s\to1^-}w(s)=0
\end{equation*}
and it satisfies
\begin{equation*}
0<w(s)\le w(\tfrac12)=\log4
\end{equation*}
for all $s\in(0,1)$.
Therefore, as a consequence of~\eqref{eq:Ent_convex_Heis_Balogh}, we get that $\mu_s\in\Dom(\Ent_{\L^{2n+1}})$ and hence, by \cref{res:civetta}, we can also estimate
\begin{equation}\label{eq:Ent_convex_Heis}
\Ent_{\L^{2n+1}}(\mu_s)
\le
(1-s)\,\Ent_{\leb^{2n+1}}(\mu_0)
+s\,\Ent_{\leb^{2n+1}}(\mu_1)
+
\sigma(s)
\end{equation}
where
\begin{equation}\label{eq:defect_Ent_Heis}
\sigma(s)
=
\inf
\set*{
h>0
:
s(1-s)\,
\frac{C_{\mathbb H^n}^2-1}{2h}\,W_2^2(\mu_0,\mu_1)
+
\int_0^h\F_{\mathbb H^n}(\P_r\mu_s)\di r
}
\end{equation}
for all $s\in[0,1]$.

Although we are not able to give a more explicit formula for the function in~\eqref{eq:defect_Ent_Heis}, it appears that, at least in some cases when $\mu_0$ and $\mu_1$ are very close to each other,    inequality~\eqref{eq:Ent_convex_Heis} is more precise than inequality~\eqref{eq:Ent_convex_Heis_Balogh} for intermediate times, that is, $\sigma(s)<w(s)$ for some $s\in(0,1)$.
As a trivial example, if $\mu_0=\mu_1$, then $\mu_s=\mu_0$ for all $s\in(0,1)$ and thus  inequality~\eqref{eq:Ent_convex_Heis_Balogh} reduces to $0\le w(s)$ for all $s\in(0,1)$, while $\sigma(s)=0$ for all $s\in[0,1]$.
As a less trivial example, exploiting the fact that right translations are optimal transport maps in $\mathbb H^n$, see~\cite{AR04}*{Example~5.7} and~\cite{FJ08}*{Section~2.1}, we can prove the following result.
Here we let 
\begin{equation*}
\tilde\F_{\mathbb H^n}(f)=\int_{\mathbb H^n\cap\set*{f>0}}\frac{\|\tilde\nabla_{\mathbb H^n} f\|_{\mathbb H^n}^2}{f}\di x
\end{equation*}
be the Fisher information in the Heisenberg group $\mathbb H^n$ relative to the \emph{right-invariant horizontal gradient} $\tilde\nabla_{\mathbb H^n}$.

\begin{proposition}[An estimate of $\sigma$ for right translations]
Let $\mu_0=f_0\leb^{2n+1}\in\Dom(\Ent_{\L^{2n+1}})$ be such that $f_0\in\Cont_c^1(\R^{2n+1})$ with $\tilde\F_{\mathbb H^n}(f_0)<+\infty$.
Let $u\in\mathbb H^n$ be a horizontal point in $\mathbb H^n$, i.e.\ $u_{2n+1}=0$, and define $T_s(x)=xu_s$ for all $x\in\mathbb H^n$ and $s\in[0,1]$, where $u_s=su$. 
Then
$s\mapsto\mu_s=(T_s)_\sharp\mu_0$
is the unique $W_2$-geodesic 
joining $\mu_0$ with $\mu_1=(T_1)_\sharp\mu_0$ and 
\begin{equation*}
\sigma(s)
\le
\dcc(u,\neu)
\sqrt{2s(1-s)\,(C_{\mathbb H^n}^2-1)\,\tilde\F_{\mathbb H^n}(f)}
\end{equation*}
for all $s\in(0,1)$.
In particular, for any $\eps\in(0,1)$ there exists $\delta>0$ such that
\begin{equation*}
\dcc(u,\neu)\sqrt{\tilde\F_{\mathbb H^n}(f)}<\delta
\implies
\sigma(s)<w(s)\
\text{for all}\
s\in(\eps,1-\eps).
\end{equation*}
\end{proposition}  

\begin{proof}
The fact that $s\mapsto\mu_s=(T_s)_\sharp\mu_0$ is the unique $W_2$-geodesic 
joining $\mu_0$ with $\mu_1=(T_1)_\sharp\mu_0$ can be proved arguing as in~\cite{FJ08}*{Section~2.1} since, up to a rotation fixing the vertical axis $\set*{x\in\R^{2n+1} : x_i=0\ \text{for all}\ i=1,\dots,2n}$, one can also assume $u_i=0$ for all $i=2,\dots,2n$. 
We thus omit the details.
By the optimality of right translations, we have
\begin{align*}
W_2^2(\mu_0,\mu_1)
=
\int_{\mathbb H^n}\dcc^2(x,T_1(x))\di\mu_0(x)
=
\int_{\mathbb H^n}\dcc^2(x,xu)\di\mu_0(x)
=
\dcc^2(u,\neu).
\end{align*} 
Moreover, we can write $\mu_s=f_s\leb^{2n+1}$, where $f_s=f_0\circ T_s\in\Cont_c^1(\R^{2n+1})$, so that
\begin{equation*}
\nabla_{\mathbb H^n}(\P_rf_s)
=
\nabla_{\mathbb H^n}(f_s\star\p_r)
=
(\tilde\nabla_{\mathbb H^n}f_s)\star\p_r
\end{equation*}
for all $r>0$ and $s\in(0,1)$, and thus
\begin{align*}
\F_{\mathbb H^n}(\P_r f_s)
=
\int_{\mathbb H^n}\frac{|((\tilde\nabla_{\mathbb H^n}f_s)\star\p_r)(x)|^2_{\mathbb H^n}}{(f_s\star\p_r)(x)}\di x
\le
\int_{\mathbb H^n}\frac{|\tilde\nabla_{\mathbb H^n}f_s(x)|^2_{\mathbb H^n}}{f_s(x)}\di x
=
\int_{\mathbb H^n}\frac{|\tilde\nabla_{\mathbb H^n}f_0(x)|^2_{\mathbb H^n}}{f_0(x)}\di x
\end{align*}
for all $r>0$ and $s\in(0,1)$ by Jensen inequality, arguing as in the proof of~\cite{AS19}*{Lemma~4.5} (for right convolutions instead of left ones).
Hence, from~\eqref{eq:defect_Ent_Heis}, we get
\begin{align*}
\sigma(s)
&\le
s(1-s)\,
\frac{C_{\mathbb H^n}^2-1}{2h}\,W_2^2(\mu_0,\mu_1)
+
\int_0^h\F_{\mathbb H^n}(\P_r\mu_s)\di r
\\
&\le
s(1-s)\,
\frac{C_{\mathbb H^n}^2-1}{2h}\,\dcc^2(u,\neu)
+
h\,\tilde\F_{\mathbb H^n}(f_0)
\end{align*}
for all $h>0$ and $s\in(0,1)$. The conclusion thus follows by optimizing with respect to~$h>0$ and by recalling~\eqref{eq:defect_Ent_Balogh}. 
\end{proof}

\subsubsection{The \texorpdfstring{$\SU(2)$}{SU(2)} group}
\label{subsec:SU(2)_group}

The group $\SU(2)$ is the Lie group of $2\times 2$ complex unitary matrices with determinant~$1$.
Its Lie algebra $\mathfrak{su}(2)$ consists of all $2\times 2$ complex unitary skew-Hermitian matrices with trace~$0$.
Following the notation of~\cite{BB09}, a basis of $\mathfrak{su}(2)$ is given by the \emph{Pauli matrices}
\begin{align*}
X=
\begin{pmatrix}
0 & 1\\
-1 & 0\\
\end{pmatrix},
\quad
Y=
\begin{pmatrix}
0 & i\\
i & 0\\
\end{pmatrix},
\quad
Z=
\begin{pmatrix}
i & 0\\
0 & -i\\
\end{pmatrix},
\end{align*}
satisfying the relations
\begin{align*}
[X,Y]=2Z,
\quad
[Y,Z]=2X,
\quad
[Z,X]=2Y.
\end{align*}
We keep the same notation $X$, $Y$ and $Z$ for the left invariant vector fields on $\SU(2)$ corresponding to the Pauli matrices.
Similarly as before, we let 
\begin{equation*}
\nabla_{\SU(2)}f
=
(Xf)X+(Yf)Y
\end{equation*}
be the \emph{horizontal gradient} of $f$.
Using the cylindric coordinates
\begin{align*}
(r,\theta,z)
\mapsto
\exp(r\cos\theta\,X+r\sin\theta\,Y)
\,
\exp(\zeta\,Z)
=
\begin{pmatrix}
e^{i\zeta}\cos r & e^{i(\theta-\zeta)}\sin r\\
-e^{-i(\theta-\zeta)}\sin r & e^{-i\zeta}\cos r\\
\end{pmatrix}
\end{align*}
valid for $r\in[0,\frac\pi2)$, $\theta\in[0,2\pi]$ and $\zeta\in[-\pi,\pi]$ (originally introduced in~\cite{CS01}), the normalized bi-invariant Haar measure $\m\in\Prob(\SU(2))$ can be written as
\begin{equation*}
\di\m
=
\frac1{4\pi^2}\,\sin(2r)\di r\di\theta\di\zeta.
\end{equation*}
Once the left-invariant Riemannian metric $\scalar*{\cdot,\cdot}_{\SU(2)}$ making the basis $X$, $Y$, $Z$ orthonormal is introduced, we can endow the group $\SU(2)$ with the Carnot--Carathéodory distance $\dcc$ defined analogously as in~\eqref{eq:def_cc_distance_Carnot}. 
The resulting metric space $(\SU(2),\dcc)$ is a complete, separable, locally compact and geodesic space.

The standard sub-Laplacian operator is $\Delta_{\SU(2)}=X^2+Y^2$ and, again by H\"ormander Theorem, the heat operator $\de_t-\Delta_{\SU(2)}$ has a smooth fundamental solution $\p\colon(0,+\infty)\times\SU(2)\to(0,+\infty)$ which induces the associated heat flow $(\P_t)_{t\ge0}$ by right convolution (so that we can still identify $\H_t=\P_t$ for all $t\ge0$).

Arguing similarly in the case of Carnot groups (recall the previously cited~\cites{HS20,LeDLP19,LP20}), it is possible to identify the space $\Sob^{1,2}(\SU(2),\dcc,\m)$ with the horizontal Sobolev space
\begin{equation*}
\Sob^{1,2}_{\SU(2)}(\SU(2))
=
\set*{f\in\Leb^2(\SU(2),\m) : Xf,Yf\in\Leb^2(\SU(2),\m)}
\end{equation*}
defined using integration by parts against test functions, so that $\Sob^{1,2}(\SU(2),\dcc,\mu)$ is Hilbertian, the Cheeger energy coincides with the horizontal Dirichlet energy
\begin{equation*}
\Ch(f)
=
\int_{\SU(2)}\|\nabla_{\SU(2)}f\|_{\SU(2)}^2\di\m,
\quad
\text{for all}\
f\in\Sob^{1,2}_{\SU(2)}(\SU(2)),
\end{equation*}
and $|\D f|_w=\|\nabla_{\SU(2)}f\|_{\SU(2)}$.
We again refer the reader to~\cite{LeDLP19}*{Theorem~1.2} for a proof of these identifications.

Exploiting the group structure of $\SU(2)$ similarly to the case of Carnot groups, we get that if $f\in\Sob^{1,2}_{\SU(2)}(\SU(2))$ with $\|\nabla_{\SU(2)} f\|_{\SU(2)}\le L$, then $f$ agrees $\mu$-a.e.\ with a $\dcc$-Lipschitz function with Lipschitz constant not larger than~$L$.

Therefore, $(\SU(2),\dcc,\m)$ is a heat-smoothing admissible group and, combining~\cite{BB09}*{Theorem~4.10} with \cref{res:kuwada_equivalence} and \cref{res:th_equivalence}, we get the following result. 

\begin{theorem}[Equivalence in $\SU(2)$]
\label{res:th_equiv_SU(2)}
Let $(\SU(2),\dcc,\m)$ be as above.
There exists a constant $C_{\SU(2)}\ge\sqrt{2}$ satisfying the following four equivalent properties.

\begin{enumerate}[label=(\roman*)]

\item {\normalfont [$\BE_w(C_{\SU(2)}e^{-2t},\infty)$]}
If 
$f\in\Cont^\infty(\SU(2))$, 
then
$\Gamma^{\SU(2)}(\P_t f)\le C_{\SU(2)}^2e^{-4t}\,\P_t\Gamma^{\SU(2)}(f)$
for all $t\ge0$.

\item {\normalfont [Kuwada]}
If 
$\mu,\nu\in\Prob_2(\SU(2))$,
then
$W_2(\P_t\mu,\P_t\nu)\le C_{\SU(2)}e^{-2t}\,W_2(\mu,\nu)$
for all $t\ge0$.

\item {\normalfont [$\EVI_w(C_{\SU(2)}e^{-2t})$]}
If 
$\mu_0,\mu_1\in\Dom(\Ent_{\m})$, 
then
\begin{equation*}
W_2^2(\P_{t_1}\mu_1,\P_{t_0}\mu_0)
-C_{\SU(2)}^2\,\frac{4(t_1-t_0)}{e^{4t_1}-e^{4t_0}}\,W_2^2(\mu_1,\mu_0)
\le
2\,(t_1-t_0)\,
\big(\Ent_{\m}(\P_{t_0}\mu_0)-\Ent_{\m}(\P_{t_1}\mu_1)\big)	
\end{equation*}
for all $0\le t_0\le t_1$.

\item
{\normalfont [$\RCD_w(C_{\SU(2)}e^{-2t},\infty)$]}
If 
$s\mapsto\mu_s\in\mathrm{Geo}([0,1];\Prob_2(\SU(2)))$ is a geodesic connecting
$\mu_0,\mu_1\in\Dom(\Ent_{\m})$, 
then
\begin{align*}
\Ent_{\m}(\P_{t+h}\mu_s)
&\le 
(1-s)\,\Ent_{\m}(\P_t\mu_0)
+s\,\Ent_{\m}(\P_t\mu_1)\\
&\quad+\frac{s(1-s)}{2h}
\left(\frac{4C_{\SU(2)}^2h}{e^{4(t+h)}-e^{4t}}\,
W_2^2(\mu_0,\mu_1)-W_2^2(\P_t\mu_0,\P_t\mu_1)\right)
\end{align*} 
for all $s\in[0,1]$, $t\ge0$ and $h>0$. 
\end{enumerate}	
\end{theorem}

Since $(\SU(2),\dcc,\m)$ is a \emph{Sasakian manifold}, the resulting sub-Riemannian structure on~$\SU(2)$ is \emph{ideal}, see~\cite{BR20}*{Definition~13 and Section~7.4} for the precise definitions.
Thus, according to~\cite{BR20}*{Theorem~39}, if $\mu_0,\mu_1\in\Dom(\Ent_{\m})$ are two compactly supported probability measures, then the unique Wassestein geodesic $s\mapsto\mu_s\in\mathrm{Geo}([0,1];\Prob_2(\SU(2)))$ joining them satisfies $\mu_s\ll\m$ for all $s\in[0,1]$. 
Thanks to~\cite{BR20}*{Theorem~9 and Corollary~67} (see also~\cites{AL14,LLZ16}), it actually holds that $\mu_s\in\Dom(\Ent_\m)$ for all $s\in(0,1)$ and the function $s\mapsto\Ent_\m(\mu_s)$ satisfies an inequality similar to~\eqref{eq:Ent_convex_Heis_Balogh}.

Up to our knowledge, there is no analogue of the entropy dissipation for $\Leb^1$-densities proved in~\cite{AS19}*{Proposition~4.2} for the $\SU(2)$ group and we can only rely on the general result for $\Leb^1\cap\Leb^2$-densities obtained in~\cite{AGS14}*{Proposition~4.22}. 
Thus, at the present moment, an inequality for the function $s\mapsto\Ent_\m(\mu_s)$ similar to~\eqref{eq:civetta} holds in the $\SU(2)$ group under the additional assumption that $\frac{\di\mu_s}{\di\m}\in\Leb^2(X,\m)$ for some  $s\in(0,1)$.
Also, up to our knowledge, it is not known if right translations in the $\SU(2)$ group are optimal transport maps.
For this reason, a comparison of the entropic inequalities in the $\SU(2)$ group analogous to the one done above for Carnot groups is not easily reachable at the present moment.
We will hopefully come back to this topic in a future work.


\end{document}